\documentclass[a4paper,10pt]{amsart}

\usepackage[utf8]{inputenc}
\usepackage{amssymb}
\usepackage{amsmath}
\usepackage{enumerate}

\newtheorem{theorem}{Theorem}[section]

\newtheorem{lemma}[theorem]{Lemma}
\newtheorem{proposition}[theorem]{Proposition}
\newtheorem{corollary}[theorem]{Corollary}
\newtheorem{definition}{Definition}[section]
\theoremstyle{remark}
\newtheorem{remark}[theorem]{Remark}
\theoremstyle{definition}

\numberwithin{equation}{section}

\newcommand{\R}{\ensuremath{\mathbb{R}}}
\newcommand{\N}{\ensuremath{\mathbb{N}}}
\newcommand{\E}{\ensuremath{\mathcal{E}}}
\newcommand{\Levy}{\ensuremath{\mathcal{L}}}
\newcommand{\dif}{\mathrm{d}}
\newcommand{\ra}{\rightarrow}
\newcommand{\ua}{\uparrow}
\newcommand{\da}{\downarrow}
\newcommand{\mt}{\mapsto}
\newcommand{\alp}{\alpha}

\newcommand{\Div}{\mbox{div}}
\DeclareMathOperator*{\esssup}{ess \, sup}
\DeclareMathOperator*{\essinf}{ess \, inf}
\DeclareMathOperator{\sgn}{sgn}
\DeclareMathOperator{\Ent}{E}
\DeclareMathOperator{\Lip}{Lip}
\DeclareMathOperator{\supp}{supp}
\DeclareMathOperator*{\cO}{\mathit{O}}
\DeclareMathOperator*{\co}{\mathit{o}}
\DeclareMathOperator*{\sinc}{sinc}
 
\begin{document}

\title[Continuous dependence estimates for fractional
PDEs]{Optimal continuous dependence estimates for fractional degenerate parabolic equations}

\author[N.~Alibaud]{Natha\"{e}l Alibaud}
\address[Natha\"{e}l Alibaud]{UMR CNRS 6623, Universit\'e de Franche-Comt\'e\\
16 route de Gray\\ 25 030 Besan\c{c}on cedex, France\\ 
and\\
\'Ecole Nationale Sup\'erieure de M\'ecanique et des Microtechniques, 26 chemin de l'\'Epitaphe, 25030 Besan\c{c}on cedex, France\\
and\\
Department of Mathematics, Faculty of Science\\
Prince of Songkla University\\
Hat Yai, Songkhla, Thailand, 90112}
\email{nathael.alibaud\@@{}ens2m.fr}

\author[S.~Cifani]{Simone Cifani}
\address[Simone Cifani]{Department of Mathematics\\
Norwegian University of Science and Technology (NTNU)\\
N-7491 Trondheim, Norway} \email[]{simone.cifani\@@{}math.ntnu.no}

\author[E.~R.~Jakobsen]{Espen R. Jakobsen}
\address[Espen R. Jakobsen]{Department of Mathematics\\
Norwegian University of Science and Technology (NTNU)\\
N-7491 Trondheim, Norway} \email[]{erj\@@{}math.ntnu.no}
\urladdr{http://www.math.ntnu.no/\~{}erj/}



\thanks{This research was supported by the Research Council of Norway
  (NFR) through the project``Integro-PDEs: Numerical methods,
  Analysis, and Applications to Finance,'' and by the ``French ANR
  project CoToCoLa, no. ANR-11-JS01-006-01.''} 

\begin{abstract}
We derive continuous dependence 
estimates for weak entropy
solutions of degenerate parabolic equations with nonlinear
fractional diffusion. 
The diffusion term involves the fractional
Laplace operator, $\Delta^{\alp/2}$ for $\alp \in (0,2)$.
Our results are quantitative and we exhibit an example 
for which they are optimal.
We cover the dependence on
the nonlinearities, and for the first time, the
 Lipschitz dependence on $\alpha$ in the $BV$-framework.
The former estimate (dependence on nonlinearity) 
is robust in the sense that it 
is stable 
in the limits~$\alp \da 0$ and~$\alp \ua 2$. In the limit $\alp\ua2$, $\Delta^{\alp/2}$ converges to the usual
Laplacian, and we show rigorously that we recover the 
optimal
continuous dependence result of \cite{CoGr99} for local
degenerate parabolic equations (thus providing an alternative proof). 
\end{abstract}

\maketitle

\tableofcontents

\section{Introduction}
In this paper we consider the following Cauchy problem: 
\begin{equation}\label{1}
\begin{cases}
\partial_t u+\Div f(u) +(-\triangle)^{\frac{\alpha}{2}} \varphi(u)=0 &\text{in} \quad  Q_T:=\mathbb{R}^d\times(0,T),\\
u(x,0)=u_{0}(x), &\text{in} \quad \R^d,
\end{cases}
\end{equation}
where~$T>0$ is fixed, $u=u(x,t)$ is the unknown function, $\Div$ and $\triangle$
denote divergence and Laplacian with 
respect to~$x$, and~$(-\triangle)^{\frac{\alpha}{2}}$,
$\alpha \in (0,2)$, is the fractional Laplacian e.g. defined
as 
\begin{equation}\label{referee-fourier-def}
(-\triangle)^{\frac{\alpha}{2}}
\phi:=\mathcal{F}^{-1} \left( |2 \, \pi \cdot|^{\alpha} \mathcal{F} \phi
\right)
\end{equation}
with the Fourier transform $\mathcal{F} \phi(\xi):=\int_{\R^d} e^{-2 \, i \, \pi \, x \cdot \xi} \, \phi(x) \, \dif x$. Notice that \eqref{referee-fourier-def} is compatible with the formula $-\triangle \phi=\mathcal{F}^{-1} \left( |2 \, \pi \cdot|^{2} \mathcal{F} \phi
\right)$. 
Throughout the paper we assume that
\begin{align}
&u_0 \in L^\infty \cap L^1 \cap BV(\R^d),\label{u0}\\
&f \in \left(W_{\scriptstyle \text{loc}}^{1,\infty}(\R)\right)^d
\text{ with } f(0)=0,\label{flux}\\
& \varphi \in W_{\scriptstyle \text{loc}}^{1,\infty}(\R) \text{ is nondecreasing with } \varphi(0)=0.\label{Aflux}
\end{align}

\begin{remark}
Subtracting constants from $f$ and~$\varphi$ if necessary, there is no loss of generality in assuming that~$f(0)=0$ and~$\varphi(0)=0$.
\end{remark}

The fractional Laplacian is the generator of the symmetric
$\alp$-stable process, the most famous pure jump L\'evy process. There
is a large literature on L\'evy 
processes, we refer to e.g. \cite{Sat99} for more details, and they
are important in many modern 
applications. Being very selective, we mention radiation
hydrodynamics~\cite{Ros89,ScTa92,RoYo07}, anomalous diffusion in
semiconductor growth~\cite{Woy01}, over-driven gas
detonations~\cite{Cla02}, mathematical finance~\cite{CoTa04}, and  flow
in porous media~\cite{DeQuRoVa11,DeQuRoVa12}.

Due to the second part of assumption \eqref{Aflux},
the term $(-\triangle)^{\frac{\alpha}{2}} \varphi(u)$ is a nonlinear
and nonlocal diffusion term. It formally converges
toward $\varphi(u)$ and $-\triangle \varphi(u)$ as $\alpha \da 0$
and $\alpha \ua 2$ respectively. Hence, Equation~\eqref{1} could be
seen as a nonlocal ``interpolation'' between the
hyperbolic equation
\begin{equation}\label{nat:scl}
\partial_t u+\Div f(u)+\varphi(u)=0,
\end{equation}
and the degenerate parabolic equation
\begin{equation}\label{nat:local-eq}
\partial_t u+\Div f(u)-\triangle  \varphi(u)=0.
\end{equation}
Equation \eqref{1} is said to be supercritical if $\alpha<1$,
subcritical if $\alpha>1$, and critical if $\alpha=1$. The diffusion function $\varphi$ is said to be
strongly 
degenerate 
if
$\varphi'$ 
vanishes 
on a nontrivial interval.
Equation~\eqref{1} 
can therefore be of mixed hyperbolic parabolic type depending on
the choice of $\alpha$ and $\varphi$. Note that
in the mathematical community,
interest in nonlinear nonlocal diffusions is 
in fact very recent, and only few results exist; cf. 
e.g. \cite{BiKaIm10,BiKaMo10,CaVa10,CiJa11,DeQuRoVa11,AlCiJa12,DeQuRoVa12,DeQuRoVapr}
and the references therein.

Let us give the main references for the well-posedness of the Cauchy problems for  \eqref{nat:scl} and \eqref{nat:local-eq}. For a more complete bibliography, see the books \cite{Dib93,Daf05,Vaz07} and the references in \cite{KaRi03}. In the hyperbolic case where $\varphi' \equiv 0$, we get the scalar conservation law $\partial_t u+\Div f(u)=0$. The solutions of this equation could develop discontinuities in finite time and the weak solutions of the Cauchy problem are generally not unique.
The most famous uniqueness result relies on the notion of entropy solutions introduced in \cite{Kru70}. 
In the pure diffusive case where $f' \equiv 0$, there is no more creation of shock and the initial-value problem for $\partial_t u-\triangle \varphi(u)=0$ admits a unique weak solution, cf. \cite{BrCr79}.
Much later, the adequate notion of entropy solutions for
mixed hyperbolic parabolic equations
was introduced in~\cite{Car99}.
This paper
focuses on 
an initial-boundary value problem.
For a general well-posedness result applying to both
\eqref{nat:scl} and \eqref{nat:local-eq}, see e.g.~\cite{KaRi03}.

At the same time, there has been a large interest in nonlocal versions of these equations. The first work seems to be \cite{CoGrLo96} on nonlocal time fractional derivatives,  
cf. also \cite{JaWi03}. The study of nonlocal diffusion terms has probably been initiated by \cite{BiFuWo98}. Now,  the well-posedness is quite well-understood in the
nondegenerate linear case where $\varphi(u)=u$. Smooth solutions exist and are unique for subcritical
equations~\cite{BiFuWo98,DrGaVo03}, shocks could
occur~\cite{AlDrVo07,KiNaSh08} and weak solutions could be
nonunique~\cite{AlAn10} for supercritical equations, 
entropy
solutions exist and are always unique~\cite{Ali07,KaUl11}; cf. also \cite{ChCz10,ChCzSi10} for original regularizing effects.
Very recently, the well-posedness theory
of entropy solutions
was extended
in~\cite{CiJa11} to
cover
the
full problem~\eqref{1}, even for strongly degenerate~$\varphi$. See also \cite{DeQuRoVa11,DeQuRoVa12} on fractional porous medium type equations, and \cite{DeQuRoVapr} on a logarithmic diffusion equation.

This paper is devoted to continuous dependence estimates
for~\eqref{1}, i.e. explicit estimates on the difference of two
entropy solutions $u$ and~$v$ in terms of the difference of their respective
data $(\alpha,u_0,f,\varphi)$
and~$(\beta,v_0,g,\psi)$. Let us point out that we investigate quantitative results which should be distinguished from qualitative ones. By qualitative, we mean stability results only stating that if $(\alpha_n,u_0^n,f_n,\varphi_n)
$ converges toward $(\alpha,u_0,f,\varphi)$, then the associated entropy solutions $u_n$ converge toward $u$. For scalar conservation laws, the first quantitative result on the continuous dependence on $f$ appeared in \cite{Daf72} and also in \cite{Luc80} some years later. Roughly speaking, it states that for $BV$ initial data $u_0 = v_0$,
\begin{equation}\label{referee-dafe}
\|u(\cdot,t)-v(\cdot,t)\|_{L^1} = \cO \left(\|f'-g'\|_\infty\right),
\end{equation}
where throughout the $L^\infty$-norm is always taken over
the range of~$u_0$. Next, the optimal error in $\sqrt{\epsilon}$ for the parabolic regularization $\partial_t u^\epsilon+\Div f(u^\epsilon)-\epsilon \, \triangle u^\epsilon=0$ of scalar conservation laws was established in \cite{Kuz76}. In that paper, the author has developed a general method of error estimation based on the Kruzhkov's device of doubling the variables \cite{Kru70}. We use this method in the present paper. As far as degenerate parabolic equations are concerned, the continuous dependence on $\varphi$ 
was first investigated in \cite{BeCr81} for the equation $\partial_t u-\triangle
\varphi(u)=0$. 
Here the
motivation was to obtain qualitative results under very general
assumptions.
Quantitative results were obtained in \cite{BoPe98,CoGr99} for the full equation \eqref{nat:local-eq}. In \cite{BoPe98}, the authors established alternative estimates to \eqref{referee-dafe} involving weaker norms,
as roughly speaking an estimate in $\|f-g\|_\infty^\frac{1}{2}$.
They gave different estimates for the $\varphi$-dependence with $\psi \equiv 0$. An estimate for nontrivial $\psi$ was given in \cite{CoGr99}. Roughly 
speaking, it states that if $u$ has the same data as $v$  except for
$\varphi \neq \psi$, then
\begin{equation}\label{intro-eol}
\|u(\cdot,t)-v(\cdot,t)\|_{L^1} = \cO \left(\|\sqrt{\varphi'}-\sqrt{\psi'}\|_\infty\right).
\end{equation}
Recently, Estimates \eqref{referee-dafe} and \eqref{intro-eol} were extended in \cite{KaRi03,ChKa05,ChKa06} to anisotropic diffusions and $(x,t)$-dependent data; cf. also \cite{AnBeKaOu09,Lukpr} for recent qualitative results on local equations. 
For nonlocal equations, a number of papers were concerned with convergence rates for vanishing viscosity methods \cite{ScTa92,CoGrLo96,Dro03,DrIm06,Ali07}.
To the best of our knowledge, the first estimate on the ``general continuous dependence 
on the data''
was given in \cite{KaUl11}. 
It concerns the case of
linear nondegenerate L\'evy diffusions. The main novelty was the explicit dependence in the L\'evy measure, corresponding to the explicit dependence in $\alpha$ for the particular case of the fractional Laplacian.
In~\cite{AlCiJa12}, 
the authors of the present paper established a
continuous dependence estimates for
general
nonlinear degenerate L\'evy diffusions.
For a qualitative result in the spirit of \cite{BeCr81}, see the very
recent work \cite{DeQuRoVa12} on the fractional porous medium equation $\partial_t u+(-\triangle)^{\alpha/2} (|u|^{m-1} \, u)=0$, $m>0$. In that paper, the continuous dependence on $(\alpha,m,u_0)$ is established under more general assumptions.

Before explaining our main contributions, let us refer the reader to
more or less related work. The theory of continuous dependence
estimates for nonlocal equations was probably initiated in the
context of viscosity solutions of fully nonlinear integro-PDEs, cf. \cite{JaKa05} and the references therein. See also \cite{Imb05,DrIm06} for error estimates for vanishing viscosity methods.
The question of 
$\alpha$-continuity has been raised earlier, e.g. when
looking for a priori estimates that are robust or
uniform as $\alpha \ua 2$.  Such results can be found in e.g. \cite{CaSi09,CaSi11}, see also  \cite{KaScpr} and the references therein.

The starting point of the present paper is the general
theory of~\cite{AlCiJa12}. 
It is worth mentioning that different estimates could be difficult to compare, as e.g \eqref{referee-dafe} with the estimate in $\|f-g\|_\infty^\frac{1}{2}$ of \cite{BoPe98}. Hence, a remarkable feature is that the estimates in \eqref{referee-dafe} and \eqref{intro-eol} are optimal for linear equations,
cf. the discussion of Section \ref{sec:optim}. A natural question is whether the estimates of \cite{AlCiJa12} applied to \eqref{1} possess such a property. The answer is positive
only
in the supercritical case $\alp<1$. In
this paper,
we obtain optimal
estimates
for all cases. To do so we
restart the proofs from the beginning, by taking into account the
homogeneity properties of the fractional Laplacian. The main
ingredients are 
a new linearization argument a la Young measure theory/kinetic
formulations, and for the linear case, a clever change of the (jump)
$z$-variable in \eqref{nat:Levy-form}. This change of variable allows
us to adapt ideas from viscosity solution theory 
developed in e.g. \cite{JaKa05}.
Let us also refer the reader to
\cite{Say91} for other applications of this change of variable in the
context of viscosity solutions. Roughly speaking, we prove that
\begin{equation}
\label{nonloc}
\|u(\cdot,t)-v(\cdot,t)\|_{L^1} =
\begin{cases}
\cO \left( \|(\varphi')^{\frac{1}{\alpha}}-(\psi')^{\frac{1}{\alpha}}\|_\infty \right), & \mbox{$\alpha>1$,}\\
\cO \left( \|\varphi' \, \ln \varphi'-\psi' \, \ln \psi'\|_\infty \right), & \mbox{$\alpha=1$,}\\
\cO \left( \|\varphi'-\psi'\|_\infty\right), & \mbox{$\alpha<1$},
\end{cases}
\end{equation}
with uniform constants in the limits $\alpha \da 0$ and $\alpha \ua 2$.
Note well that just as in~\cite{AlCiJa12}, our
proofs work directly with the entropy solutions without needing tools like entropy defect
measures, etc.. And even though these tools play a key role in the local
second-order theory, the arguments here really seem to be less
technical relying only on basic convex
inequalities and integral calculus.  Hence, it
seems interesting to mention that 
we recover the 
result~\eqref{intro-eol} rigorously from \eqref{nonloc} by passing to
the limit.
Another remarkable feature is that a simple rescaling transforms the Kuznetsov type  estimate~\eqref{nonloc} into the following time continuity estimate: 
\begin{equation*}
\|u(\cdot,t)-u(\cdot,s)\|_{L^1} =
\begin{cases}
\cO \left( |t^{\frac{1}{\alpha}}-s^{\frac{1}{\alpha}}| \right), & \mbox{$\alpha>1$,}\\
\cO \left( |t \, \ln t-s \, \ln s| \right), & \mbox{$\alpha=1$,}\\
\cO \left( |t-s|\right), & \mbox{$\alpha<1$}.
\end{cases}
\end{equation*}
This result is optimal and strictly better than earlier results in
\cite{CiJapr}, see 
Remark \ref{remtime}. E.g. for positive times, we get Lipschitz
regularity in time with values in $L^1(\R^d)$. This is a regularizing
effect in time when $\alp\geq1$ and $u$ not more than $BV$ initially.
 
In the second main contribution of this paper, we focus on the continuous dependence  on
$\alpha$. By
stability arguments, it is possible to show that the unique entropy
solution $u=:u^\alpha$ is continuous in $\alpha \in [0,2]$ with
values in $L^1_{\rm loc}$. In this paper, we prove that in the
$BV$-framework, it is in fact locally Lipschitz continuous 
in $\alpha \in (0,2)$ with values in $C([0,T];L^1)$.
To the best of our knowledge, such an $\alpha$-regularity result has
never been obtained before. More precisely, the theory of
\cite{AlCiJa12} implies the result for $\alpha \in (0,1)$ but not
for $\alpha \in [1,2)$. For 
the latter
range of exponents, all the results 
cited above 
are either qualitative or suboptimal. The new ingredient to get the Lipschitz regularity is again a change of
(the jump) 
variable.
It seems interesting to recall that the type of
Equation \eqref{1} could change from parabolic when $\alpha>1$ to
hyperbolic when $\alpha<1$. As a consequence, quite different
behaviors are observed in the $\varphi$- and $t$-continuity when
$\alpha$ crosses $1$, cf. the continuous dependence estimates above. A
natural question is thus whether such kind of phenomena arises in the
$\alpha$-regularity? We prove that the answer is positive by carefully
estimating the best Lipschitz constant of the function $\alpha \mt
u^\alpha$ with respect to the position of $\alpha$ and the other
data. More precisely, for $\lambda \in (0,2)$ we define
$$
\Lip_\alpha(u;\lambda):=\limsup_{\alpha,\beta \ra \lambda} \frac{\|u^\alpha-u^\beta\|_{C([0,T];L^1)}}{|\alpha-\beta|},
$$
and roughly speaking we prove that
\begin{equation*}
\Lip_\alpha(u;\lambda) =
\begin{cases}
\cO \left( M^\frac{1}{\lambda} \, |\ln M|  \right), & \mbox{$\lambda>1$,}\\
\cO \left( M \, \ln^2 M \right), & \mbox{$\lambda=1$,}\\
\cO \left( M \right), & \mbox{$\lambda<1$},
\end{cases}
\end{equation*}
for $M:=T \, \|\varphi'\|_\infty$, and
\begin{equation*}
\Lip_\alpha(u;\lambda) =
\begin{cases}
\cO \left( |u_0|_{BV}  \right), & \mbox{$\lambda>1$,}\\
\cO \left( |u_0|_{BV} \, \ln^2 \frac{\|u_0\|_{L^1}}{|u_0|_{BV}} \right), & \mbox{$\lambda=1$,}\\
\cO \left(\|u_0\|_{L^1}^{1-\lambda} \, 
|u_0|_{BV}^\lambda \left|\ln \frac{\|u_0\|_{L^1}}{|u_0|_{BV}}\right| \right), & \mbox{$\lambda<1$}.
\end{cases}
\end{equation*}
 We also exhibit
an example of an equation for which these estimates are optimal in the
regimes where $M$ is sufficiently small or
$\frac{\|u_0\|_{L^1}}{|u_0|_{BV}}$ is sufficiently large. 

Another natural question is whether $\alpha \mt u^\alp$ is
Lipschitz continuous 
up to the boundaries $\alpha=0$ and
$\alpha=2$. The answer is negative for $\alpha=0$
and remains open for $\alpha=2$. For the reader's convenience, more details and open questions are given at the end of Section \ref{sec:main-results}.

To conclude, note that even if we adapt some ideas from viscosity solution theory,
the definition of relevant generalized solution and the mathematical
arguments are very different from 
the ones in e.g. \cite{JaKa05}. Moreover we obtain optimal results here, and, in an a work in progress, 
we adapt ideas of this paper to obtain new results in the viscosity 
solution setting.

The rest of the paper is organized as follows. In
Section~\ref{nat:sec-preliminaries}, we recall the well-posedness
theory for fractional degenerate parabolic equations. In
Section~\ref{sec:main-results}, we state our main results: continuous
dependence with respect to the nonlinearities and
the order of the fractional Laplacian. In
Section~\ref{sec-remind}, we recall the general continuous dependence
estimates of~\cite{AlCiJa12} along with a general Kuznetsov
type of Lemma. Sections \ref{sec-sup}--\ref{sec:limit} are devoted to the
proofs of our main results. In Section~\ref{sec:optim}, we exhibit an
example of an equation for which we rigorously show that our estimates are optimal. Finally, there is an appendix containing
technical lemmas and computations from the different proofs.

\subsection*{Notation} The symbols $\nabla$ and $\nabla^2$ denote the
$x$-gradient and $x$-Hessian. The symbols $\|\cdot\|$ and~$|\cdot|$
are used for norms and semi-norms, respectively. The symbol $\sim$ is used for asymptotic equality ``up to a constant.''
The symbols $\wedge$ and $\vee$ are used for the minimum and maximum between two reals. For any $a,b \in \R$, we use the shorthand notation $\mbox{co}
\{a,b\}$ to design the interval $(a \wedge b,a \vee b)$. The surface measure of the unit sphere of~$\R^d$ is denoted by~$S_d$. 

\section{Preliminaries}\label{nat:sec-preliminaries}

In this section we recall some basic facts on the fractional Laplacian
and fractional degenerate parabolic equations. We start by a
L\'evy--Khinchine type representation formula. For $\alpha \in (0,2)$ and
all $\phi \in C^\infty_c(\R^d)$,  $x \in \R^d$, and $r>0$,
\begin{equation}\label{nat:Levy-form}
\begin{split}
-(-\triangle)^{\frac{\alpha}{2}} \phi(x) & = G_{d}(\alpha) \int_{|z| < r} \frac{\phi(x+z)-\phi(x)-\nabla \phi(x) \cdot z}{|z|^{d+\alpha}} \, \dif z\\
& \quad +G_{d}(\alpha) \int_{|z| > r} \frac{\phi(x+z)-\phi(x)}{|z|^{d+\alpha}} \, \dif z \\
& =: \Levy^\alpha_r[\phi](x)+\Levy^{\alpha,r}[\phi](x),
\end{split}
\end{equation}
where
\begin{equation*}
G_{d}(\alpha):=\frac{2^{\alpha-1} \, \alpha \, \Gamma \left(\frac{d+\alpha}{2}\right)}{\pi^{\frac{d}{2}} \Gamma \left(\frac{2-\alpha}{2} \right)}.
\end{equation*}
The result is standard, see e.g. \cite{Lan72,Imb05,DrIm06} and the references therein. Here are some properties on the coefficient that will be needed later:
\begin{equation}\label{properties-Gdalpha}
\begin{cases}
\smallskip
\mbox{$G_d(\alpha)>0$ is smooth (and analytic) with respect to $\alpha \in (0,2)$;}\\
\lim_{\alpha \da 0} \frac{S_d \, G_d(\alpha)}{\alpha}=1 \mbox{ and } \lim_{\alpha \ua 2} \frac{S_d \, G_d(\alpha)}{d \, (2-\alpha)}=1, 
\end{cases}
\end{equation}
where $S_d$ is the surface measure of the unit sphere of $\R^d$. 

We then proceed to define entropy solutions of \eqref{1}.
For each $k \in \R$, we consider the Kruzhkov \cite{Kru70} entropy $u \mt |u-k|$
and entropy flux
\begin{equation*}
u \mt q_f(u,k):=\sgn (u-k) \, (f(u)-f(k)),
\end{equation*}
where throughout this paper we always consider the following
everywhere representation of the sign function:
\begin{equation}
\label{nat:representation-sign}
\sgn(u) :=
\begin{cases}
\pm 1& \mbox{if~$ \pm u > 0$,}\\
0 & \mbox{if~$u=0$.}
\end{cases}
\end{equation}
By monotonicity \eqref{Aflux} of $\varphi$, 
\begin{equation}\label{nat:key}
\sgn (u-k) \, (\varphi(u)-\varphi(k)) = |\varphi(u)-\varphi(k)|,
\end{equation}
and then we formally deduce from \eqref{nat:Levy-form} that for any function $u=u(x,t)$,
$$
\sgn (u-k) \,(- (-\triangle)^{\frac{\alpha}{2}}) \, \varphi(u) \leq
\Levy^\alpha_r [|\varphi(u)-\varphi(k)|]+ \sgn (u-k) \,
\Levy^{\alpha,r}[\varphi(u)]. 
$$
This Kato type inequality is the starting point of the entropy formulation from~\cite{CiJa11}. 

\begin{definition}[Entropy solutions]
\label{L1-entropy}
Let~$\alpha \in (0,2)$,~$u_0 \in L^\infty \cap L^1(\R^d)$, and
\eqref{flux}--\eqref{Aflux} hold. We say that $u\in L^\infty(Q_T) \cap L^\infty \left(0,T;L^1\right)$ 
is an entropy solution of~\eqref{1} provided that for all~$k\in\mathbb{R}$,~$r>0$, and all nonnegative~$\phi\in C^\infty_c(\R^{d} \times [0,T))$,
\begin{equation}
\begin{split}\label{entropy_ineq}
& \int_{Q_{T}} \Big(|u-k|\,\partial_t\phi+q_{f}(u,k) \cdot \nabla \phi \Big) \, \dif x \, \dif t \\
& + \int_{Q_{T}} \Big(|\varphi(u)-\varphi(k)| \,\Levy^\alpha_r [\phi]+ \sgn (u-k)\,\Levy^{\alpha,r} [\varphi(u)]\,\phi
\Big) \, \dif x \, \dif t \\
& 
+\int_{\R^d}
|u_0(x)-k|\,\phi(x,0) \, \dif x\geq 0.
\end{split}
\end{equation}
\end{definition}

\begin{remark}\label{rem-def-cont}
Under our assumptions, the entropy solutions are continuous in
time with values in~$L^1(\R^d)$ (cf. Theorem \ref{tm:WP}
below). Hence we get an equivalent definition if we take $\phi\in
C^\infty_c(\R^{d+1})$ and add the term~$-\int_{\R^d} |u(x,T)-k|\,\phi(x,T)
\, \dif x $ to \eqref{entropy_ineq}; see \cite{CiJa11} for more details. 
\end{remark}

Here is the well-posedness result from~\cite{CiJa11}. 
\begin{theorem}\emph{(Well-posedness)}
\label{tm:WP}
Let~$\alpha \in (0,2)$,~$u_0 \in L^\infty \cap L^1(\R^d)$, and
\eqref{flux}--\eqref{Aflux} hold. Then there
exists a unique entropy solution~$u\in L^\infty(Q_T)\cap C
\left([0,T];L^1\right)$ of~\eqref{1}, satisfying
\begin{equation}\label{nat:nonincrease}
\begin{cases}
\essinf u_0 \leq u \leq \esssup u_0,\\
\|u\|_{C \left([0,T];L^1 \right)}\leq\|u_0\|_{L^1},\\
|u|_{L^\infty \left(0,T;BV\right)}\leq|u_0|_{BV}.
\end{cases}
\end{equation}
Moreover, if $v$ is an entropy solution of~\eqref{1}
with~$v(\cdot,0)=v_0(\cdot) \in L^\infty \cap L^1(\R^d)$, then
\begin{equation}\label{nat:L1-contraction}
\|u-v\|_{C([0,T];L^1)} \leq \|u_0-v_0\|_{L^1}.
\end{equation}
\end{theorem}

\section{The main results}\label{sec:main-results}

We state our main results in this section. They compare
the entropy solution $u$ of \eqref{1} to the entropy solution $v$ of
\begin{equation}\label{bb1}
\begin{cases}
\partial_t v+\Div g(v)+(-\triangle)^{\frac{\beta}{2}} \psi(v) =0,\\
v(\cdot,0)=v_0(\cdot),
\end{cases}
\end{equation}
under the assumptions that
\begin{equation}\label{set-assumptions}
\begin{cases}
\alpha,\beta \in (0,2),\\
\mbox{$u_0 \in L^\infty \cap L^1 \cap BV(\R^d)$,~$v_0 \in L^\infty \cap L^1(\R^d)$},\\
f,g \in \left(W_{\scriptstyle \text{loc}}^{1,\infty}(\R)\right)^d \text{ with } f(0)=0=g(0),\\
\varphi,\psi \in W_{\scriptstyle \text{loc}}^{1,\infty}(\R) \text{ are nondecreasing with } \varphi(0)=0=\psi(0).
\end{cases}
\end{equation}
From now on, we will use the shorthand notation
\begin{eqnarray*}
 \|f'-g'\|_{\infty} & := & \esssup_{I(u_0)} |f'-g'|,\\
 \|\varphi'-\psi'\|_{\infty} & := & \esssup_{I(u_0)} |\varphi'-\psi'|,
\end{eqnarray*}
where~$I(u_0):=(\essinf u_0, \esssup u_0)$. We 
will 
also define
\begin{equation}\label{log-entrop}
\Ent_i (u_0):=
|u_0|_{BV} \left\{1+\left(\ln \frac{\|u_0\|_{L^1}}{|u_0|_{BV}}\right)^i\right\} \mathbf{1}_{\frac{\|u_0\|_{L^1}}{|u_0|_{BV}}>1},
\end{equation}
with the convention that~$\Ent_i (u_0)=0$ if~$|u_0|_{BV}=0$
($i=1,2$). These quantities will appear when computing the optimal constants in our main estimates. Notice that we always have $0 \leq \Ent_i (u_0)\leq
\|u_0\|_{L^1}$.


Here is our first main result.
\begin{theorem}\emph{(Continuous dependence on 
the nonlinearities)}\label{th:nonlin} 
Let \eqref{set-assumptions} hold with $\alpha=\beta$, and let $u$ and
$v$ be the entropy solutions of \eqref{1} and \eqref{bb1}
respectively. Then we have
\begin{equation}\label{main-esti-first}
\begin{split}
\|u-v\|_{C\left([0,T];L^{1}\right)} 
\leq \|u_0-v_0\|_{L^{1}}
+T \, |u_0|_{BV} \, \|f'-g'\|_{\infty}
+C \, \mathcal{E}^{\varphi-\psi}_{T,\alpha,u_0},
\end{split}
\end{equation}
with~$C=C(d,\alpha)$ and 
\begin{equation}\label{main-constants-first}
\begin{split}
& \mathcal{E}^{\varphi-\psi}_{T,\alpha,u_0} = \begin{cases}
\bigskip
T^{\frac{1}{\alpha}} \, |u_0|_{BV} \, \|(\varphi')^\frac{1}{\alpha}-(\psi')^\frac{1}{\alpha} \|_{\infty},  & \alpha\in(1,2),\\
\medskip
T \, \Ent_1(u_0) \, \|\varphi'-\psi'\|_{\infty} \\
\qquad + T \,(1+ |\ln T|) \, |u_0|_{BV} \, \|\varphi'-\psi'\|_{\infty}
\medskip\\
\bigskip
\qquad \qquad +
T   \, |u_0|_{BV} 
\, \| \varphi' \, \ln \varphi'-\psi' \, \ln \psi' \|_{\infty},  & \alpha=1,\\
T \, \|u_0\|_{L^1}^{1-\alpha} \, |u_0|_{BV}^{\alpha} \, \|\varphi'-\psi'\|_{\infty}, & \alpha \in(0,1).
\end{cases}
\end{split}
\end{equation}
\end{theorem}

The proof of this result can be found in Sections \ref{sec-sup} and
\ref{sec-sub}. 

\begin{remark}
We emphasize that 
this result 
is 
optimal with respect to the
modulus in $\varphi$. In the regimes where $T$ is 
sufficiently small or $\frac{\|u_0\|_{L^1}}{|u_0|_{BV}}$ is
sufficiently large, it is also optimal
with respect to the dependence of $T$ and $u_0$. See 
the 
discussion
of 
Section 
\ref{sec:optim}  for more details. In particular, see Proposition~\ref{prop:opt-nonlin} and
Remark~\ref{rem-opt-nonlin}.
\end{remark}

Note that our result is robust in the sense that the
  constant~$C=C(d,\alpha)$ in Theorem \ref{th:nonlin} has finite
  limits 
as~$\alpha \da 0$ or $\alp\ua2$. This will be seen during the proof,
cf. Remarks \ref{rem-pal1}\eqref{pal1} and
\ref{rem-l2}\eqref{item-rem-l2}. Hence, we can recover
the known 
continuous dependence 
estimates of the limiting cases $\alp=0$ and
$\alp=2$ (cf.~\eqref{intro-eol}), i.e. for Equations  \eqref{nat:scl} and
\eqref{nat:local-eq}.  

To show this we start by identifying the limits of the solutions $u^\alp$ of
\eqref{1} as $\alp\da0$ and $\alp\ua2$.

\begin{theorem}[Limiting equations]\label{thm:limit}
Let~$u_0 \in L^\infty \cap L^1(\R^d)$,~\eqref{flux}--\eqref{Aflux} hold, and for each~$\alpha \in (0,2)$, let $u^\alp$ denote the entropy solution
of \eqref{1}. Then~$u^\alpha$ converges in $C([0,T];L^1_{\mathrm{loc}})$, as~$\alp \da 0$ (resp.~$\alp \ua 2$), to the unique entropy solution~$u \in L^\infty(Q_T) \cap C([0,T];L^1)$ of 
\eqref{nat:scl} (resp.~\eqref{nat:local-eq}) with initial condition
$u_0$. 
\end{theorem}

Let us recall that under our assumptions there are unique entropy
solutions of~\eqref{nat:scl} and~\eqref{nat:local-eq} with initial
data $u_0$; cf.~\cite{Kru70,Car99,KaRi03}.  The proof of
Theorem~\ref{thm:limit} can be found in Section \ref{sec:limit}, as
well as the definitions of entropy solutions of
\cite{Kru70,Car99,KaRi03}. 

Now we prove that the estimates hold in the limiting cases $\alp=0$ and $\alp=2$.

\begin{corollary}[Limiting estimates]
Theorem \ref{th:nonlin} holds with
$\alp\in [0,2]$.
\end{corollary}

\begin{proof}
We only do the proof for~$\alpha=2$, the case~$\alpha=0$ being
similar. Let~$u$ and~$v$ denote the entropy solutions of~\eqref{1}
and~\eqref{bb1} with~$\alpha=2$ respectively.
Moreover, for each~$\alpha \in (0,2)$, we denote by $u^\alpha$ and~$v^\alpha$ the entropy solutions of~\eqref{1} and~\eqref{bb1} respectively, and~$\mathcal{E}(\alp)$  the right-hand side of
\eqref{main-esti-first}. Then 
$$
u-v=(u-u^\alp)+(u^\alp-v^\alp)+(v^\alp-v),
$$ 
and the triangle inequality
and Theorems  \ref{th:nonlin} and \ref{thm:limit} imply that for all~$R>0$,
\begin{align*}
\|(u-v) \, \mathbf{1}_{|x|<R}\|_{C([0,T];L^1)}\leq
{\co} (1)+\mathcal{E}(\alp)+{\co} (1) 
\end{align*}
as~$\alp \ua 2$ and $R$ is fixed.
By the monotone convergence theorem, Remark~\ref{rem-l2}\eqref{item-rem-l2}, and $\alp$-continuity of
$\mathcal{E}^{\varphi-\psi}_{T,\alpha,u_0}$ at~$\alpha=2$, the result
follows by first sending $\alp\ua 2$ and then sending $R\ra+\infty$.
\end{proof}

\begin{remark}
By our results for $\alp=2$, we get back the 
modulus
of
\cite{CoGr99},
$$\mathcal{E}^{\varphi-\psi}_{T,\alp=2,u_0} =
\sqrt{T} \, |u_0|_{BV} \, \|\sqrt{\varphi'}-\sqrt{\psi'}\|_{\infty}.$$
Our approach also gives an alternative 
proof of
this result.
\end{remark}

Optimal time regularity for \eqref{1} is another corollary of Theorem \ref{thm:limit}.

\begin{corollary}[Modulus of continuity in time]\label{th:time} 
Let~$\alpha \in [0,2]$ and~\eqref{u0}--\eqref{Aflux} hold. Let $u$ be the entropy solution of \eqref{1}. Then for all $t,s \geq 0$, 
\begin{equation}\label{esti-time}
\begin{split}
\|u(\cdot,t)-u(\cdot,s)\|_{L^{1}} 
\leq 
|u_0|_{BV} \, \|f'\|_{\infty} \, |t-s|  
+C \, \mathcal{E}^{t-s}_{\alpha,u_0,\varphi},
\end{split}
\end{equation}
with~$C=C(d,\alpha)$, 
\begin{equation*}
\begin{split}
& \mathcal{E}^{t-s}_{\alpha,u_0,\varphi} = \begin{cases}
\bigskip
|u_0|_{BV} \, \|(\varphi')^{\frac{1}{\alpha}}\|_\infty \, |t^{\frac{1}{\alpha}}-s^{\frac{1}{\alpha}}|,  & \alpha\in(1,2],\\
\medskip
\Ent_1 (u_0) \, \|\varphi'\|_\infty \, |t-s|\\
\medskip
\qquad + |u_0|_{BV} \, \|\varphi'\|_\infty \, (1+ \|\ln \varphi'\|_\infty) \, |t-s|\\ 
\bigskip
\qquad \qquad +
|u_0|_{BV} \, \|\varphi'\|_\infty \, | t \, \ln t-s \, \ln s|,  & \alpha=1,\\
\|u_0\|_{L^1}^{1-\alpha} \, |u_0|_{BV}^{\alpha} \, \|\varphi'\|_\infty \, |t-s|, & \alpha \in[0,1),
\end{cases}
\end{split}
\end{equation*}
and where~$\Ent_1 (u_0)$ is defined in~\eqref{log-entrop}.
\end{corollary}

\begin{remark}
\label{remtime}
This result is optimal with respect to the modulus in time, and also with
respect to the dependence of $\varphi$ and $u_0$ in the regimes where
$\|\varphi'\|_\infty$ is sufficiently small or the ratio
$\frac{\|u_0\|_{L^1}}{|u_0|_{BV}}$ is 
sufficiently large, cf. Remark \ref{rem_optt}. 
The result improves earlier results by the two
last authors in \cite{CiJapr} where the modulus was given as
$$
\mathcal{E}^{t-s}_{\alpha,u_0,\varphi}=C(\alpha,u_0,\varphi)\left\{ 
\begin{array}{ll}
|t-s|^{\frac1{\alp}},&\alp> 1,\\
|t-s| \left(1+|\ln|t-s||\right),&\alp=1,\\
|t-s|,&\alp< 1.
\end{array}
\right.$$ 
The optimal new results give a strictly better 
modulus of continuity 
when
$\alp\in[1,2]$ at the initial time\footnote{Since $\liminf_{t,s \da 0} \frac{|t^\frac{1}{\alpha}-s^\frac{1}{\alpha}|}{|t-s|^{\frac1{\alp}}}=0=\liminf_{t,s \da 0} \frac{|t \, \ln t -s \, \ln
s|}{|t-s| \, |\ln |t-s||}$ (take 
$t_n,s_n
\da 0$ and $\frac{t_n}{s_n} \ra 1$).} and for positive times
$u \in  W_{\scriptstyle   \text{loc}}^{1,\infty}((0,+\infty];L^1)$. The Lipschitz in time result is a 
regularizing effect when the solution is no more than $BV$ initially.
\end{remark}

\begin{proof}
We fix~$t,s>0$ and introduce the rescaled solutions~$v(x,\tau):=u(x,t \, \tau)$ and~$w(x,\tau):=u(x,s \, \tau)$. These are solutions of~\eqref{1} with initial data~$u_0$, new respective fluxes~$t \, f$ and~$s \, f$, and new respective diffusion functions~$t \, \varphi$ and~$s \, \varphi$. The result immediately follows from the preceding corollary applied at time~$\tau=1$.
\end{proof}

Next we consider the continuous dependence on~$\alpha$. 
Given
$\lambda \in (0,2)$, we 
define ``the best Lipschitz constant'' of $\alpha \mt u^\alpha$ at the position $\alpha=\lambda$ as follows:
\begin{equation}
\label{nat:lip}
\Lip_\alpha(u;\lambda):= \limsup_{\alpha,\beta \ra
 \lambda} \frac{\|u^\alpha-u^\beta\|_{C([0,T];L^1)}}{|\alpha-\beta|},
\end{equation}
where~$u^\alpha$ denotes the unique entropy solution of~\eqref{1}. 

\begin{theorem}\emph{(Lipschitz continuity in~$\alpha$)}\label{nat:th-Levy}
Let $\lambda \in (0,2)$ and~\eqref{u0}--\eqref{Aflux} hold. Then
\begin{equation}\label{main-esti-second}
\Lip_\alpha(u;\lambda) \leq 
C \, \begin{cases}
\bigskip
M^{\frac{1}{\lambda}} \, (1 + |\ln M|) \,
|u_0|_{BV},  & 
\lambda \in(1,2),\\
\medskip
M \Ent_2 (u_0) +M \, (1 + \ln^2 M) \, |u_0|_{BV}, & 
\lambda = 1,
\bigskip
\\
M  \, \|u_0\|_{L^1}^{1-\lambda} \, |u_0|_{BV}^{\lambda}  \left( 1 + \left|\ln \frac{\|u_0\|_{L^1}}{|u_0|_{BV}} \right|\right), & 
\lambda\in(0,1),\\
\end{cases}
\end{equation}
where~$C=C(d,\lambda)$,~$M:=T \, \|\varphi'\|_\infty$ and~$\Ent_2 (u_0)$ is defined in~\eqref{log-entrop}. In particular, the function $\alpha \in (0,2) \mt u^\alpha \in C([0,T];L^1)$ is locally Lipschitz continuous. 
\end{theorem}

The proof of Theorem~\ref{nat:th-Levy} can be found in Sections \ref{sec-sup} and
\ref{sec-sub}. 
\begin{remark}
This result is optimal with respect to the dependence of
$M$ and $u_0$ in the regimes where $M$ is 
sufficiently small or  $\frac{\|u_0\|_{L^1}}{|u_0|_{BV}}$ is
sufficiently large. An example is given in Section
\ref{sec:optim},  cf. Proposition~\ref{prop:opt-powers} and Remark \ref{rem-opt-powers}.
\end{remark}

\begin{remark} With Theorem \ref{th:nonlin}, Corollary \ref{th:time}
  and Theorem \ref{nat:th-Levy} in hands, we can easily get an
  explicit continuous dependence estimate of~$u$ with respect to the
  quintuplet $(t,\alpha,u_0,f,\varphi)$ under \eqref{set-assumptions}.
\end{remark}

\subsection*{Further comments and open problems}

\subsubsection*{{\rm A.} Robustness of the Lipschitz estimates in~$\alpha$ as $\alp \da 0$ or $\alp \ua 2$} 

In Theorem~\ref{nat:th-Levy}, $C=C(d,\lambda)$ blows up as $\lambda \da 0$ or $\lambda \ua 2$, and
we do not get Lipschitz regularity in $\alp$ up to the boundaries
$\alpha = 0$ and $\alpha=2$.  

At $\alpha=0$, we can do no better because
the entropy solutions of \eqref{1} may not even converge toward the entropy solution of \eqref{nat:local-eq} in 
$L^1$ as $\alpha \da 0$.
The reason is that the mass preserving property could be lost at the limit. This was already observed in Section 11 of \cite{DeQuRoVa12} for the fractional porous medium equation \eqref{referee-fpm} below. Note that the convergence always holds in $L^1_{\scriptstyle \text{loc}}$ by Theorem \ref{thm:limit}, so that an interesting question is whether it holds in $L^p$ for any $p \in (1,+\infty)$. To the best of our knowledge, this problem is still open at least for the full equation \eqref{1}.

At $\alpha=2$, it is an open problem whether $\alpha \mt
u^\alpha$ is Lipschitz  with values in $L^1$ or not. 
This 
problem 
is related to the following
problems: Do the entropy solutions of \eqref{1} converge toward
the entropy solution of \eqref{nat:local-eq} in $L^1$ or $L^p$ as
$\alpha \ua 2$? If yes, what is the optimal rate of convergence? Note that here again the convergence holds in $L^1_{\rm loc}$ by Theorem
\ref{thm:limit}, and it moreover holds in $L^1$ for Equation \eqref{referee-fpm}  by~\cite{DeQuRoVa12}. 

\subsubsection*{{\rm B.} Implications for the fractional porous medium
  equation}

In \cite{DeQuRoVa12}, the following 
Cauchy problem is studied: 
\begin{equation}\label{referee-fpm}
\partial_t u+(-\triangle)^{\alpha/2} (|u|^{m-1} \, u)=0 \quad \mbox{and} \quad
u(\cdot,0)=u_0(\cdot), 
\end{equation}
where $\alp\in(0,2)$ and $m>0$.
The authors prove that if $u_0 \in L^1(\R^d)$, there exists a unique
mild solution which under further assumptions ($m \geq 1$ is
sufficient) is the (unique) strong solution. By Theorems 10.1 and 10.3 of 
\cite{DeQuRoVa12}, this solution is continuous in the data
$(\alpha,m,u_0) \in D \times L^1(\R^d)$ with values in $C([0,+\infty);L^1)$, 
where 
$$
D:=
\left\{(\alpha,m):0<\alpha \leq 2,  \,  m > \frac{(d-\alpha)^+}{d}
\right\}.
$$
We will now show that this dependence is locally Lipschitz in some cases.

Let us first establish the equivalence between entropy and strong solutions.
\begin{lemma}
\label{lem_equiv}
Let $u_0 \in L^\infty \cap L^1(\R^d)$, $m \geq 1$, and $u$ be the
unique entropy solution of \eqref{referee-fpm} given by Theorem
\ref{tm:WP} (with 
$T=+\infty$). Then $u$ coincides with the unique strong solution of
\eqref{referee-fpm} (cf. Definition 3.5 in \cite{DeQuRoVa12}).
\end{lemma}
\begin{proof}
Note that $u \in L^\infty(\R^d \times (0,+\infty)) \cap C([0,+\infty);L^1)$. By Lemma \ref{Clem2}, we also have  
$|u|^{m-1} \, u \in L^2(0,+\infty;H^\frac{\alpha}{2})$. 
Here $H^\frac{\alpha}{2}(\R^d)$ 
is the usual fractional Sobolev space
defined in \eqref{referee-def-norm}.  Let us also recall that $u$ satisfies the equation in $\mathcal{D}'(\R^d \times (0,+\infty))$ and the initial condition $u(\cdot,0)=u_0(\cdot)$ almost everywhere, cf. \cite{CiJa11}.
It follows that $u$ is a weak solution in the
sense of Definition 3.1 in \cite{DeQuRoVa12}. Since $u$ is bounded,
Corollary 8.3 of \cite{DeQuRoVa12} completes the proof.  
\end{proof}

Theorems \ref{th:nonlin} and
\ref{nat:th-Levy} and Lemma \ref{lem_equiv} 
then 
imply the following result:
\begin{corollary}
For all $T>0$, 
the
unique strong solution $u$ to \eqref{referee-fpm} is locally Lipschitz
continuous in $(\alpha,m,u_0) \in \tilde D\times \left(L^\infty \cap L^1
\cap BV(\R^d)\right)$ with values in $C([0,T];L^1)$, where
$$
\tilde D:=
\left\{(\alpha,m):0<\alpha < 2,  \,  m  >  1\right\}.
$$
\end{corollary}

If
$u_0 \notin L^\infty \cap BV(\R^d)$, it 
is 
possible 
to find an explicit (non-Lipschitz) modulus of continuity for the function $(\alpha,m) \in \tilde{D} \mt u \in C([0,T];L^1)$. To do so, it suffices to use an approximation argument and the $L^1$-contraction principle. 
It is an open problem whether this would give an optimal modulus or not. 
It is also 
an open problem to find an explicit modulus 
when $(\alpha,m) \notin  \tilde D$.

\section{Two general results from \cite{AlCiJa12}}\label{sec-remind}

In this section we recall two key results developed in \cite{AlCiJa12}
for the more general case where the diffusion operator can be the
generator of an arbitrary pure jump L\'evy process. 
First we state the Kuznetsov type lemma of \cite{AlCiJa12} that measures
the $L^1$-distance between~$u$ and an arbitrary function~$v$.
From now on, let $\epsilon$ and $\nu$ be positive parameters and
$\phi^{\epsilon,\nu} \in C^\infty(\R^{2d+2})$ denote the test function
\begin{eqnarray}
\phi^{\epsilon,\nu}(x,t,y,s) := \theta_\nu(t-s) \,
\rho_{\epsilon}(x-y):=\frac{1}{\nu} \,
\theta \left(\frac{t-s}{\nu} \right) \frac{1}{\epsilon^d} \, \rho \left(\frac{x-y}{\epsilon}\right),
\label{nat:test-kuznetsov}
\end{eqnarray}
where
\begin{equation*}
\begin{cases}
\theta \in C_{c}^{\infty}(\R), \quad  
\theta \geq 0, \quad \supp \theta \subseteq [-1,1], \quad \int
\theta =1,\\
\rho \in C_{c}^{\infty}(\mathbb{R}^{d}),
\quad
\rho \geq 0, 
\quad  \mbox{and $\int
\rho =1$.}
\end{cases}
\end{equation*}
We also let~$m_u(\nu)$ denote the modulus of continuity in time of $u
\in C\left([0,T];L^1\right)$.  
\begin{lemma}[Kuznetsov type Lemma]\label{lem:kuznetsov}
Let~$\alpha \in (0,2)$,~$u_0 \in L^\infty \cap L^1 \cap BV(\R^d)$, and let us assume~\eqref{flux}--\eqref{Aflux}. Let~$u$ be the entropy
solution of~\eqref{1} and let~$v\in L^\infty(Q_T)\cap C
\left([0,T];L^1\right)$ be such that~$v(\cdot,0)=v_0(\cdot)$. Then for all~$r,\epsilon>0$ and $T>\nu>0$,
\begin{equation}\label{kuz}
\begin{split}
& \|u(\cdot,T)-v(\cdot,T)\|_{L^{1}}\\
& \leq \|u_0-v_0\|_{L^{1}}+C_{\rho} \, |u_0|_{BV} \, \epsilon+2\,m_{u}(\nu) \vee  m_{v}(\nu) \\
& \quad -\int_{Q_{T}^2}|v(x,t)-u(y,s)|\,\partial_t \phi^{\epsilon,\nu}(x,t,y,s)\, \dif w\\
& \quad  -\int_{Q_{T}^2} q_f(v(x,t),u(y,s)) \cdot \nabla_x\phi^{\epsilon,\nu}(x,t,y,s)\, \dif w\\
&\quad + \int_{Q_{T}^2} 
|\varphi(v(x,t))-\varphi(u(y,s))|\, \Levy^\alpha_r[\phi^{\epsilon,\nu}(x,t,\cdot,s)](y)
\, \dif w \\
& \quad - \int_{Q_{T}^2} \sgn (v(x,t)-u(y,s))\,\Levy^{\alpha,r}[\varphi(u(\cdot,s))](y) \,\phi^{\epsilon,\nu}(x,t,y,s) \, \dif w\\
& \quad +\int_{\R^d \times Q_T} |v(x,T)-u(y,s)|\,\phi^{\epsilon,\nu}(x,T,y,s) \, \dif x \, \dif y \, \dif s\\
& \quad -\int_{\R^d \times Q_T}
|v_0(x)-u(y,s)|\,\phi^{\epsilon,\nu}(x,0,y,s) \, \dif x \, \dif y
\, \dif s
\end{split}
\end{equation}
where~$\dif w:=\dif x\, \dif t \, \dif y\,\dif s$ and
the constant~$C_{\rho}$ only depends on~$\rho$. 
\end{lemma}
\begin{proof}
This is Lemma 3.1 of \cite{AlCiJa12} with the particular
diffusion operator \eqref{nat:Levy-form}. 
\end{proof}

In the setting of this paper, the general continuous dependence
estimates of~\cite{AlCiJa12} take the following form: 
\begin{theorem}\label{th:nonlin-o}
Let us assume~\eqref{set-assumptions} and let~$u$ and~$v$ be the respective entropy solutions of~\eqref{1}
and~\eqref{bb1}. Then for all~$r>0$,
\begin{equation*}
\begin{split}
\|u-v\|_{C\left([0,T];L^{1}\right)} 
& \leq \|u_0-v_0\|_{L^{1}}
+ T \, |u_0|_{BV} \, \|f'-g'\|_{\infty}
+\mathcal{E}^{\alpha-\beta,\varphi-\psi}_{T,\alpha,\beta,u_0,\varphi,r}
\end{split}
\end{equation*}
with
\begin{equation}\label{te-gt}
\begin{split}
& \mathcal{E}^{\alpha-\beta,\varphi-\psi}_{T,\alpha,\beta,u_0,\varphi,r} =\\
& 
\begin{cases}
\medskip 
 T  \int_{|z|> r} \|u_0(\cdot+z)-u_0(\cdot)
\|_{L^1} \, \dif \mu_\alpha(z) \, \|\varphi'-\psi'\|_{\infty}& \\
\bigskip
\qquad \qquad + c_d \, \sqrt{T} \, |u_0|_{BV} \,\sqrt{\int_{|z|< r}
|z|^2 \, \dif \mu_\alpha(z)
 \, \|\varphi'-\psi'\|_{\infty}}, & \mbox{$\alpha=\beta$},\\
\medskip
 M  \int_{|z|> r}
\|u_0(\cdot+z)-u_0(\cdot) \|_{L^1} \, \dif | \mu_\alpha- \mu_\beta|(z) &\\
\qquad \qquad +  
c_d \, \sqrt{M} \, |u_0|_{BV} \,\sqrt{\int_{|z| < r} |z|^2 \,\dif | \mu_\alpha- \mu_\beta|(z)}, & \mbox{$\varphi=\psi$,}
\end{cases}
\end{split}
\end{equation}
where~$\dif \mu_\alpha(z)=\frac{G_d(\alpha)}{|z|^{d+\alpha}} \, \dif z$,~$M=T \, \|\varphi'\|_\infty$ and~$c_d=\sqrt{\frac{4 \, d^2}{d+1}}$.
\end{theorem}

\begin{proof}
This is Theorems 3.3 and 3.4 of \cite{AlCiJa12} with the special choice
of diffusion \eqref{nat:Levy-form} and L\'evy measure
$\frac{G_d(\alpha)}{|z|^{d+\alpha}} \, \dif z$. 
\end{proof}

\section{Continuous dependence in the supercritical case}\label{sec-sup}

In this section we use Theorem \ref{th:nonlin-o} to prove
Theorems \ref{th:nonlin} and \ref{nat:th-Levy} for supercritical diffusions.

\begin{proof}[Proof of Theorem~\ref{th:nonlin} when~$\alpha<1$]
We use Estimate~\eqref{te-gt} with $\beta=\alpha$. 
The worst term $\sqrt{\int_{|z|< r}|z|^2 \, \dif \mu_\alpha (z)
 \, \|\varphi'-\psi'\|_{\infty}}$ 
vanishes when $r\da0$, and hence
$$
\mathcal{E}^{\alpha-\beta,\varphi-\psi}_{T,\alpha,\beta,u_0,\varphi,r}
\ \underset{r\da0}{\longrightarrow} \ I:=T \int \|u_0(\cdot+z)-u_0(\cdot)
\|_{L^1} \, \dif \mu_\alpha(z) \, \|\varphi'-\psi'\|_\infty.
$$
To estimate this integral, we consider separately the domains $|z| >
\tilde{r}$ and $|z| < \tilde{r}$ for 
arbitrary~$\tilde{r}>0$. In the second domain, we use
the inequality $\|u_0(\cdot+z)-u_0(\cdot) 
\|_{L^1} \leq |u_0|_{BV} \, |z|$. A direct computation using the fact
that $\alpha<1$, then leads to
\begin{equation*}
\begin{split}
& I \leq 2 \, T \, \|u_0\|_{L^1} \, \|\varphi'-\psi'\|_{\infty} \, S_d \, \frac{G_d(\alpha)}{\alpha}  \, \tilde{r}^{-\alpha}+T \, |u_0|_{BV} \, \|\varphi'-\psi'\|_{\infty} \, 
S_d \,\frac{G_d(\alpha)}{1-\alpha} \, \tilde{r}^{1-\alpha}
\end{split}
\end{equation*} 
(where $S_d$ is the surface measure of the unit sphere of~$\R^d$). We
complete the proof by taking $\tilde{r}=\|u_0\|_{L^1} \, |u_0|_{BV}^{-1}$.   
\end{proof}

\begin{remark}\label{rem-pal1}
\begin{enumerate}
\item \label{pal1} From the proof, we have $C \leq
  S_d \left(\frac{2 \, G_d(\alpha)}{\alpha}+\frac{G_d(\alpha)}{1-\alpha}\right)$
  in~\eqref{main-esti-first}  when~$\alpha<1$. By \eqref{properties-Gdalpha}, $\lim_{\alp\da0}C(d,\alpha)$ is finite and
  independent of $d$.   
\item 
We also have $C \leq
  S_d \left(\frac{2 \, G_d(\alpha)}{\alpha}+\frac{G_d(\alpha)}{1-\alpha}\right)$
 when~$\alpha<1$  in~\eqref{esti-time}, since we have seen that this estimate is a simple rewriting of the preceding one by rescaling the time variable.  
\end{enumerate}
\end{remark}

\begin{proof}[Proof of Theorem~\ref{nat:th-Levy} when~$\lambda \in (0,1)$]
Given $\alpha,\beta \in (0,2)$, we use Theorem~\ref{th:nonlin-o} with $u=u^\alpha$ and $v=u^\beta$, i.e. with $(u_0,f,\varphi)=(v_0,g,\psi)$. As in
the preceding proof, we pass to the limit as~$r \da 0$
in~\eqref{te-gt} and we cut the remaining integral in two parts. We find that
\begin{equation}\label{tech-sd-1}
\begin{split}
& \|u^\alpha-u^\beta\|_{C([0,T];L^1)} \\
& \leq 2 \, M
\, \|u_0\|_{L^1} \underbrace{\int_{|z|> \tilde{r}}  \dif |\mu_\alpha-\mu_\beta|(z)}_{=:J_1}+M \, |u_0|_{BV} \underbrace{\int_{|z| < \tilde{r}} |z| \, \dif |\mu_\alpha-\mu_\beta|(z)}_{=:J_2}.
\end{split}
\end{equation}
In the rest of the proof we use the letter $C$ to denote various constants $C=C(d,\lambda)$. 

We have
\begin{align}\label{lc-ref-start}
 J_1 & =  \int_{|z|>\tilde{r}} |G_d(\alpha) \, |z|^{-d-\alpha}-G_d(\beta) \, |z|^{-d-\beta}| \, \dif z\\
&  \leq |G_d(\alpha)-G_d(\beta)| \, \max_{\sigma=\alpha,\beta} \int_{|z| > \tilde{r}} \frac{\dif z}{|z|^{d+\sigma}}
\nonumber
\\
& \quad + (G_d(\alpha) \vee G_d(\beta)) \underbrace{\int_{|z| > \tilde{r}} \left||z|^{-d-\alpha}-|z|^{-d-\beta}\right| \dif z}_{=:\tilde{J}_{1}},\nonumber
\end{align}
where $
\tilde{J}_{1} \leq S_d  \left|\frac{\tilde{r}^{-\alpha}}{\alpha}-\frac{\tilde{r}^{-\beta}}{\beta} \right|+2 \, S_d \left|\frac{1}{\alpha}-\frac{1}{\beta}\right| \mathbf{1}_{\tilde{r} <1}.
$
We have estimated $\tilde{J}_{1}$ using the fact that
$|z|^{-d-\alpha}-|z|^{-d-\beta}$ has a sign both inside and 
outside the unit ball. By \eqref{properties-Gdalpha} and a simple passage to the limit under the integral sign, 
\begin{equation*}
\begin{split}
\limsup_{\alpha,\beta \ra \lambda} \frac{J_1}{|\alpha-\beta|} 
\leq \underbrace{C \, (\tilde{r}^{-\lambda}+\mathbf{1}_{\tilde{r}<1})}_{\leq C \, \tilde{r}^{-\lambda}}+C \, \underbrace{\limsup_{\alpha,\beta \ra \lambda} \frac{1}{|\alpha-\beta|} \left|\frac{\tilde{r}^{-\alpha}}{\alpha}-\frac{\tilde{r}^{-\beta}}{\beta} \right|}_{=:\tilde{\tilde{J}}_{1}}.
\end{split}
\end{equation*}
By the Taylor formula with integral remainder, 
$$
\tilde{\tilde{J}}_{1}=\limsup_{\alpha,\beta \ra \lambda} \Big| 
\int_0^1 \frac{\alpha_\tau \, \tilde{r}^{-\alpha_\tau} \, \ln \tilde{r}+\tilde{r}^{-\alpha_\tau}}{\alpha_\tau^2} \, \dif \tau \Big| \leq C \, \tilde{r}^{-\lambda} \, (1+|\ln \tilde{r}|),
$$
where $\alpha_\tau:=\tau \, \alpha+(1-\tau) \, \beta$. We deduce the following estimate:
\begin{equation}\label{tech-p-10}
\limsup_{\alpha,\beta \ra \lambda} \frac{J_{1}}{|\alpha-\beta|} \leq C \, \tilde{r}^{-\lambda}
\, (1 + |\ln \tilde{r}|).
\end{equation}  
Let us notice that this estimate works for all $\lambda \in (0,2)$. By similar arguments, we also have
$$
\limsup_{\alpha,\beta \ra \lambda} \frac{J_2}{|\alpha-\beta|} \leq C \, \tilde{r}^{1-\lambda}
\, (1 + |\ln \tilde{r}|),
$$ 
but this time we have to use that $\lambda<1$. Inserting these inequalities
into~\eqref{tech-sd-1}, we find that for all~$\tilde{r}>0$,
\begin{equation*}
\begin{split}
\Lip_\alpha (u;\lambda) \leq  C \, M \, (1 + |\ln \tilde{r}|) \, ( \|u_0\|_{L^1} \, \tilde{r}^{-\lambda}+ |u_0|_{BV} \,  \tilde{r}^{1-\lambda}).
\end{split}
\end{equation*}
To conclude we take $\tilde{r}=\|u_0\|_{L^1} \,
|u_0|_{BV}^{-1}$. 
\end{proof}

\begin{remark}
\begin{enumerate}
\item
When~$\alpha \geq 1$, the estimate in $\varphi-\psi$ of Theorem~\ref{th:nonlin-o} is 
not
optimal. 
Indeed, let $\alpha=\beta$, $u_0$ be such that $\|u_0(\cdot+z)-u_0(\cdot)
\|_{L^1} \sim |z|$ as $z \ra 0$, and $\omega_{\varphi-\psi}:=\inf_{r>0} \mathcal{\E}^{\alpha-\beta,\varphi-\psi}_{T,\alpha,\beta,u_0,\varphi,r}$ be the best modulus given by Theorem~\ref{th:nonlin-o}. Then
\begin{equation*}
\omega_{\varphi-\psi} \sim
\begin{cases}
\|\varphi'-\psi'\|_\infty^\frac{1}{\alpha},& \alpha>1 ,\\
\|\varphi'-\psi'\|_\infty \, \left|\ln \|\varphi'-\psi'\|_\infty
\right|, & \alpha=1,
\end{cases}
\end{equation*}
as ${\|\varphi'-\psi'\|_\infty \ra 0}$, thanks to the minimization giving $r \sim
\|\varphi'-\psi'\|_\infty^\frac{1}{\alpha}$. These moduli are strictly worse than those in
\eqref{main-constants-first} e.g. when $\varphi' \equiv a$, $\psi' \equiv b$, $a,b >0$.\footnote{Indeed $\lim_{a,b \ra c} \frac{|a^\frac{1}{\alpha}-b^\frac{1}{\alpha}|}{|a-b|^\frac{1}{\alpha}}=0=\lim_{a,b \ra c} \frac{|a \, \ln a-b \, \ln b|}{|a -b| \left| \ln |a -b| \right|}$ for any $c > 0$ and even for $c=0^+$ by taking liminfs.}
\item Theorem~\ref{th:nonlin-o} does not imply the local Lipschitz continuity in $\alpha \in [1,2)$. 
Indeed, let $\varphi=\psi$ be nontrivial and $u_0$ be as above. Then the modulus $\omega_{\alpha-\beta}:=\inf_{r>0} \mathcal{E}^{\alpha-\beta,\varphi-\psi}_{T,\alpha,\beta,u_0,\varphi,r}$ is worse than any Lipschitz modulus since $\lim_{\alpha,\beta \ra \lambda} \frac{\omega_{\alpha-\beta}}{|\alpha-\beta|}=+\infty$ for all $\lambda \in [1,2)$.\footnote{If not, there are $\alpha_n,\beta_n \ra \lambda$ and $r_n \ra r_\ast \in [0,+\infty]$ such that $\lim \frac{\omega_{\alpha_n-\beta_n}}{|\alpha_n-\beta_ n|} <+\infty$ and
\begin{equation*}
\begin{split}
& \frac{\omega_{\alpha_n-\beta_n}}{|\alpha_n-\beta_ n|} 
=\co(1)+ \\
& \underbrace{\int_{|z|> r_n}
\|u_0(\cdot+z)-u_0(\cdot) \|_{L^1} \, \frac{\dif |\mu_{\alpha_n-}\mu_{\beta_n}|(z)}{|\alpha_n-\beta_n|}}_{=:I_n}+
\underbrace{\sqrt{\int_{|z| < r_n} |z|^2 \, \frac{\dif |\mu_{\alpha_n-}\mu_{\beta_n}|(z)}{|\alpha_n-\beta_n|^2}}}_{=:J_n}
\end{split}
\end{equation*}
($M=c_d \,\sqrt{M} \,  |u_0|_{BV}=1$ to simplify).
By Fatou's lemma $\liminf J_n^2 \geq \int_{|z| < r_\ast} |z|^2 \, (+\infty) \, \dif z$ and
$\liminf I_n \geq \int_{|z| > r_\ast} \|u_0(\cdot+z)-u_0(\cdot) \|_{L^1} \, |G'_d(\lambda) -G_d(\lambda) \, \ln |z|| \, |z|^{-d-\lambda} \, \dif z$. This is not possible since these integrals can not be both finite at the same time.}
\end{enumerate}
\end{remark}

\section{Continuous dependence in the critical and subcritical cases}
\label{sec-sub}

Since we can not use Theorem~\ref{th:nonlin-o}
any more, we start from Lemma~\ref{lem:kuznetsov} and take
advantage of the homogeneity of the fractional Laplacian. 
We thus use the Kruzhkov type doubling of variables techniques introduced
in \cite{Kru70} along with ideas from
\cite{Kuz76}; see also \cite{ScTa92,CoGrLo96,Dro03,JaWi03,DrIm06,Ali07,RoYo07,CiJa11,KaUl11,AlCiJa12,CiJapr}
for other applications of this technique to
nonlocal equations. We recall that the
idea is to consider $v$ to be a function of $(x,t)$, 
$u$ to be a function of $(y,s)$, and use the 
approximate unit $\phi^{\epsilon,\nu}(x,t,y,s)$ 
in \eqref{nat:test-kuznetsov} as a test function. For brevity, we do not specify
the variables  of~$u,v$, and $\phi^{\epsilon,\nu}$ when the context is
clear. Finally, we recall that $\dif w=\dif x \, \dif t \, \dif y \, \dif s$.

\subsection{A technical lemma}
In order to adapt the ideas of~\cite{Kuz76} to the nonlocal case, we
need the following Kato type of inequality. 
The reader could skip this technical subsection at the first reading.

\begin{lemma}\label{lem-kato}
Let~$\alpha \in (0,2)$,~$c,\tilde{c} \in \R$,~$\gamma,\tilde{\gamma} \in \R$ and~$I$ be a real interval with a positive lower bound. Let~$u,v \in L^1(Q_T)$,~$\varphi$ satisfy~\eqref{Aflux} and~$\phi^{\epsilon,\nu}$ be the test function in~\eqref{nat:test-kuznetsov}. Then
\begin{equation*}
\begin{split}
& \E\\ 
& := \int_{Q_{T}^2} \int_{|z| \in I} \sgn(v(x,t)-u(y,s))\\
& \quad \cdot\frac{\left\{\varphi \left(v (x+\tilde{c} \, |z|^{\tilde{\gamma}-1} \, z ,t)\right)-\varphi \left(u(y+c \, |z|^{\gamma-1} \, z,s)\right)\right\}-\left\{\varphi(v(x,t))-\varphi(u(y,s))\right\}}{|z|^{d+\alpha}}\\ 
& \quad \cdot\phi^{\epsilon,\nu}(x,t,y,s) \, \dif z \, 
   \dif w\\
& \leq \int_{Q_{T}^2}  \int_{|z| \in I} \left|\varphi(v(x,t))-\varphi(u(y,s))\right|  \theta_\nu(t-s) \frac{\rho_\epsilon \left(x-y+h(z)\right)-\rho_\epsilon(x-y)}{|z|^{d+\alpha}} \,\dif z \, \dif w,
\end{split}
\end{equation*}
with~$h(z):=(\tilde{c} \, |z|^{\tilde{\gamma}-1}-c \, |z|^{\gamma-1}) \, z$.
In particular, if~$c=\tilde{c}$ and~$\gamma=\tilde{\gamma}$, then~$\E \leq 0$.
\end{lemma}

\begin{proof}
Note that~$\E$ is well-defined as ``convolution-like integral
of~$L^1$-functions.'' Indeed,~$\phi^{\epsilon,\nu}(x,t,y,s)=\theta_\nu(t-s)
\, \rho_\epsilon(x-y)$, where~$\theta_\nu$ and~$\rho_\epsilon$ are
approximate units, so that by Fubini,   
\begin{equation*}
\begin{split}
& \int_{Q_{T}^2} \int_{|z| \in I} \phi^{\epsilon,\nu} \\
& \cdot \left|\frac{\left\{ \varphi \left(v (x+\tilde{c} \, |z|^{\tilde{\gamma}-1} \, z ,t)\right)-\varphi \left(u(y+c \, |z|^{\gamma-1} \, z,s)\right)\right\}-\{ \varphi(v)-\varphi(u)\}}{|z|^{d+\alpha}} \right|  \dif z \, \dif w \\
& \leq 2 \left(\|\varphi(u)\|_{L^1(Q_T)}+\|\varphi(v)\|_{L^1(Q_T)}\right)  \int_{|z|\in I} \frac{\dif z}{|z|^{d+\alpha}}<+\infty,
\end{split}
\end{equation*} 
since~$u$ and~$v$ are~$L^\infty \cap L^1$,~$\varphi$ is~$W_{\scriptstyle \text{loc}}^{1,\infty}$ with~$\varphi(0)=0$, and~$\inf I>0$.

Then by~\eqref{nat:key}
and the nonnegativity of~$\phi^{\epsilon,\nu}$,
\begin{equation*}
\begin{split}
\E & \leq \int_{Q_{T}^2} \int_{|z| \in I} \phi^{\epsilon,\nu} \\
& \quad \cdot\frac{\left|\varphi \left(v (x+\tilde{c} \, |z|^{\tilde{\gamma}-1} \, z ,t)\right)-\varphi \left(u(y+c \, |z|^{\gamma-1} \, z,s)\right)\right| -|\varphi(v)-\varphi(u) |}{|z|^{d+\alpha}} \, \dif z \, 
   \dif w\\
&  = \int_{Q_{T}^2}  \int_{|z| \in I} |\varphi(v)-\varphi(u)| \\
& \quad   \cdot\underbrace{
\left\{
\phi^{\epsilon,\nu} 
(x+\tilde{c} \, |z|^{\tilde{\gamma}-1} \, z,t,y+c \, |z|^{\gamma-1} \, z,s)
-\phi^{\epsilon,\nu}\right\}
}_{=\theta_\nu(t-s) \left\{ \rho_\epsilon 
\left(x-y+(\tilde{c} \, |z|^{\tilde{\gamma}-1} - c \, |z|^{\gamma-1} ) \, z\right)
-\rho_\epsilon(x-y) \right\}} \frac{\dif z}{|z|^{d+\alpha}} \, \dif w;
\end{split}
\end{equation*}
the last line has been obtained by splitting the integral in two
pieces and using the change of variable $
(x+\tilde{c} \, |z|^{\tilde{\gamma}-1} \, z,t,y+c \, |z|^{\gamma-1} \, z,s,-z)
\mt (x,t,y,s,z)$. The proof is complete.
\end{proof}

\subsection{Proof of Theorem~\ref{th:nonlin}} During the proof we
freeze the nonlinear diffusion functions and  use a sort of
linearization procedure. The techniques could look a little
bit like the ones in Young measure theory and kinetic
formulations~\cite{LiPeTa94,BoPe98,ChKa06}.  

\begin{proof}[Proof of Theorem~\ref{th:nonlin}] \
\smallskip

\noindent {\bf 1.} {\it Initial reduction.} We first reduce the proof
to the case where 
\begin{equation}\label{hnd}
\begin{cases}
v_0=u_0,\\
\mbox{$\varphi'$ and~$\psi'$ vanish outside~$I(u_0)$ and take values in~$[\, \underline{\Lambda},\overline{\Lambda}\,]$,}
\end{cases}
\end{equation}
with~$I(u_0)=(\essinf u_0,\esssup u_0)$ and for some~$\overline{\Lambda} \geq \underline{\Lambda}>0$.
Let us justify that we can do this without loss of generality.

Since $u$ takes its values
in~$I(u_0)$ by~\eqref{nat:nonincrease}, we can redefine $\varphi$ 
to be constant outside this interval without changing the solutions
of the initial-value problem \eqref{1}. Hence 
$\overline{\Lambda}$ could be taken as a Lipschitz constant
of~$\varphi$ on $I(u_0)$. In a similar way, we could also modify~$\psi$
outside~$I(u_0)$ if~$v_0=u_0$. The last assumption is no restriction. Indeed,
by \eqref{nat:L1-contraction},
$$
\|u-v\|_{C([0,T];L^1)} \leq \|u-w\|_{C([0,T];L^1)}+\underbrace{\|w-v\|_{C([0,T];L^1)}}_{\leq \|u_0-v_0\|_{L^1}}
$$
for the entropy solution~$w$ of~\eqref{bb1} with initial data
$u_0$; hence, \eqref{main-esti-first} of Theorem~\ref{th:nonlin} holds
for~$u-v$ whenever it does for~$u-w$. Finally,
if~$\underline{\Lambda}$ does not exist, we can always consider 
sequences~$\varphi_n(\xi):=\varphi(\xi)+\frac{\xi}{n}$
and~$\psi_n(\xi):=\psi(\xi)+\frac{\xi}{n}$ for which it does. The
associated entropy solutions~$u_n$ and~$v_n$ 
respectively converge to~$u$ and~$v$ in~$C([0,T];L^1)$ by e.g.
Theorem~\ref{th:nonlin-o}. Consequently, if we could
prove \eqref{main-esti-first} for $u_n-v_n$, it would follow for $u-v$
by going to the limit. 

In the rest of the proof we always assume \eqref{hnd}.

\medskip

\noindent {\bf 2.} {\it Applying Kuznetsov.} Let
us use the entropy inequality~\eqref{entropy_ineq} for~$v=v(x,t)$
with~$k=u(y,s)$ fixed and~$\phi(x,t):=\phi^{\epsilon,\nu}(x,t,y,s)$. By
Remark~\ref{rem-def-cont} and an integration of $(y,s)$ over $Q_T$, we
find that  
\begin{equation*}
\begin{split}
& \int_{Q_{T}^2} \Big(|v-u|\,\partial_t \phi^{\epsilon,\nu}+q_g(v,u) \cdot \nabla_x \phi^{\epsilon,\nu}\Big) \, \dif w\\ 
& +\int_{Q_{T}^2} |\psi(v)-\psi(u)| \, \Levy^{\alpha}_{r}[\phi^{\epsilon,\nu}(\cdot,t,y,s)](x)  \, \dif w\\ 
&  +\int_{Q_{T}^2} \sgn (v-u) \, \Levy^{\alpha,r}[\psi(v(\cdot,t))](x) \,\phi^{\epsilon,\nu} \, \dif w\\
&-\int_{\R^d \times Q_T} |v(x,T)-u(y,s)|\,\phi^{\epsilon,\nu}(x,T,y,s) \, \dif x \, \dif y \, \dif s\\
& +\int_{\R^d \times Q_T} |v_0(x)-u(y,s)|\,\phi^{\epsilon,\nu}(x,0,y,s) \, \dif x \, \dif y \, \dif s \geq 0.
\end{split}
\end{equation*}
Inserting this inequality into the Kuznetsov inequality \eqref{kuz}, we obtain
for all~$r,\epsilon>0$ and~$T>\nu>0$, 
\begin{equation}\label{nat:cnl-1}
\begin{split}
&\|u(\cdot,T)-v(\cdot,T)\|_{L^{1}} \leq \\
& C(d) \, |u_0|_{BV} \, \epsilon+2\,(m_u(\nu) \vee m_v(\nu))\\
& \underbrace{+\int_{Q_{T}^2}(q_g-q_f)(v,u) \cdot \nabla_x  \phi^{\epsilon,\nu} \, \dif w}_{=:\E_1}\\
&\underbrace{+\int_{Q_{T}^2} \Big(|\psi(v)-\psi(u)| \, \mathcal{L}^{\alpha}_r[\phi^{\epsilon,\nu}(\cdot,t,y,s)](x)+ |\varphi(v)-\varphi(u)| \,\mathcal{L}^{\alpha}_r[\phi^{\epsilon,\nu}(x,t,\cdot,s)](y)\Big) \,
\dif w}_{=:\E_2} \\
&\underbrace{+\int_{Q_{T}^2} \sgn(v-u)\,
  \left(\Levy^{\alpha,r}[\psi(v(\cdot,t))](x)-\Levy^{\alpha,r}[\varphi(u(\cdot,s))](y)\right)
  \phi^{\epsilon,\nu} \, \dif w}_{=:\E_3}
\end{split}
\end{equation}
where~$C(d)=C_{\rho}$ from~\eqref{kuz}. During the proof,~$C(d)$ will denote various constant depending only on~$d$. 

\medskip

\noindent {\bf 3.} {\it Estimates of $\E_1$ and $\E_2$.} A standard estimate shows that
\begin{align}\label{esti-kuz}
\E_1 & \leq  T \, |u_0|_{BV} \, \|f'-g'\|_\infty,
\end{align}
see e.g. \cite{Daf72,Luc80,Daf05}. Let us estimate~$\E_2$.  By Taylor's formula, 
\begin{align*}
\rho_\epsilon(x+z)-\rho_\epsilon(x)-\nabla \rho_\epsilon(x) \cdot z
= \int_0^1 (1-\tau) \, \nabla^2 \rho_\epsilon(x+\tau \, z) \cdot z^2 \, \dif \tau
\end{align*}
for all~$x,z \in \R^d$. Since~$\rho_\epsilon \in C^\infty_c(\R^d)$, we infer that~$\Levy^\alpha_r[\rho_\epsilon] \in L^1(\R^d)$ with  
\begin{align*}
 \|\Levy^\alpha_r[\rho_\epsilon]\|_{L^1} &\leq G_d(\alpha) \int_{|z|<r}  \int_0^1  (1-\tau) \, |z|^{-d+2-\alpha} \int_{\R^d} |\nabla^2 \rho_\epsilon(x+\tau \, z)|  \, \dif x  \, \dif \tau\, \dif z \\
& = C(d,\alpha,\epsilon) \, r^{2-\alpha}.
\end{align*}
Moreover, by Definitions~\eqref{nat:Levy-form} and~\eqref{nat:test-kuznetsov},
$$
\Levy^\alpha_r[\phi^{\epsilon,\nu}(\cdot,t,y,s)](x)=\theta_\nu(t-s) \, \Levy^\alpha_r[\rho_\epsilon](x-y).
$$ 
By Fubini and the convolution like structure of the integral, it follows that 
\begin{align*}
& \int_{Q_{T}^2} |\psi(v(x,t))-\psi(u(y,s))| \, \mathcal{L}^{\alpha}_r[\phi^{\epsilon,\nu}(\cdot,y,t,s)](x) \, \dif w \\ 
& \leq \left(\|\psi(v)\|_{L^1(Q_T)}+\|\psi(u)\|_{L^1(Q_T)}\right) C(d,\alpha,\epsilon) \, r^{2-\alpha},
\end{align*} 
since $\int \theta_\nu =1$. In a similar way we can
estimate the $\varphi$-integral and conclude that
\begin{align}
\E_2 \leq C_\epsilon \, r^{2-\alpha}.   \label{nat:end2}
\end{align}
From now on $C_\epsilon$ will denote various constants depending among
other things on~$\epsilon$, but not on~$r,\nu$. For later use we note that
$\E_2\ra0$ as~$r,\nu \da 0$ and $\epsilon$ is fixed.

\medskip
\noindent {\bf 4.} {\it Estimate of $\E_3$ -- the linear case.} We
consider the case $\varphi' \equiv a$ and~$\psi' \equiv b$ for~$a,b >0$. In this case
\begin{equation}\label{lp-form}
\begin{split}
\E_3 &= G_d(\alpha) \int_{Q_{T}^2} \int_{|z|>r} \sgn (v-u) \, \phi^{\epsilon,\nu} \\
& \quad \cdot\frac{a \left(v(x+z,t)-v\right)  - b \left(u(y+z,s)-u\right)}{|z|^{d+\alpha}}
  \, \dif z \, \dif w.
  \end{split}
\end{equation}

By the change of variables $z \mt b^{\frac{1}{\alpha}} \,
z$, we see that
\begin{align*}
 b \, \Levy^{\alpha,r}[v(\cdot,t)](x) & = G_d(\alpha) \int_{|z|>r} \frac{v(x+z,t)-v(x,t)}{|b^{-\frac{1}{\alpha}} \, z|^{d+\alpha}} \, b^{-\frac{d}{\alpha}} \, \dif z \\
& = G_d(\alpha) \int_{|z|>b^{-\frac{1}{\alpha}} r} \frac{v (x+b^{\frac{1}{\alpha}} \, z ,t)-v(x,t)}{|z|^{d+\alpha}} \, \dif z, 
\end{align*}
and similarly that
\begin{equation*}
a \, \Levy^{\alpha,r}[u(\cdot,s)](y)=G_d(\alpha) \int_{|z|>a^{-\frac{1}{\alpha}} r} \frac{u (y+a^{\frac{1}{\alpha}} \, z,s )-u(y,s)}{|z|^{d+\alpha}} \, \dif z.
\end{equation*}
It follows that
\begin{equation}\label{cut-change-nonlin}
\begin{split}
\E_3  & =  G_d(\alpha)
\int_{Q_{T}^2} \int_{(a \vee b)^{-\frac{1}{\alpha}} r <|z| < (a \wedge b)^{-\frac{1}{\alpha}} r} \dots \frac{\dif z}{|z|^{d+\alpha}} \\
& \quad +G_d(\alpha) \int_{Q_{T}^2} \int_{|z|>(a \wedge b)^{-\frac{1}{\alpha}} r} \sgn(v-u) \\ 
& \quad \quad \cdot\frac{\left(v (x+b^{\frac{1}{\alpha}} \, z ,t)-u(y+a^{\frac{1}{\alpha}} \, z,s)\right)-(v-u)}{|z|^{d+\alpha}} \, \phi^{\epsilon,\nu} \, \dif z \, 
   \dif w \\
& =: \E_{3,1}+\E_{3,2},
\end{split}
\end{equation}
where~$\E_{3,1}$ contains only the~$u$-terms if~$a \geq b$, or only
the~$v$-terms in the other case. In the $u$-case, e.g.,
\begin{equation*}
\begin{split}
\E_{3,1}  &=  G_d(\alpha)
\int_{Q_{T}^2} \int_{a^{-\frac{1}{\alpha}} r <|z| < b^{-\frac{1}{\alpha}} r} \sgn(u-v) \, \frac{u (y+a^{\frac{1}{\alpha}} \, z ,s)-u}{|z|^{d+\alpha}} \, \phi^{\epsilon,\nu} \, \dif z \, 
   \dif w.
\end{split}
\end{equation*}

The estimates for $\E_{3,1}$ are similar in both cases, and we only
detail the $u$-case. As in the proof of Lemma~\ref{lem-kato}, we use that 
\begin{equation*}
\begin{split}
& \sgn(u(y,s)-v(x,t)) \, \left(u(y+a^{\frac{1}{\alpha}} \, z ,s)-u(y,s)\right)\\
& \leq \left|u (y+a^{\frac{1}{\alpha}} \, z ,s)-v(x,t) \right| -\left|u(y,s) -v(x,t) \right|,
\end{split}
\end{equation*}
to deduce that
\begin{align*}
\E_{3,1} & \leq G_d(\alpha)\int_{Q_{T}^2} \int_{a^{-\frac{1}{\alpha}} r <|z| < b^{-\frac{1}{\alpha}} r} \frac{\left|u(y+a^{\frac{1}{\alpha}} \, z ,s)-v(x,t) \right| -\left|u -v  \right|}{|z|^{d+\alpha}} \, \phi^{\epsilon,\nu}  \, \dif z \, \dif w \\ 
& = G_d(\alpha) \int_{Q_{T}^2} |u-v| \\ 
& \quad \cdot\int_{a^{-\frac{1}{\alpha}} r <|z| < b^{-\frac{1}{\alpha}} r} \underbrace{\left(\phi^{\epsilon,\nu}(x,t,y+a^{\frac{1}{\alpha}} \, z,s)-\phi^{\epsilon,\nu}\right)}_{=\theta_\nu(t-s) \left( \rho_\epsilon(x-y-a^{\frac{1}{\alpha}} \, z)-\rho_{\epsilon}(x-y) \right)} \, |z|^{-d-\alpha} \, \dif z \, \dif w.
\end{align*}
We continue as in the derivation of \eqref{nat:end2}, and use a
Taylor expansion with integral remainder of $\rho_\epsilon$. Since the first order term
contains the factor
$$
\int_{a^{-\frac{1}{\alpha}} r <|z| < b^{-\frac{1}{\alpha}} r}
\frac{z}{|z|^{d+\alpha}} \, \dif z=0, 
$$
we find an estimate
similar to \eqref{nat:end2}, namely
\begin{equation}\label{esti-reste-change-nonlin}
\E_{3,1} \leq C_\epsilon  \left(\|u\|_{L^1(Q_T)} + \|v\|_{L^1(Q_T)} \right) r^{2-\alpha}.
\end{equation}
We emphasize that $C_\epsilon$ can be chosen to be independent of $a$ and $b$
by \eqref{hnd} (more precisely $C_\epsilon = C(d,\alpha,\epsilon,\underline{\Lambda},\overline{\Lambda})$;
this will be important in the next step. 

\medskip
\noindent {\bf 5.} {\em Estimate of $\E_{3,2}$}. Note that
$a,b$ are arbitrary reals such that \eqref{hnd}
holds, i.e. $\overline{\Lambda} \geq a,b \geq \underline{\Lambda}$, and let $r_2 \geq r_1 >0$. Since 
$\underline{\Lambda}>0$ and $r$ will be sent to zero, we assume without loss of
generality that $r_1 > \underline{\Lambda}^{-\frac{1}{\alpha}} \, r$. In particular,~$r_1 > (a \wedge b)^{-\frac{1}{\alpha}} \, r$. Then
\begin{equation}\label{cut-nonline-E3}
\begin{split}
\E_{3,2} & = \sum_{i=1}^3
G_d(\alpha) \int_{Q_{T}^2} \int_{|z| \in I_i} \sgn(v-u) \\ 
 &  \quad \cdot\frac{\left(v (x+b^{\frac{1}{\alpha}} \, z ,t)-u (y+a^{\frac{1}{\alpha}} \, z,s)\right)-(v-u)}{|z|^{d+\alpha}} \, \phi^{\epsilon,\nu}\, \dif z \, 
   \dif w \\
& =:  \sum_{i=1}^3 \E_{3,2,i}, 
\end{split}
\end{equation}
where~$I_1=(r_2,+\infty)$,~$I_2=(r_1,r_2)$ and~$I_3=((a \wedge
b)^{-\frac{1}{\alpha}} \, r, r_1 )$. 

By adding and subtracting $ 
\sgn(v-u) \, u (y+b^{\frac{1}{\alpha}} \, z ,s)
$ and using Lemma~\ref{lem-kato}
with~$c=\tilde{c}=b^{\frac{1}{\alpha}}$ and~$\gamma=\tilde{\gamma}=1$,
we find that
\begin{equation*}
\begin{split}
\E_{3,2,i} & \leq G_d(\alpha) \int_{Q_{T}^2} \int_{|z| \in I_i} \sgn(v-u)  \, \frac{u (y+b^{\frac{1}{\alpha}} \, z ,s)-u (y+a^{\frac{1}{\alpha}} \, z,s)}{|z|^{d+\alpha}} \, \phi^{\epsilon,\nu} \, \dif z \, \dif w.
\end{split}
\end{equation*}
By the~$BV$-regularity of~$u$, we then immediately deduce that  
\begin{equation*}
\E_{3,2,2} \leq 
G_d(\alpha) \, |u|_{L^1(0,T;BV)} \, |a^{\frac{1}{\alpha}}-b^{\frac{1}{\alpha}}| \int_{r_1<|z|<r_2} \frac{|z|\,\dif z}{|z|^{d+\alpha}}.  
\end{equation*}
Moreover, going back to the original
variables~$a^{\frac{1}{\alpha}} \, z \mt z$ and~$b^{\frac{1}{\alpha}}
\, z\mt z$, we find that
\begin{align*}
& \int_{Q_{T}^2} \int_{|z| >r_2} \sgn(v-u) \frac{u (y+a^{\frac{1}{\alpha}} \, z ,s)}{|z|^{d+\alpha}} \, \phi^{\epsilon,\nu}\, \dif z \, \dif w\\
& =a \int_{Q_{T}^2} \int_{|z| >a^{\frac{1}{\alpha}} \, r_2} \sgn(v-u) \frac{u (y+ z ,s)}{|z|^{d+\alpha}} \, \phi^{\epsilon,\nu} \, \dif z \, \dif w,
\end{align*}
and a similar formula for the $b$-term. Hence we find that
\begin{equation*}
\begin{split}
 \E_{3,2,1} \leq \ & G_d(\alpha) \, (b-a) \int_{Q_{T}^2} \int_{|z| >(a \vee b)^{\frac{1}{\alpha}} \, r_2} \sgn(v-u) \, \frac{u(y+ z ,s)}{|z|^{d+\alpha}}  \, \phi^{\epsilon,\nu}  \, \dif z \, \dif w \\
& +G_d(\alpha) \, \sgn(a-b) \, (a \wedge b) \int_{Q_{T}^2} \int_{(a \wedge b)^{\frac{1}{\alpha}} \, r_2 < |z| <(a \vee b)^{\frac{1}{\alpha}} \, r_2} \dots,
\end{split}
\end{equation*}
where the integrands are the same. Since $\phi^{\epsilon,\nu}$ is an
approximate unit,  
\begin{equation*}
\E_{3,2,1}  \leq C(d) \, \frac{G_d(\alpha)}{\alpha} \, \|u\|_{L^1(Q_T)} \, \frac{|a-b|}{a \vee b} \, r_2^{-\alpha},
\end{equation*}
where $C(d)=2 \, S_d$.

It remains to estimate~$\E_{3,2,3}$ in~\eqref{cut-nonline-E3}. By Lemma~\ref{lem-kato}, with~$c=a^{\frac{1}{\alpha}}$ and~$\tilde{c}=b^{\frac{1}{\alpha}}$, 
\begin{equation}\label{repa-p}
\begin{split}
\E_{3,2,3} & \leq G_d(\alpha) \int_{Q_{T}^2}  \int_{(a \wedge b)^{-\frac{1}{\alpha}} \, r <|z|<r_1} |v-u| \, \theta_\nu(t-s)  \\
 & \quad \cdot \left\{ \rho_\epsilon (x-y+h(z))-\rho_\epsilon(x-y) \right\} |z|^{-d-\alpha} \, \dif z \, \dif w 
\end{split}
\end{equation}
with~$h(z):=(b^{\frac{1}{\alpha}}-a^{\frac{1}{\alpha}}) \, z$. After
a Taylor expansion of~$\rho_\epsilon$ with integral
remainder, we find that
\begin{equation*}
\begin{split}
\E_{3,2,3} & \leq G_d(\alpha)  \int_{Q_{T}^2} \int_{(a \wedge b)^{-\frac{1}{\alpha}} \, r <|z|<r_1} \int_0^1 (1-\tau) \, |v-u| \, \theta_\nu(t-s) \, |z|^{-d-\alpha}  \\
& \quad  \cdot\nabla^2 \rho_\epsilon \left(x-y + \tau \, h(z)\right) \cdot h(z)^2 \, \dif \tau \, \dif z \, \dif w.
\end{split}
\end{equation*}
Remember that the integral of the first order term in $z$ is zero by symmetry. 
By a standard argument, $|v-u|$ is~$BV$ in $y$ as composition
of a $BV$ with a Lipschitz function (cf. e.g.~\cite{BoPe98}). Hence, by an integration by parts with respect to~$y$,  
\begin{equation*}
\begin{split}
 \E_{3,2,3}   
& \leq G_d(\alpha)  \int_0^{T} \int_{Q_{T}} \int_{(a \wedge b)^{-\frac{1}{\alpha}} \, r <|z|<r_1} \int_0^1 (1-\tau) \, \theta_\nu(t-s) \, |z|^{-d-\alpha} \,  \\
& \quad \cdot\left\{\int_{\R^d} \nabla \rho_\epsilon \left(x-y + \tau \, h(z) \right) \cdot h(z) \, h(z) \cdot \dif \nabla_y |v(x,t)-u(\cdot,s)|(y)\right\}\\ 
& \quad \dif \tau \, \dif z \, \dif x \, \dif t \, \dif s.
\end{split}
\end{equation*}
We use the notation~$\dif \nabla_y |v(x,t)-u(\cdot,s)|(y)$ in
case~$\nabla_y |v-u|$ is a measure. Then $|\nabla_y |v-u|| \leq
|\nabla u|$ in the sense of measures since $y$ is the space
variable of $u$. It follows that 
\begin{equation*}
\begin{split}
 \E_{3,2,3} & \leq  G_d(\alpha) \int_0^{T} \int_{Q_{T}} \int_{|z|<r_1}  \int_0^1 (1-\tau) \, \theta_\nu(t-s) \, |z|^{-d-\alpha} \, |h(z)|^2\\
& \quad \cdot\left\{\int_{\R^d} \left|\nabla \rho_\epsilon \left(x-y + \tau \, h(z) \right) \right|  \dif |\nabla u(\cdot,s)|(y) \right\} \dif \tau \, \dif z \, \dif x \, \dif t \, \dif s. 
\end{split}
\end{equation*}
By Fubini\footnote{applied for fixed~$s$, so that~$\dif
  |\nabla u(\cdot,s)|(y) \, \dif z \, \dif x \, \dif t$ 
is a tensor product of~$\sigma$-finite measures!} we integrate with
respect to~$(x,t)$ before~$(y,s)$, and then we use that
$h(z)=(b^{\frac{1}{\alpha}}-a^{\frac{1}{\alpha}}) \, z$ and $\int
|\nabla \rho_\epsilon| =\frac{1}{\epsilon} \int |\nabla \rho| =
\frac{C(d)}{\epsilon}$ (by \eqref{nat:test-kuznetsov}), to see that
\begin{equation}\label{rep-p-b}
\begin{split}
 \E_{3,2,3} & \leq G_d(\alpha) \,  \int_0^{T} 
 \int_{|z|<r_1} \int_0^1 (1-\tau) \, |z|^{-d-\alpha} \\
& \quad  \cdot |h(z)|^2 \, |u(\cdot,s)|_{BV} \,  \dif \tau \, \dif z
\, \dif s\,\int_{Q_T} \theta_\nu|\nabla
\rho_\epsilon|\,\dif x\,\dif t\\
&\leq  C(d) \, \frac{G_d(\alpha)}{2-\alpha} \, |u|_{L^1(0,T;BV)} \, (a^{\frac{1}{\alpha}}-b^{\frac{1}{\alpha}})^2 \, \frac{r_1^{2-\alpha}}{\epsilon}.
\end{split}
\end{equation}

\medskip
\noindent {\bf 6.} {\em Estimate of $\E_{3}$ -- conclusion in the linear case}.
By the estimates of {\bf 4} and {\bf 5}, \eqref{cut-change-nonlin},
\eqref{cut-nonline-E3}, etc., we can then conclude that
\begin{equation}
\label{esti-diff-bc}
\begin{split}
\E_3 & \leq \E_{3,1}+\E_{3,2,1}+\E_{3,2,2}+\E_{3,2,3}\\
&\leq C_\epsilon \left(\|u\|_{L^1(Q_T)} + \|v\|_{L^1(Q_T)} \right) r^{2-\alpha}\\
& \quad + C(d) \, G_d(\alpha) \, \bigg\{\frac{1}{\alpha} \, \|u\|_{L^1(Q_T)} \, \frac{|a-b|}{a \vee b} \, r_2^{-\alpha}\\
& \quad + |u|_{L^1(0,T;BV)} \, |a^{\frac{1}{\alpha}}-b^{\frac{1}{\alpha}}| \int_{r_1<|z|<r_2} \frac{|z|\,\dif z}{|z|^{d+\alpha}} \\
& \quad + \frac{1}{2-\alpha} \, |u|_{L^1(0,T;BV)} \, (a^{\frac{1}{\alpha}}-b^{\frac{1}{\alpha}})^2 \, \frac{r_1^{2-\alpha}}{\epsilon}\bigg\},
\end{split}
\end{equation}
for arbitrary $r_2 \geq r_1 > \underline{\Lambda}^{-\frac{1}{\alpha}} \, r$. 
Note that the $\frac{1}{a \vee b}$-term has to be handled with care
since it could be large in the general case when $\varphi'$ and
$\psi'$ can be degenerate. 

We conclude the estimate of $\E_3$ by choosing the values of constants $r_1$
and $r_2$. In the critical case where $\alpha=1$, we take $r_1=T
\wedge 1$ and~$r_2=1 \vee \frac{\|u_0\|_{L^1}}{(a \vee b) \, |u_0|_{BV}}$. Notice that if $|u_0|_{BV}=0$, then $u_0 \equiv 0$ as constant integrable function, and \eqref{main-esti-first} reduces to \eqref{nat:nonincrease}. In the sequel, we thus assume without loss of generality that $|u_0|_{BV} \neq 0$. Note then that
$+\infty>r_2\geq r_1 =T \wedge 1 > \underline{\Lambda}^{-\frac{1}{\alpha}} \, r$ for $r$
small enough ($r\da0$ in the end). 
By easy computation and Lemma~\ref{lem-tech-nl} of the Appendix, 
\begin{align*}
& |a-b| \int_{r_1<|z|<r_2} \frac{|z|\,\dif z}{|z|^{d+1}} = C \, |a-b| \, (\ln r_2-\ln r_1)\\
& \leq C \,|a-b| \left\{|\ln T|+\mathbf{1}_{\frac{\|u_0\|_{L^1}}{|u_0|_{BV}}>1} \, \ln \frac{\|u_0\|_{L^1}}{|u_0|_{BV}}+\left(-\ln (a \vee b)\right)^+ \right\}\\
& \leq C \, \Big\{\Big(1 + |\ln T| +|u_0|_{BV}^{-1} \, \Ent_1(u_0) \Big) \, |a-b|+|a \, \ln a- b \, \ln b|\Big\},
\end{align*}
where~$C=C(d)$ and where $\Ent_1(u_0)$ is defined in \eqref{log-entrop}. We finally deduce from~\eqref{esti-diff-bc} that, when~$\alpha=1$,
\begin{equation}\label{big-modif-2}
\begin{split}
\E_3 & \leq C_\epsilon \left(\|u\|_{L^1(Q_T)} + \|v\|_{L^1(Q_T)} \right)  r\\
& \quad +  C(d) \, \Big\{
\|u\|_{L^1(Q_T)} \, \frac{|u_0|_{BV}}{\|u_0\|_{L^1}} \,|a-b|\\
& \quad + \left(1 + |\ln T| +|u_0|_{BV}^{-1} \, \Ent_1(u_0)\right) |u|_{L^1(0,T;BV)} \,|a-b|\\
& \quad +|u|_{L^1(0,T;BV)} \, |a \, \ln a- b \, \ln b|\\
& \quad + T \, |u|_{L^1(0,T;BV)} \, (a-b)^2 \, \frac{1}{\epsilon} \Big\},
\end{split}
\end{equation}
for all~$T \wedge 1 > \underline{\Lambda}^{-1} \, r$. To divide by $\|u_0\|_{L^1}$,  we have assumed without loss of generality that we are not in the case where 
$\|u_0\|_{L^1} = 0$, for which \eqref{main-esti-first} also reduces to \eqref{nat:nonincrease}.

When~$\alpha>1$, we simply choose~$r_2=+\infty$ in~\eqref{esti-diff-bc} and we get
\begin{equation}\label{big-modif-1}
\begin{split}
\E_3 & \leq C_\epsilon \left(\|u\|_{L^1(Q_T)} + \|v\|_{L^1(Q_T)} \right)  r^{2-\alpha}\\
& \quad +C(d) \,G_d(\alpha)\, \Big\{ \frac{1}{\alpha-1} \, |u|_{L^1(0,T;BV)} \, |a^{\frac{1}{\alpha}}-b^{\frac{1}{\alpha}}| \, r_1^{1-\alpha}\\
& \quad + \frac{1}{2-\alpha} \, |u|_{L^1(0,T;BV)} \, (a^{\frac{1}{\alpha}}-b^{\frac{1}{\alpha}})^2 \, \frac{r_1^{2-\alpha}}{\epsilon} \Big\},
\end{split}
\end{equation}
for all~$r_1>\underline{\Lambda}^{-\frac{1}{\alpha}} \, r$.

\medskip

\noindent {\bf 7.} {\it Estimate of $\E_3$ - the general case via
  linearization.}   
The idea is now to reduce to the linear case in step {\bf 4} by
freezing the ``diffusion coefficients'' $\varphi'(\xi)$
and~$\psi'(\xi)$. To do so, we introduce the function
\begin{equation}\label{si}
\chi_{a}^{b}(\xi):=\sgn (b-a) \, \mathbf{1}_{(a \wedge b,a \vee b)}(\xi),
\end{equation}
for $\xi,a,b \in \R$.  By \eqref{nat:cnl-1}, we then find that
\begin{equation}\label{e3-lp}
\begin{split}
\E_{3} & = G_d(\alpha) \int_{Q_{T}^2} \int_{|z|>r} \sgn(v-u)  \\
  & \quad \cdot\frac{\int_{v(x,t)}^{v(x+z,t)} \psi'(\xi) \, \dif \xi -\int_{u(y,s)}^{u(y+z,s)} \varphi'(\xi) \, \dif \xi}{|z|^{d+\alpha}}
  \, \phi^{\epsilon,\nu} \, \dif z \, \dif w\\
& = G_d(\alpha) \int_{Q_{T}^2} \int_{|z|>r} \int \sgn(v-u)  \\
  & \quad \cdot\frac{\chi_{v(x,t)}^{v(x+z,t)}(\xi) \, \psi'(\xi)  -\chi_{u(y,s)}^{u(y+z,s)}(\xi) \, \varphi'(\xi)}{|z|^{d+\alpha}}
  \, \phi^{\epsilon,\nu} \, \dif \xi \, \dif z \, \dif w.
\end{split}
\end{equation}
Let us notice that this integral is well-defined, since e.g. $\int |\chi_a^b (\xi)| \, \dif \xi = |b-a|$ and, $\varphi'$ and $\psi'$ are assumed bounded by \eqref{hnd}. 

For each~$\delta>0$, we define a regularized version of $\E_3$ as
\begin{equation}\label{def-a-lp}
\begin{split}
\E_{3}(\delta) & := G_d(\alpha) \int_{Q_{T}^2} \int_{|z|>r} \int \int \sgn(v-u)  \\
  & \quad \cdot \frac{\chi_{v(x,t)}^{v(x+z,t)}(\zeta) \, \psi'(\xi)  -\chi_{u(y,s)}^{u(y+z,s)}(\zeta) \, \varphi'(\xi)}{|z|^{d+\alpha}}
  \, \phi^{\epsilon,\nu} \, \omega_\delta(\xi-\zeta) \, \dif \zeta \, \dif \xi \, \dif z \, \dif w,
\end{split}
\end{equation}
where the approximate unit~$\omega_\delta(\xi):=\frac{1}{\delta} \, \omega
\left( \frac{\xi}{\delta} \right)$, and
\begin{equation*}
\omega \in C_b^\infty \cap L^1(\R), \quad \omega >0, \quad \int \omega =1. 
\end{equation*}
For each $\zeta,\xi \in \R$, let
$\Omega_\xi(\zeta):=\int_{-\infty}^\zeta\omega_\delta(\xi-w)\,\dif
w-\int_{-\infty}^0\omega_\delta(\xi-w)\,\dif w$, \label{po} and note that
$$
\int \chi_{v(x,t)}^{v(x+z,t)}(\zeta) \, \omega_\delta(\xi-\zeta) \, \dif \zeta=\int_{v(x,t)}^{v(x+z,t)} \Omega_\xi'(\zeta) \, \dif \zeta=\Omega_\xi(v(x+z,t))-\Omega_\xi(v(x,t)).
$$
Moreover, $\sgn (v-u)=\sgn \left(\Omega_\xi(v)-\Omega_\xi(u)\right)
$ since~$\Omega_\xi(\cdot)$ is increasing, and since~$\Omega_\xi(\cdot)$ is smooth and
vanishes at zero, $\Omega_\xi(u)$ and $\Omega_\xi(v)$ have similar
boundedness, integrability, and regularity properties as~$u$
and~$v$. It follows that 
\begin{equation*}
\begin{split}
& \E_{3}(\delta)\\
& =  G_d(\alpha) \int \int_{Q_{T}^2} \int_{|z|>r} \sgn \left(\Omega_\xi(v)-\Omega_\xi(u)\right) \phi^{\epsilon,\nu} \\
  & \quad \cdot\frac{\psi'(\xi) \left(\Omega_\xi(v(x+z,t))-\Omega_\xi(v)\right)  - \varphi'(\xi) \left(\Omega_\xi(u(y+z,s))-\Omega_\xi(u)\right)}{|z|^{d+\alpha}}
  \, \dif z \, \dif w \, \dif \xi.
\end{split}
\end{equation*}
This integrand has similar form and properties as the one
in~\eqref{lp-form} for fixed $\xi$!

We continue in the critical case when $\alpha=1$. We argue as in step
{\bf 4} with $a=\varphi'(\xi)$ and
$b=\psi'(\xi)$. By \eqref{big-modif-2} we get that for all $T \wedge 
1> \underline{\Lambda}^{-1} \, r$,  
\begin{equation*}
\begin{split}
\E_3(\delta) & \leq  \int C_\epsilon \left(\|\Omega_\xi(u)\|_{L^1(Q_T)} + \|\Omega_\xi(v)\|_{L^1(Q_T)} \right) r \, \dif \xi\\
& \quad + C(d) \, \int \bigg\{\|\Omega_\xi(u)\|_{L^1(Q_T)} \, \frac{|u_0|_{BV}}{\|u_0\|_{L^1}} \,|\varphi'(\xi)-\psi'(\xi)|\\
& \quad + \left(1 + |\ln T| +|u_0|_{BV}^{-1} \, \Ent_1(u_0)\right) |\Omega_\xi(u)|_{L^1(0,T;BV)} \,|\varphi'(\xi)-\psi'(\xi)|\\
& \quad +|\Omega_\xi(u)|_{L^1(0,T;BV)} \, |\varphi'(\xi) \, \ln \varphi'(\xi)- \psi'(\xi) \, \ln \psi'(\xi)|\\
& \quad + T \, |\Omega_\xi(u)|_{L^1(0,T;BV)} \, (\varphi'(\xi)-\psi'(\xi))^2 \, \frac{1}{\epsilon} \bigg\} \, \dif \xi\\
& \leq  C_\epsilon \, r  \int \|\Omega_\xi(u)\|_{L^1(Q_T)} + \|\Omega_\xi(v)\|_{L^1(Q_T)} \, \dif \xi \\
& \quad +C(d) \, \Big\{A \, \frac{|u_0|_{BV}}{\|u_0\|_{L^1}}\,   \|\varphi'- \psi'\|_\infty \\ 
& \quad + \left(1 + |\ln T| +|u_0|_{BV}^{-1} \, \Ent_1(u_0)\right) B \, \|\varphi'- \psi'\|_\infty \\
& \quad +B \, \|\varphi' \, \ln \varphi'- \psi' \, \ln \psi'\|_\infty  \\
& \quad +T \, B \,  \|\varphi'- \psi'\|_\infty^2 \, \frac{1}{\epsilon} \Big\},
\end{split}
\end{equation*}
with~$A = \int \|\Omega_\xi(u)\|_{L^1(Q_T)} \, \dif \xi$,~$B = \int |\Omega_\xi(u)|_{L^1(0,T;BV)} \, \dif \xi$, and
$$
\|\varphi'-\psi'\|_{\infty}=\esssup_{I(u_0)} |\varphi'-\psi'|.
$$
The supremum above can be taken only on $I(u_0)$, since $\varphi'$ and $\psi'$ are assumed to vanish outside this interval by \eqref{hnd}.
Note also that $C_\epsilon=C(d,\alpha,\epsilon,\underline{\Lambda},\overline{\Lambda})$ can be chosen independent of
$\varphi'(\xi)$ and~$\psi'(\xi)$ as discussed below
\eqref{esti-reste-change-nonlin}. 
A standard argument, see Appendix
\ref{app-app}, then reveals that
\begin{eqnarray}
\int \|\Omega_\xi(u)\|_{L^1(Q_T)} \, \dif \xi & = & \|u\|_{L^1(Q_T)},\label{YL}\\
\int |\Omega_\xi(u)|_{L^1(0,T;BV)} \, \dif \xi & = & |u|_{L^1(0,T;BV)}\label{YBV},
\end{eqnarray} 
and hence that $A \leq T \,
\|u_0\|_{L^1}$ and~$B \leq T \, |u_0|_{BV}$ by
\eqref{nat:nonincrease}. 

By standard computations given in Appendix \ref{app-app},
\begin{equation}\label{YC}
\lim_{\delta \da 0} \E_3(\delta)=\E_3, 
\end{equation}
and it follows after going to the limit in the estimate above, that
\begin{equation}\label{esti-diff-t-cc}
\begin{split}
\E_3 & \leq C_\epsilon \,  r\\
& \quad +C(d) \, \Big\{T \, \Ent_1(u_0) \, \|\varphi'- \psi'\|_\infty\\
& \quad + T \, (1 + |\ln T|) \, |u_0|_{BV} \,\|\varphi'- \psi'\|_\infty\\
& \quad + T \, |u_0|_{BV} \, \|\varphi' \, \ln \varphi'- \psi' \, \ln \psi'\|_\infty\\
& \quad + T^2 \, |u_0|_{BV} \, \|\varphi'- \psi'\|_\infty^2 \, \frac{1}{\epsilon}\Big\},
\end{split}
\end{equation}
for all~$T \wedge 1> \underline{\Lambda}^{-1} \, r$ when $\alp=1$.

When~$\alpha>1$, similar arguments
using~\eqref{big-modif-1} show that for
all~$r_1>\underline{\Lambda}^{-\frac{1}{\alpha}} \, r$,  
\begin{equation}\label{esti-diff-t-dc}
\begin{split}
\E_3 & \leq C_\epsilon \,  r^{2-\alpha}\\
& \quad +C(d) \, \Big\{\frac{G_d(\alpha)}{\alpha-1} \, T \, |u_0|_{BV} \,  \|(\varphi')^{\frac{1}{\alpha}}-(\psi')^{\frac{1}{\alpha}}\|_\infty \, r_1^{1-\alpha}\\
& \quad + \frac{G_d(\alpha)}{2-\alpha} \,T \, |u_0|_{BV} \,  \|(\varphi')^{\frac{1}{\alpha}}-(\psi')^{\frac{1}{\alpha}}\|_\infty^2 \, \frac{r_1^{2-\alpha}}{\epsilon}\Big\}.
\end{split}
\end{equation}

\medskip 

\noindent {\bf 8.} {\it Conclusion.} We have to insert the estimates
of the three preceding steps into \eqref{nat:cnl-1}. Let us begin by
the case where~$\alpha=1$. By~\eqref{esti-kuz},~\eqref{nat:end2}
and~\eqref{esti-diff-t-cc},  
\begin{equation*}
\begin{split}
\|u(\cdot,T)-v(\cdot,T)\|_{L^{1}} & \leq 2\, (m_u(\nu) \vee m_v(\nu)) +C_\epsilon \, r \\
& \quad +T \, |u_0|_{BV} \, \|f'-g'\|_\infty\\
&\quad +C(d) \, \bigg\{|u_0|_{BV} \, \epsilon\\
& \quad +T \,  \Ent_1(u_0) \, \|\varphi'- \psi'\|_\infty\\ 
& \quad + T \left(1 + |\ln T|\right) |u_0|_{BV} \,\|\varphi'- \psi'\|_\infty\\
& \quad + T \,  |u_0|_{BV} \, \|\varphi' \, \ln \varphi'- \psi' \, \ln \psi'\|_\infty\\
& \quad + T^2 \,  |u_0|_{BV} \, \|\varphi'- \psi'\|_\infty^2 \, \frac{1}{\epsilon}\bigg\},
\end{split}
\end{equation*}
for all~$r,\epsilon>0$ and~$T>\nu>0$ such that~$T \wedge 1 >
\underline{\Lambda}^{-1} \, r$. We complete the proof by sending $r$ and $\nu$
to zero, and taking $\epsilon=T \, \|\varphi'- \psi'\|_\infty$. 

When~$\alpha>1$, we find using \eqref{esti-diff-t-dc} that 
\begin{equation*}
\begin{split}
\|u(\cdot,T)-v(\cdot,T)\|_{L^{1}} & \leq 2\,(m_u(\nu) \vee m_v(\nu)) +C_\epsilon \, r^{2-\alpha} \\
& \quad +T \, |u_0|_{BV} \, \|f'-g'\|_\infty\\
& \quad +C(d) \, \bigg\{|u_0|_{BV} \, \epsilon\\
& \quad + \frac{G_d(\alpha)}{\alpha-1} \, T \,  |u_0|_{BV} \, \|(\varphi')^{\frac{1}{\alpha}}-(\psi')^{\frac{1}{\alpha}}\|_\infty \, r_1^{1-\alpha}\\
& \quad + \frac{G_d(\alpha)}{2-\alpha} \, T \,  |u_0|_{BV} \, \|(\varphi')^{\frac{1}{\alpha}}-(\psi')^{\frac{1}{\alpha}}\|_\infty^2 \, \frac{r_1^{2-\alpha}}{\epsilon}\bigg\},
\end{split}
\end{equation*}
for all~$r,\epsilon>0$,~$T>\nu>0$ and~$r_1>\underline{\Lambda}^{-\frac{1}{\alpha}} \, r$.
We conclude by choosing $\epsilon=T^\frac{1}{\alpha} \,
\|(\varphi')^{\frac{1}{\alpha}}-(\psi')^{\frac{1}{\alpha}}\|_\infty$
and $r_1=T^\frac{1}{\alpha}$. The proof of Theorem~\ref{th:nonlin} is complete.
\end{proof}

\begin{remark}\label{rem-l2}
\begin{enumerate}
\item \label{item-rem-l2} From the proof, we find that $C \leq C(d) \left(1+ \frac{G_d(\alpha)}{\alpha-1}+ \frac{G_d(\alpha)}{2-\alpha}\right)$ in \eqref{main-esti-first} when $\alpha>1$. By \eqref{properties-Gdalpha},  $\lim_{\alpha \ua 2} C(d,\alpha)$ is finite and only depends on $d$.
\item \label{rem-ec-2} In particular, $C \leq
  C(d) \left(1+ \frac{G_d(\alpha)}{\alpha-1}+ \frac{G_d(\alpha)}{2-\alpha}\right)$
 when~$\alpha>1$ also  in~\eqref{esti-time}.  
\end{enumerate}
\end{remark}

\subsection{Proof of Theorem~\ref{nat:th-Levy}}
Here no linearization procedure is needed since $\varphi=\psi$. The new
difficulty comes from the fact that the two L\'evy
measures are different. A key idea is to change variables to work with only
one measure.  

\begin{proof}[Proof of Theorem~\ref{nat:th-Levy}] We argue as in the preceding proof with $u=u^\alpha$ and $v=u^\beta$, i.e. $(u_0,f,\varphi)=(v_0,g,\psi)$. To simplify references to similar computations, we still use the letters $u$ and $v$ for a while.

\smallskip

\noindent {\bf 1.} {\it Applying Kuznetsov, initial estimates.}
As in step {\bf 1} in the proof of Theorem~\ref{th:nonlin}, we apply
Lemma \ref{lem:kuznetsov} and estimate the $\mathcal
L_r$-terms. We obtain estimates similar to \eqref{nat:cnl-1},
and \eqref{nat:end2}, and conclude that for
all $\alpha,\beta \in (0,2)$, $r,\epsilon>0$ and $T>\nu>0$, 
\begin{equation}\label{bilan-ag-fp}
\begin{split}
& \|u(\cdot,T)-v(\cdot,T)\|_{L^{1}}\\
& \leq  C(d) \, |u_0|_{BV} \, \epsilon+2\,(m_{u}(\nu) \vee m_{v}(\nu)) +C_\epsilon \,  (r^{2-\alpha}+r^{2-\beta})\\
& \quad +\underbrace{\int_{Q_{T}^2}\sgn(v-u)\,
  \left(\Levy^{\beta,r}[\varphi(v(\cdot,t))](x)-\Levy^{\alpha,r}[\varphi(u(\cdot,s))](y)\right)
  \phi^{\epsilon,\nu}\, \dif w}_{=:\E_3}.
\end{split}
\end{equation} 
The new $r^{2-\beta}$-term comes from the new $\mathcal
L_r^\beta$-term in the estimate corresponding to $\E_2$. Note that the
terms in $\E_3$ only involve one function $\varphi$, but different
 $\alp,\beta$. Most of the remaining proof consists in estimating $\E_3$.

\medskip

\noindent {\bf 2.} {\it Change of variables and first estimate of $\E_3$.} We perform several
changes of variables to move the differences between $\mathcal L^{\alp,r}$ and
$\mathcal L^{\beta,r}$ from the L\'evy measure to the
$z$-translations. This is similar in spirit to what we did in the
preceding proof to obtain \eqref{cut-change-nonlin}. First we let
$\tilde{z}=  |z|^{\gamma^{-1}-1} \, z$ ($\gamma>0$), and note
that~$\dif \tilde z=\gamma^{-1}  |z|^{d \, (\gamma^{-1}-1)} \, \dif
z$\footnote{Indeed, $\dif \tilde z=F(z) \, \dif z$ for
  $F(z)=|\mbox{det} \, (D\,(|z|^{\gamma^{-1}-1} \, z))|$ and
$$
 D\,(|z|^{\gamma^{-1}-1} \, z)  = (\gamma^{-1}-1) \, |z|^{\gamma^{-1}-3} \, z \otimes z+|z|^{\gamma^{-1}-1} \, \mbox{Id}.
$$
Hence $F$ is positive, $F(\lambda \, z)=|\lambda|^{d \,
  (\gamma^{-1}-1)} \, F(z)$ for all~$\lambda \in \R$, and radial since
$$
F(R \, e)=\left|\mbox{det} \, \left((\gamma^{-1}-1) \, R \, e \, (R\,e)^t+R \, R^t\right)\right|=\left|\mbox{det} \, \left(R \, ((\gamma^{-1}-1) \, e \, e^t+\mbox{Id}) \, R^t\right)\right|=\gamma^{-1},
$$
for all orthogonal matrices $R\in\R^{d \times d}$ and column vectors $e$ of the canonical basis.
}
so that
$|z|^{-d-\beta} \, \dif z=\gamma \,  
|\tilde{z}|^{-d-\beta \, \gamma} \, \dif \tilde{z}$. 
Take
$\gamma=\gamma_\beta:=\sqrt{\frac{\alpha}{\beta}}$, and check that $-d-\beta \, \gamma=-d-\sqrt{\alpha \,
  \beta}$ and 
\begin{equation*}
\Levy^{\beta,r}[\varphi(v(\cdot,t))](x) = G_d(\beta) \, \gamma_\beta \int_{|z|>r^{\gamma_\beta^{-1}}} \frac{\varphi \left(v (x+ |z|^{\gamma_\beta-1} \, z ,t)\right)-\varphi(v(x,t))}{|z|^{d+\sqrt{\alpha \, \beta}}} \, \dif z.
\end{equation*}
Then we use the change of variable~$z
\mt \left(G_d(\beta)
  \gamma_\beta\right)^{\frac{1}{\sqrt{\alpha  \beta}}}  z$ and
get that
\begin{equation*}
\Levy^{\beta,r}[\varphi(v(\cdot,t))](x)= \int_{|z|>r_\beta} \frac{\varphi \left(v (x+ 
c_\beta \, |z|^{\gamma_\beta-1} \, z ,t)\right)-\varphi(v(x,t))}{|z|^{d+\sqrt{\alpha \, \beta}}} \, \dif z, 
\end{equation*}
where~$c_\beta:=\left(G_d(\beta) \,
  \gamma_\beta\right)^\frac{1}{\beta}>0$
and~$r_\beta:=\left(G_d(\beta) \,
  \gamma_\beta\right)^{-\frac{1}{\sqrt{\alpha \, \beta}}}
r^{\gamma_\beta^{-1}}>0.$ Similar computations for~$u$ show that
\begin{equation*}
\Levy^{\alpha,r}[\varphi(u(\cdot,s))](y)= \int_{|z|>r_\alpha} \frac{\varphi (u \left(y+ 
c_\alpha \, |z|^{\gamma_\alpha-1} \, z ,s))-\varphi(u(y,s)\right)}{|z|^{d+\sqrt{\alpha \, \beta}}} \, \dif z,
\end{equation*}
where~$\gamma_\alpha :=\sqrt{\frac{\beta}{\alpha}}, \quad
c_\alpha :=\left(G_d(\alpha) \, \gamma_\alpha\right)^\frac{1}{\alpha}$
and~$r_\alpha :=\left(G_d(\alpha) \,
  \gamma_\alpha\right)^{-\frac{1}{\sqrt{\alpha \, \beta}}}
r^{\gamma_\alpha^{-1}}$. Hence 
\begin{equation*}
\begin{split}
\E_3  & =  
\int_{Q_{T}^2} \int_{r_\alpha \wedge r_\beta <|z| < r_\alpha \vee r_\beta} \dots \frac{\dif z}{|z|^{d+\sqrt{\alpha \, \beta}}}\\
& \quad +\int_{Q_{T}^2} \int_{|z|>r_\alpha \vee r_\beta} \sgn(v-u) \\ 
& \quad \quad \cdot\frac{\left\{\varphi \left(v (x+ 
c_\beta \, |z|^{\gamma_\beta-1} \, z ,t)\right)-\varphi \left(u (y+ 
c_\alpha \, |z|^{\gamma_\alpha-1} \, z ,s)\right)\right\}-\{\varphi(v)-\varphi(u)\}}{|z|^{d+\sqrt{\alpha \, \beta}}}\\
& \quad \quad \cdot\phi^{\epsilon,\nu} \, \dif z \, 
   \dif w\\
&=: \E_{3,1}+\E_{3,2},
\end{split}
\end{equation*}
where the integrand of $\E_{3,1}$ only contains either $u$-terms
or $v$-terms. As in the preceding
proof, cf. \eqref{cut-change-nonlin} and
\eqref{esti-reste-change-nonlin}, we find that
\begin{equation}\label{esti-reste-change-p}
\E_{3,1} \leq C_\epsilon {\co}_r(1),
\end{equation}
where ${\co}_r(1)=\max_{\sigma=\alpha,\beta} (r_\alpha \vee r_\beta)^{2 \, \gamma_\sigma-\sqrt{\alpha \, \beta}} \ra 0$ as $r \da 0$ and $\alpha,\beta$ are fixed.

Most of the remaining proof consists in estimating $\E_{3,2}$. Before continuing, let us list the following properties that will be needed: for any $d \in \mathbb{N}$ and $\lambda \in (0,2),$
\begin{equation}
\label{properties-change}
\begin{cases}
\lim_{\alpha,\beta \ra \lambda} \gamma_\alpha=\lim_{\alpha,\beta \ra \lambda} \gamma_\beta=1,\\
\lim_{\alpha,\beta \ra \lambda} c_\alpha=\lim_{\alpha,\beta \ra \lambda} c_\beta=G_d(\lambda)^{\frac{1}{\lambda}}>0,\\
\lim_{\alpha,\beta \ra \lambda} \frac{|\gamma_\alpha-\gamma_\beta|}{|\alpha-\beta|} = \frac{1}{\lambda},\\
\limsup_{\alpha,\beta \ra \lambda} \frac{|c_\alpha-c_\beta|}{|\alpha-\beta|}<+\infty.
\end{cases}
\end{equation} 
In particular, the limsup is a constant of the form $C=C(d,\lambda)$ (note also that this limsup is in fact a limit but this is will not be needed). These properties are immediate consequences of~\eqref{properties-Gdalpha}.

\medskip

\noindent {\bf 3.} {\it First estimate of $\E_{3,2}$.} We introduce
parameters $r_2 \geq r_1 >0$. 
Notice that
$r_1 > r_\alpha \vee r_\beta$ for sufficiently small $r$ ($r \da 0$ in the next step). Let us define
\begin{equation*}
\begin{split}
&\E_{3,2}  =  \sum_{i=1}^3 \E_{3,2,i} := \sum_{i=1}^3
\int_{Q_{T}^2} \int_{|z| \in I_i} \sgn(v-u)\\ 
&  \cdot\frac{\left\{\varphi \left(v (x+ 
c_\beta \, |z|^{\gamma_\beta-1} \, z ,t)\right)-\varphi \left(u (y+ 
c_\alpha \, |z|^{\gamma_\alpha-1} \, z ,s)\right)\right\}-\{\varphi(v)-\varphi(u)\}}{|z|^{d+\sqrt{\alpha \, \beta}}}\\ &  \cdot\phi^{\epsilon,\nu} \, \dif z \, 
   \dif w  
\end{split}
\end{equation*}
for $I_1=(r_2,+\infty)$,~$I_2=(r_1,r_2)$ and~$I_3=(r_\alpha \vee r_\beta, r_1)$. 
An application of Lemma~\ref{lem-kato} with~$c=\tilde c=c_\beta$ and
$\gamma=\tilde\gamma=\gamma_\beta$, shows that
\begin{equation}
\label{tech-nl-1-pp}
\begin{split}
\E_{3,2,i} & \leq \int_{Q_{T}^2} \int_{|z| \in I_i} \sgn(v-u) \, \phi^{\epsilon,\nu}  \\
& \quad \cdot\frac{\varphi \left(u (y+ 
c_\beta \, |z|^{\gamma_\beta-1} \, z ,s)\right)-\varphi \left(u (y+ 
c_\alpha \, |z|^{\gamma_\alpha-1} \, z ,s)\right)}{|z|^{d+\sqrt{\alpha \, \beta}}}  \, \dif z \, \dif w.
\end{split}
\end{equation}
We now estimate these terms. 

Let us begin with $\E_{3,2,1}$. Going back to the original variables, $c_\alpha \,
|z|^{\gamma_\alpha-1} z \mt z$,
\begin{align*}
& \int_{Q_{T}^2} \int_{|z| >r_2} \sgn(v-u) \frac{\varphi \left(u (y+ 
c_\alpha \, |z|^{\gamma_\alpha-1} \, z ,s)\right)}{|z|^{d+\sqrt{\alpha \, \beta}}} \, \phi^{\epsilon,\nu}\, \dif z \, \dif w\\
& =G_d(\alpha) \int_{Q_{T}^2} \int_{|z| >c_\alpha \, r_2^{\gamma_\alpha}} \sgn(v-u) \frac{\varphi \left(u (y+ z ,s)\right)}{|z|^{d+\alpha}} \, \phi^{\epsilon,\nu} \, \dif z \, \dif w.
\end{align*}
Let us continue by assuming that $c_\alpha \, r_2^{\gamma_\alpha} \geq c_\beta \, r_2^{\gamma_\beta}$. 
By the above identity and a similar one for the
$\beta$-term, we then find that
\begin{equation*}
\begin{split}
& \E_{3,2,1}  \\
& \leq \int_{Q_{T}^2} \int_{|z| >c_\alpha \, r_2^{\gamma_\alpha}} \sgn(v-u) \, \varphi(u(y+ z ,s)) \, \phi^{\epsilon,\nu}  
\left(\frac{G_d(\beta)}{|z|^{d+\beta}}- \frac{G_d(\alpha)}{|z|^{d+\alpha}}\right) \dif z \, \dif w  \\
& \quad +G_d(\beta) \int_{Q_{T}^2} \int_{c_\beta \, r_2^{\gamma_\beta} < |z| < c_\alpha \, r_2^{\gamma_\alpha}}  \sgn(v-u) \, \varphi(u(y+ z ,s)) \, \phi^{\epsilon,\nu} \, \frac{\dif z\, \dif w}{|z|^{d+\beta}}.
\end{split}
\end{equation*}
By \eqref{Aflux} and \eqref{nat:nonincrease}, $
\|\varphi(u)\|_{L^1(Q_T)} \leq M \, \|u_0\|_{L^1}
$ for $M=T \, \esssup_{I(u_0)} |\varphi'|$, and then by Fubini,
\begin{equation*}
\begin{split}
& \E_{3,2,1} \\
& \leq  M \, \|u_0\|_{L^1} \left\{ \int_{|z| >c_\alpha \, r_2^{\gamma_\alpha}} \left|\frac{G_d(\beta)}{|z|^{d+\beta}}- \frac{G_d(\alpha)}{|z|^{d+\alpha}}\right| \dif z+ G_d(\beta) \int_{c_\beta \, r_2^{\gamma_\beta} < |z| < c_\alpha \, r_2^{\gamma_\alpha}}  \frac{\dif z}{|z|^{d+\beta}}
 \right\}.
\end{split}
\end{equation*}
Doing the same reasoning when $c_\alpha \, r_2^{\gamma_\alpha} < c_\beta \, r_2^{\gamma_\beta}$ and taking the maximum, we finally get
\begin{equation}
\label{vlc-tech1}
\begin{split}
\E_{3,2,1} & \leq  M \, \|u_0\|_{L^1} \int_{|z| >(c_\alpha \, r_2^{\gamma_\alpha}) \wedge (c_\beta \, r_2^{\gamma_\beta})} \left|\frac{G_d(\beta)}{|z|^{d+\beta}}- \frac{G_d(\alpha)}{|z|^{d+\alpha}}\right| \dif z\\
& \quad + C(d) \, M \, \|u_0\|_{L^1} \max_{\sigma=\alpha,\beta} \int_{|z| \in {\scriptstyle \text{co}} \{c_\alpha \, r_2^{\gamma_\alpha},c_\beta \, r_2^{\gamma_\beta}\}}  \frac{\dif z}{|z|^{d+\sigma}},
\end{split}
\end{equation}
where $C(d)=\max_{[0,2]} G_d$ is finite by \eqref{properties-Gdalpha} and from now on $\mbox{co}  \{a,b\}$ designs the interval $(a \wedge b ,a \vee b)$.

Next, by~\eqref{Aflux} 
and~\eqref{nat:nonincrease},~$|\varphi(u)|_{L^1(0,T;BV)} \leq M \,
|u_0|_{BV}$. Hence by integrating first with respect to $y$ in \eqref{tech-nl-1-pp}, we find that
\begin{align}\label{tech-nl-1-p}
\E_{3,2,2} \leq M \, |u_0|_{BV} \int_{r_1<|z|<r_2} \left|c_\alpha \, |z|^{\gamma_\alpha}-c_\beta \, |z|^{\gamma_\beta} \right| \frac{\dif z}{|z|^{d+\sqrt{\alpha \, \beta}}}.
\end{align}

Finally, by Lemma~\ref{lem-kato} 
\begin{equation*}
\begin{split}
\E_{3,2,3} & \leq \int_{Q_{T}^2}  \int_{r_\alpha \vee r_\beta <|z|<r_1} |\varphi(v)-\varphi(u)| \, \theta_\nu(t-s)  \\
 & \quad \cdot\left\{ \rho_\epsilon (x-y+h(z))-\rho_\epsilon(x-y) \right\} \frac{\dif z\,\dif w}{|z|^{d+\sqrt{\alpha \, \beta}}} ,
\end{split}
\end{equation*}
with $h(z):=(c_\beta \, |z|^{\gamma_\beta-1}-c_\alpha \,
|z|^{\gamma_\alpha-1}) \, z$. This estimate is similar to
\eqref{repa-p}, but with a new displacement, new functions
$\varphi(u)$ and $\varphi(v)$, and the new power $\sqrt{\alpha \,
  \beta}$. By arguing as before, we find that
\begin{equation*}
\begin{split}
& \E_{3,2,3} \leq  \frac{C(d)}{\epsilon} \int_0^{T} 
 \int_{|z|<r_1} \int_0^1 (1-\tau) \,   |z|^{-d-\sqrt{\alpha \, \beta}}  \,   |h(z)|^2 \, |\varphi(u(\cdot,s))|_{BV} \, \dif \tau \, \dif z \, \dif s,  
\end{split}
\end{equation*}
instead of~\eqref{rep-p-b}. Since~$|\varphi(u)|_{L^1(0,T;BV)} \leq M \, |u_0|_{BV}$, we get that
\begin{equation}\label{vlc-tech2}
\begin{split}
\E_{3,2,3} \leq C(d) \, M \, |u_0|_{BV} \, \frac{1}{\epsilon} \int_{|z|<r_1} \left|
 c_\beta \, |z|^{\gamma_\beta}-c_\alpha \,
|z|^{\gamma_\alpha}\right|^2 \frac{\dif z}{|z|^{d+\sqrt{\alp \, \beta}}}.
\end{split}
\end{equation}

\medskip

\noindent {\bf 4.} {\it The general estimate.} 
Let us resume the preceding estimates. By \eqref{bilan-ag-fp}, \eqref{esti-reste-change-p}, \eqref{vlc-tech1}, \eqref{tech-nl-1-p}, \eqref{vlc-tech2} and the fact that $\E_{3}=\E_{3,1}+\E_{3,2,1}+\E_{3,2,2}+\E_{3,2,3}$, we have proved that for all
$\alpha,\beta \in (0,2)$, $\epsilon>0$, $T>\nu>0$, $r_2 \geq r_1>0$ and $r >0$ small enough,  
\begin{equation*}
\begin{split}
& \|u^\alpha(\cdot,T)-u^\beta(\cdot,T)\|_{L^1} \\
& \leq 2 \, \left(m_u (\nu) \vee m_v(\nu) \right)+C_\epsilon \, (r^{2-\alpha}+r^{2-\beta}+{\co}_r(1))\\
& \quad +C(d) \, |u_0|_{BV} \, \epsilon\\
& \quad + M \, \|u_0\|_{L^1} \int_{|z| >(c_\alpha \, r_2^{\gamma_\alpha}) \wedge (c_\beta \, r_2^{\gamma_\beta})} \left|\frac{G_d(\beta)}{|z|^{d+\beta}}- \frac{G_d(\alpha)}{|z|^{d+\alpha}}\right| \dif z\\
& \quad +C(d) \, M \, \|u_0\|_{L^1} \max_{\sigma=\alpha,\beta} \int_{|z| \in {\scriptstyle \text{co}} \{c_\alpha \, r_2^{\gamma_\alpha},c_\beta \, r_2^{\gamma_\beta}\}}  \frac{\dif z}{|z|^{d+\sigma}}  \\
 & \quad +M \, |u_0|_{BV} \int_{r_1<|z|<r_2} \left|c_\alpha \, |z|^{\gamma_\alpha}-c_\beta \, |z|^{\gamma_\beta} \right| \frac{\dif z}{|z|^{d+\sqrt{\alpha \, \beta}}}\\
 & \quad +C(d) \, M \, |u_0|_{BV} \, \frac{1}{\epsilon} \int_{|z|<r_1} \left|
 c_\beta \, |z|^{\gamma_\beta}-c_\alpha \,
|z|^{\gamma_\alpha}\right|^2 \frac{\dif z}{|z|^{d+\sqrt{\alp \, \beta}}}.
\end{split}
\end{equation*}
Now, we pass to the limit as $r,\nu \da 0$, thanks to \eqref{esti-reste-change-p}. 
Next, we replace the $L^1$-norm at time $T$ by  the $C([0,T];L^1)$-norm, which can be done without loss of generality since $t \, \|\varphi'\|_\infty \leq T \,  \|\varphi'\|_\infty=M$, for all $t \leq T$. Finally, we replace $\epsilon$ by $\epsilon \, |\alpha-\beta|$, which can also be done since $\epsilon$ is arbitrary. We deduce that 
for all $\alpha,\beta \in (0,2)$, $\epsilon >0$, and $r_2 \geq r_1>0,$
\begin{equation}\label{vlc-general}
\begin{split}
& \|u^\alpha-u^\beta\|_{C([0,T];L^1)} \\
& \leq C(d) \, |u_0|_{BV} \, \epsilon \, |\alpha-\beta|\\
& \quad + M \, \|u_0\|_{L^1} \underbrace{\int_{|z| >(c_\alpha \, r_2^{\gamma_\alpha}) \wedge (c_\beta \, r_2^{\gamma_\beta})} \left|\frac{G_d(\beta)}{|z|^{d+\beta}}- \frac{G_d(\alpha)}{|z|^{d+\alpha}}\right| \dif z}_{=:J_1}\\
& \quad +C(d) \, M \, \|u_0\|_{L^1} \underbrace{\max_{\sigma=\alpha,\beta} \int_{|z| \in {\scriptstyle \text{co}}  \{c_\alpha \, r_2^{\gamma_\alpha},c_\beta \, r_2^{\gamma_\beta}\}}  \frac{\dif z}{|z|^{d+\sigma}}}_{=:J_2}  \\
 & \quad +M \, |u_0|_{BV} \underbrace{\int_{r_1<|z|<r_2} \left|c_\alpha \, |z|^{\gamma_\alpha}-c_\beta \, |z|^{\gamma_\beta} \right| \frac{\dif z}{|z|^{d+\sqrt{\alpha \, \beta}}}}_{=:J_3}\\
 & \quad +\frac{C(d) \, M \, |u_0|_{BV} }{\epsilon} \underbrace{\frac{1}{|\alpha-\beta|} \int_{|z|<r_1} \left|
 c_\beta \, |z|^{\gamma_\beta}-c_\alpha \,
|z|^{\gamma_\alpha}\right|^2 \frac{\dif z}{|z|^{d+\sqrt{\alp \, \beta}}}}_{=:J_4}.
\end{split}
\end{equation}
The rest of proof consists in estimating $\limsup_{\alpha,\beta \ra \lambda} \frac{J_i}{|\alpha-\beta|}$ ($i=1,\dots,4$). 
We will 
use the letter $C$ to denote various constants $C=C(d,\lambda)$. 

\medskip

\noindent {\bf 5.} {\it The case $\lambda \in (1,2)$.} We first let $r_2 \ra +\infty$ so that $(c_\alpha \, r_2^{\gamma_\alpha}) \wedge (c_\beta \, r_2^{\gamma_\beta}) \ra +\infty$, since all these coefficients are positive (cf. step {\bf 2}). We get at the limit
\begin{equation}\label{vlc-tech-10}
J_1=J_2=0 
\end{equation}
and
$
J_3=\int_{|z|>r_1} \left|c_\alpha \, |z|^{-d-\sigma_\alpha}-c_\beta \, |z|^{-d-\sigma_\beta} \right| \dif z,
$
with $\sigma_\alpha:=\sqrt{\alpha \, \beta}-\gamma_\alpha$ and $\sigma_\beta:=\sqrt{\alpha \, \beta}-\gamma_\beta$. 

Let us estimate $J_3$. We recognize a term of the same form than in \eqref{lc-ref-start} with the new ``locally Lipschitz'' coefficients $c_\alpha,c_\beta$ and powers $\sigma_\alpha,\sigma_\beta$. 
Arguing as before,
\begin{equation*}
\begin{split}
J_3 \leq |c_\alpha-c_\beta| \, \max_{\sigma =\sigma_\alpha,\sigma_\beta} \int_{|z|>r_1} \frac{\dif z}{|z|^{d+\sigma}}+(c_\alpha \vee c_\beta) \underbrace{\int_{|z|>r_1} \left||z|^{-d-\sigma_\alpha}-|z|^{-d-\sigma_\beta} \right| \dif z}_{=:\tilde{J}_3},
\end{split}
\end{equation*}
where $
\tilde{J}_{3} \leq S_d  \left|\frac{r_1^{-\sigma_\alpha}}{\sigma_\alpha}-\frac{r_1^{-\sigma_\beta}}{\sigma_\beta} \right|+2 \, S_d \left|\frac{1}{\sigma_\alpha}-\frac{1}{\sigma_\beta}\right| \mathbf{1}_{r_1 <1}.
$
By \eqref{properties-change}, 
\begin{equation*}
\begin{split}
\limsup_{\alpha,\beta \ra \lambda} \frac{J_3}{|\alpha-\beta|} 
\leq C \,\underbrace{ (r_1^{1-\lambda}+\mathbf{1}_{r_1<1})}_{\leq C \, r_1^{1-\lambda} \mbox{ if $\lambda>1$}}+C \, \underbrace{\limsup_{\alpha,\beta \ra \lambda} \frac{1}{|\alpha-\beta|} \left|\frac{r_1^{-\sigma_\alpha}}{\sigma_\alpha}-\frac{r_1^{-\sigma_\beta}}{\sigma_\beta}  \right|}_{=:\tilde{\tilde{J}}_{3}},
\end{split}
\end{equation*}
where a Taylor expansion with integral remainder shows that 
$$
\tilde{\tilde{J}}_{3}=\limsup_{\alpha,\beta \ra \lambda} \frac{|\sigma_\alpha-\sigma_\beta|}{|\alpha-\beta|} \, \Big|\int_0^1 \frac{\sigma_\tau \, r_1^{-\sigma_\tau} \, \ln r_1+r_1^{-\sigma_\tau}}{\sigma_\tau^2} \, \dif \tau \Big| \leq C \, r_1^{1-\lambda} \, (1+|\ln r_1|),
$$
with $\sigma_\tau:=\tau \, \sigma_\alpha+(1-\tau) \, \sigma_\beta$. We deduce the following estimate:
\begin{equation}\label{tech-p-10-b}
\limsup_{\alpha,\beta \ra \lambda} \frac{J_{3}}{|\alpha-\beta|} \leq C \, r_1^{1-\lambda}
\, (1 + |\ln r_1|).
\end{equation}  
Let us notice that this estimate fails when $\lambda = 1$, because $\sigma_\alpha,\sigma_\beta \ra \lambda-1=0$ as $\alpha,\beta \ra 1$.

Let us now estimate $J_4$. By adding and subtracting terms,
\begin{equation*}
\begin{split}
J_4 & \leq \frac{1}{2} \sum_\pm \frac{|c_\alpha \mp c_\beta|^2}{|\alpha-\beta|} \underbrace{\int_{|z|<r_1} \left||z|^{\gamma_\alpha} \pm |z|^{\gamma_\beta}\right|^2 \frac{\dif z}{|z|^{d+\sqrt{\alpha \, \beta}}}}_{=:J_{4,\pm}}.
\end{split}
\end{equation*}
By expanding the squares and integrating,
$$
J_{4,\pm}=S_d \left(\frac{r_1^{2 \, \gamma_\alpha-\sqrt{\alpha \, \beta}}}{2 \, \gamma_\alpha-\sqrt{\alpha \, \beta}} +
\frac{r_1^{2 \, \gamma_\beta-\sqrt{\alpha \, \beta}}}{2 \, \gamma_\beta-\sqrt{\alpha \, \beta}}  \pm 2 \, \frac{r_1^{\gamma_\alpha+\gamma_\beta-\sqrt{\alpha \, \beta}}}{\gamma_\alpha+\gamma_\beta-\sqrt{\alpha \, \beta}}\right).
$$
By \eqref{properties-change}, the limit of $J_{4,+}$ is easy to compute and we get
$$
\limsup_{\alpha,\beta \ra \lambda} \frac{J_{4}}{|\alpha-\beta|} \leq C \, r_1^{2-\lambda}+C \, \underbrace{\limsup_{\alpha,\beta \ra \lambda} \frac{J_{4,-}}{(\alpha-\beta)^{2}}}_{=:\tilde{J}_{4,-}}. 
$$
We estimate $\tilde{J}_{4,-}$ by multiplying and dividing by $(\gamma_\alpha-\gamma_\beta)^2$ and changing the variables by 
$a:=\gamma_\alpha-\frac{\sqrt{\alpha \, \beta}}{2}$ and $b:=\gamma_\beta-\frac{\sqrt{\alpha \, \beta}}{2}$. We get
\begin{equation*}
\begin{split}
\tilde{J}_{4,-} & \leq \limsup_{\alpha,\beta \ra \lambda} \frac{(\gamma_{\alpha}-\gamma_\beta)^{2}}{|\alpha-\beta|^{2}}\\
& \quad \cdot  \limsup_{a,b \ra c} \frac{1}{|a-b|^2} \left(\frac{r_1^{2 \, a}}{2 \, a}+ \frac{r_1^{2 \, b}}{2 \, b}-\frac{2 \, r_1^{a+b}}{a+b}\right),
\end{split}
\end{equation*}
where $c:=1-\frac{\lambda}{2}>0$ is the limit of $a,b$ as $\alpha,\beta \ra \lambda$.
By \eqref{properties-change} and the estimation of the last limit in Lemma~\ref{lem-taylor}\eqref{t5} in appendix, 
$$
\tilde{J}_{4,-} \leq C \, r_1^{2 -\lambda} \, (1+\ln^2 r_1).
$$
We conclude that
\begin{equation}\label{tech-21}
\limsup_{\alpha,\beta \ra \lambda} \frac{J_{4}}{|\alpha-\beta|} \leq C \, r_1^{2-\lambda} \, (1 + \ln^2 r_1).
\end{equation} 
Note that this estimate works even if $\lambda=1$.

We are now ready to conclude the proof and show
\eqref{main-esti-second} when $\lambda \in (1,2)$. Recall that
we estimate $\Lip_\alpha (u;\lambda)$ using \eqref{vlc-general}
with $r_2=+\infty$. The limsups of the terms on the right-hand side are
estimated by \eqref{vlc-tech-10}, \eqref{tech-p-10-b} and
\eqref{tech-21}. We get for all $\epsilon >0$ and 
$r_1 > 0$, 
\begin{equation*}
\begin{split}
\Lip_\alpha (u;\lambda) \leq C \, |u_0|_{BV} \left\{\epsilon+M \left(r_1^{1-\lambda}
\, (1 + |\ln r_1|) +\frac{r_1^{2-\lambda}}{\epsilon} \, (1 + \ln^2 r_1) \right)\right\}.
\end{split}
\end{equation*}
We complete the proof by taking $\epsilon= M^\frac{1}{\lambda} \, (1 + |\ln M|)$ and $r_1=M^\frac{1}{\lambda}$.

\medskip

\noindent {\bf 6.} {\it The case $\lambda=1$.} We have to estimate
again $J_i$ in \eqref{vlc-general} ($i=1,\dots,4$). This time, we do
not let $r_2 \ra +\infty$.  

For $J_1$, we recognize again a term of the form \eqref{lc-ref-start}
and we argue in the same way to estimate it. The only difference is
that the fixed cutting parameter $\tilde{r}$ is replaced by a moving
one $(c_\alpha \, r_2^{\gamma_\alpha}) \wedge (c_\beta \,
r_2^{\gamma_\beta})$. But, by
\eqref{properties-change} it follows that
$\lim_{\alpha,\beta \ra 1} (c_\alpha \, r_2^{\gamma_\alpha}) \wedge
(c_\beta \, r_2^{\gamma_\beta})=G_d(1) \, r_2$ with $G_d(1)>0$, and we leave it to the reader to verify
that this is sufficient to extend the proof of \eqref{tech-p-10} to
the current case. Now,
this estimate becomes   
\begin{equation}
\label{lc-J1}
\limsup_{\alpha,\beta \ra 1} \frac{J_1}{|\alpha-\beta|} \leq C \, (G_d(1) \, r_2)^{-1} \left(1 + |\ln (G_d(1) \, r_2)| \right) \leq  C \, r_2^{-1} \, (1 + |\ln r_2|).
\end{equation}

For $J_2$, we use that
\begin{equation*}
\begin{split}
J_2 & = S_d \, \max_{\sigma=\alpha,\beta} \frac{1}{\sigma} \left| (c_\alpha \, r_2^{\gamma_\alpha})^{-\sigma} - (c_\beta \, r_2^{\gamma_\beta})^{-\sigma} \right|\\
& =  S_d \, \max_{\sigma=\alpha,\beta}  |c_\alpha \, r_2^{\gamma_\alpha}-c_\beta \, r_2^{\gamma_\beta}| 
\int_0^1 \left(\tau \, c_\alpha \, r_2^{\gamma_\alpha}+(1-\tau) \, c_\beta \, r_2^{\gamma_\beta} \right)^{-\sigma-1} \, \dif \tau. 
\end{split}
\end{equation*}
By \eqref{properties-change} and a simple passage to the limit under the integral sign,
$$
\limsup_{\alpha,\beta \ra 1} \frac{J_2}{|\alpha-\beta|} \leq C \, r_2^{-2}
\limsup_{\alpha,\beta \ra 1} \frac{|c_\alpha \, r_2^{\gamma_\alpha}-c_\beta \, r_2^{\gamma_\beta}| }{|\alpha-\beta|}.
$$
To estimate the last limit, we write 
$$
|c_\alpha \, r_2^{\gamma_\alpha}-c_\beta \, r_2^{\gamma_\beta}| \leq 
|c_\alpha -c_\beta | \, (r_2^{\gamma_\alpha} \vee r_2^{\gamma_\beta})+(c_\alpha \vee c_\beta ) \,|r_2^{\gamma_\alpha}-r_2^{\gamma_\beta}|,
$$
where $|r_2^{\gamma_\alpha}-r_2^{\gamma_\beta}|=|\gamma_\alpha-\gamma_\beta| \, |\ln r_2|
\int_{0}^1 r_2^{\tau \, \gamma_\alpha+(1-\tau) \, \gamma_\beta}  \, \dif \tau$. Hence, again by \eqref{properties-change},
\begin{equation}\label{lc-J2}
\limsup_{\alpha,\beta \ra 1} \frac{J_2}{|\alpha-\beta|} \leq C \, r_2^{-1} \, (1 + |\ln r_2|). 
\end{equation}

We have to do again the estimate of $J_3$, since the preceding one \eqref{tech-p-10-b} fails. 
\begin{equation*}
\begin{split}
J_3 & \leq |c_\alpha-c_\beta| \max_{\sigma =\alpha,\beta} \int_{r_1<|z|<r_2} |z|^{\gamma_\sigma} \,  \frac{dz}{|z|^{d+\sqrt{\alpha \, \beta}}}\\
& \quad +(c_\alpha \vee c_\beta) \underbrace{\int_{r_1<|z|<r_2} \left| |z|^{\gamma_\alpha}-|z|^{\gamma_\beta} \right| \frac{dz}{|z|^{d+\sqrt{\alpha \, \beta}}}}_{=:\tilde{J}_{3}},
\end{split}
\end{equation*}
so that by \eqref{properties-change} and a simple passage to the limit under the integral sign,
\begin{equation}
\label{tech-nl-1-p-b}
\limsup_{\alpha,\beta \ra 1} \frac{J_{3}}{|\alpha-\beta|}  \leq C \, (|\ln r_1|+|\ln r_2|)+C \, \limsup_{\alpha,\beta \ra 1} \frac{\tilde{J}_3}{|\alpha-\beta|}. 
\end{equation}
To estimate $\tilde{J}_{3}$, we first assume 
that $\alpha,\beta \neq 1$, so that $\gamma_\alpha-\sqrt{\alpha \,
\beta}=(1-\alpha) \, \gamma_\alpha \neq 0$ and $\gamma_\beta-\sqrt{\alpha \, \beta}\neq 0$.  Hence, 
$\tilde{J}_{3}=S_d \, (\int_{r_1}^1 \dots+\int_{1}^{r_2} \dots )$ in polar coordinates, and
\begin{equation*}
\begin{split}
\tilde{J}_{3} & \leq S_d \sum_{i=1,2} \left|\frac{r_i^{\gamma_\alpha -\sqrt{\alpha \, \beta}}-1}{\gamma_\alp -\sqrt{\alpha \, \beta}}
-\frac{r_i^{\gamma_\beta -\sqrt{\alpha \, \beta}}-1}{\gamma_\beta -\sqrt{\alpha \, \beta}}\right|.
\end{split}
\end{equation*}
By Lemma~\ref{lem-taylor}\eqref{t3} in the appendix, 
\begin{equation*}
\begin{split}
\tilde{J}_{3} & \leq 2 \, S_d \, |\gamma_\alpha-\gamma_\beta| \, \max_{i=1,2} \max_{\sigma=\alpha,\beta} (1 \vee r_i^{\gamma_\sigma -\sqrt{\alpha \, \beta}}) \,  \ln^2 r_i.
\end{split}
\end{equation*}
By sending $\alp$ or $\beta\ra1$, we see that this inequality holds
also when $\alpha$ or $\beta=1$.
 Hence, by \eqref{properties-change} and \eqref{tech-nl-1-p-b}, 
\begin{equation}\label{lc-J3}
\limsup_{\alpha,\beta \ra 1} \frac{J_{3}}{|\alpha-\beta|} \leq
C \, (|\ln r_1| \vee \ln^2 r_1 +|\ln r_2| \vee \ln^2 r_2).
\end{equation}

Finally, for $J_4$, we use \eqref{tech-21} which is still valid and we are ready to show \eqref{main-esti-second} in the critical case. 
By \eqref{vlc-general}, \eqref{lc-J1}, \eqref{lc-J2}, \eqref{lc-J3} and \eqref{tech-21}, we have for all $\epsilon>0$, and
$r_2 \geq r_1 >0$,
\begin{equation*}
\begin{split}
\Lip_\alpha(u;1)
& \leq C \, |u_0|_{BV} \, \epsilon \\
& \quad +C \, M \, \|u_0\|_{L^1} \,  r_2^{-1} \, (1 + |\ln r_2|) \\
& \quad+C \, M \, |u_0|_{BV} (|\ln r_1| \vee \ln^2 r_1 +|\ln r_2|\vee \ln^2 r_2)\\
& \quad +C \, M \, |u_0|_{BV}  \, \frac{r_1}{\epsilon} \, (1 + \ln^2 r_1).
\end{split}
\end{equation*}
We complete the proof by taking $\epsilon= M \, (1 + |\ln M|)$, $r_1=M \wedge 1$, $r_2=1 \vee \frac{\|u_0\|_{L^1}}{|u_0|_{BV}}$, and noting that
$\|u_0\|_{L^1} \leq |u_0|_{BV}$ if $r_2=1$. 
\end{proof}

\section{Proof of Theorem \ref{thm:limit}}
\label{sec:limit}

This section is devoted to the proof of Theorem \ref{thm:limit}. Let
us first recall the notions of entropy solutions of \eqref{nat:scl}
and \eqref{nat:local-eq} introduced in~\cite{Kru70,Car99}. For
\eqref{nat:local-eq}, we use an equivalent definition introduced in
\cite{KaRi03}.
\begin{definition}[Entropy solutions]
\label{edef}
Let $u_0 \in L^\infty \cap L^1(\R^d)$ and  \eqref{flux}--\eqref{Aflux} hold.
Let $u\in L^\infty(Q_T) \cap L^\infty(0,T;L^1)$.
\begin{enumerate}[\rm (1)]
\item \label{entropy-scl} $u$ is an entropy solution of \eqref{nat:scl} if, for
all $k\in\R$ and all nonnegative $\phi\in C^\infty_c(\R^{d} \times [0,T))$, 
\begin{equation*}
\begin{split}
& \int_{Q_{T}} \Big(|u-k|\,\partial_t\phi+q_{f}(u,k) \cdot \nabla \phi
-\sgn (u-k) \, \varphi(u)\,\phi \Big)\, \dif x \, \dif t \\
&+\int_{\R^d}|u_0(x)-k|\,\phi(x,0) \, \dif x\geq 0.
\end{split}
\end{equation*}
\item 
$u$ is an entropy solution of \eqref{nat:local-eq} if,
\begin{enumerate}[\rm (a)]
\item \label{H1regularity} $\varphi(u)\in L^2(0,T;H^1)$,
\item 
\label{entropy-ineq-local-eq}
and for all $k\in\R$ 
and all nonnegative $\phi\in C^\infty_c(\R^{d} \times [0,T))$,
\begin{equation*}
\begin{split}
& \int_{Q_{T}} \Big(|u-k|\,\partial_t\phi+q_{f}(u,k) \cdot \nabla \phi
+|\varphi(u)-\varphi(k)|\, \Delta \phi \Big)\, \dif x \, \dif t \\
&+\int_{\R^d}|u_0(x)-k|\,\phi(x,0) \, \dif x\geq 0.
\end{split}
\end{equation*}
\end{enumerate}
\end{enumerate}
\end{definition}

To prove Theorem \ref{thm:limit}, we need to establish some technical lemmas. Let us begin by a compactness result. 

\begin{lemma}\label{Clem1}
Let $u_0 \in L^\infty \cap L^1(\R^d)$, \eqref{flux}--\eqref{Aflux} hold, and for each $\alpha \in (0,2)$, let $u^\alpha$ be the unique entropy solution to \eqref{1}. Then,
there exist $u,w \in L^\infty(Q_T)\cap C([0,T];L^1)$ such
that~$u=\lim_{\alpha \da 0} u^\alpha$ and~$w=\lim_{\alpha \ua 2}
u^\alp$, up to subsequences, in~$C([0,T];L^1_{\rm loc})$ and almost everywhere in $Q_T$.
\end{lemma}
 
\begin{proof}
We only do the proof for $w$, the proof for $u$ being similar.
Let us consider a sequence $\alp_m \ua 2$ and let us define
$E:=\left\{u^{\alpha_m}\right\}_{m}$. We will show that $E$ is
relatively compact in $C([0,T];L^1_{\rm loc})$.  
First we take 
a sequence 
$\{u^n_0\}_{n} \subset L^\infty \cap L^1 \cap
BV(\R^d)$ that converges to $u_0$ in $L^1(\R^d)$, and let $E_n$
denote the family $\{u^{\alpha_m}_n\}_{m}$ of entropy solutions to \eqref{1}
with $\alp=\alp_m$ and $u^n_0$ as initial data. We begin by showing
that $E_n$ is 
relatively compact in $C([0,T];L^1_{\rm loc})$.

The family $E_n$ is equicontinuous
in $C([0,T];L^1)$ by Corollary \ref{th:time}, and
Remark~\ref{rem-l2}\eqref{rem-ec-2}.
For each $t \in [0,T]$, $\{u^{\alpha_m}_n(\cdot,t)\}_{m}$ is relatively compact in
$L^1_{\rm loc}(\R^d)$ by the $L^1 \cap BV$-bound \eqref{nat:nonincrease} and
Helly's theorem.  By the Arzela-Ascoli theorem, $E_n$ is relatively compact in 
$C([0,T];L^1_{\rm loc})$ for any $n \in \mathbb{N}$. 

The relative compactness of $E$, and thus the existence of $w \in
C([0,T];L^1_{\rm loc})$, is now a consequence of the $L^1$-contraction
principle since 
\begin{equation}\label{contraction-eq}
\sup_{m\in\N} \|u^{\alpha_m}-u^{\alpha_m}_n\|_{C([0,T];L^1)} \leq \|u_0-u_0^n\|_{L^1} \ra 0 \quad \mbox{as $n \ra +\infty$.}
\end{equation}

Taking a subsequence if necessary, we can assume that $u^{\alpha_m}$ converges to $w$ in $C([0,T];L^1_{\rm loc})$ and almost everywhere in $Q_T$. In particular, by the a priori estimate \eqref{nat:nonincrease}, we
infer that $w\in L^\infty(Q_T)$. To prove that $w\in
  C([0,T];L^1)$, we observe that $E$ is equicontinuous in
$C([0,T];L^1)$ by the triangle inequality, the convergence estimate
\eqref{contraction-eq}, and the equicontinuity of $E_n$. Hence, for
any $R>0$, $m \in \mathbb{N}$, and $t,s \in [0,T]$,
\begin{equation*}
\begin{split}
& \|(w(\cdot,t)-w(\cdot,s)) \, \mathbf{1}_{|x|<R}\|_{L^1} \\
& \leq \|u^{\alpha_m}(\cdot,t)-u^{\alpha_m}(\cdot,s)\|_{L^1}\\
& \quad + \|(w(\cdot,t)-u^{\alpha_m}(\cdot,t)) \, \mathbf{1}_{|x|<R}\|_{L^1}
+\|(u^{\alpha_m}(\cdot,s)-w(\cdot,s)) \, \mathbf{1}_{|x|<R}\|_{L^1}\\
& \leq \co(1)+2 \, \|(w-u^{\alpha_m}) \, \mathbf{1}_{|x|<R}\|_{C([0,T];L^1)},
\end{split}
\end{equation*}
where $\co(1) \ra 0$ as $|t-s| \ra 0$ uniformly in $R$ and $m$. We
then conclude that 
$$
\|(w(\cdot,t)-w(\cdot,s))\|_{L^1} \leq \co(1)\quad\text{as } |t-s| \ra 0
$$
by first sending $m \ra +\infty$ and then $R \ra
+\infty$ using Fatou's lemma.
\end{proof}

Let us now verify that these limits satisfy the entropy inequalities of the preceding  definition.
\begin{lemma}\label{referee-lem}
Under the assumptions of Lemma \ref{Clem1}, $u$ and $w$ satisfy the
entropy inequalities of Definition \ref{edef}\eqref{entropy-scl} and~\eqref{entropy-ineq-local-eq} respectively.
\end{lemma}  

In the proof we need the following lemma:

\begin{lemma}
\label{lem01}
A function $u \in L^\infty(Q_T) \cap L^\infty(0,T;L^1)$ is an entropy solution of~\eqref{1} (cf. Definition
\ref{L1-entropy}) if and only if for all convex~$\eta \in C^1(\R)$,
all~$r>0$ and all nonnegative~$\phi \in C_c^\infty(\R^{d} \times [0,T))$,  
\begin{equation}\label{entropy-ineq-eta}
\begin{split}
& \int_{Q_{T}} \Big(\eta(u)\,\partial_t\phi+q_{f}^\eta(u) \cdot \nabla \phi \Big) \, \dif x \, \dif t\\
& +\int_{Q_T} \Big(q_\varphi^\eta(u) \,\Levy^\alpha_r [\phi]+ \eta' (u)\,\Levy^{\alpha,r} [\varphi(u)]\,\phi
\Big) \, \dif x \, \dif t \\
&+\int_{\R^d}
\eta(u_0(x))\,\phi(x,0) \, \dif x\geq 0,
\end{split}
\end{equation}
where $q_{g}^\eta(u):=\int_0^u \eta'(\tau) \, g'(\tau) \, \dif \tau$
(for $g=f,\varphi$).
\end{lemma}  

This result is well-known for (local) conservation laws, see
e.g. \cite[p.~27]{HoRi07}. Because of the presence of the
discontinuous sign function in the Kruzhkov formulation
\eqref{entropy_ineq}, any proof will be more technical than in the local
case and we therefore provide one in Appendix \ref{eta-entropy}.

\begin{proof}[Proof of Lemma \ref{referee-lem}] We begin with the proof for $w$ which is easier. 

\smallskip

\noindent {\bf 1.} {\it Entropy inequalities for $w$.} Using the
definition of $\Levy_r^\alpha$ and $\Levy^{\alpha,r}$ in \eqref{nat:Levy-form}, we send $r \ra
+\infty$ in the entropy inequality \eqref{entropy_ineq} and find that
\begin{equation}
\label{int}
\begin{split}
&\int_{Q_{T}} \Big( |u^\alpha-k|\,\partial_t\phi+q_{f}(u^\alpha,k) \cdot \nabla \phi- |\varphi(u^\alpha)-\varphi(k)| \,(-\triangle)^{\frac{\alpha}{2}}\phi \Big)\, \dif x \, \dif t \\
&+\int_{\R^d}|u_0(x)-k|\,\phi(x,0) \, \dif x\geq 0.
\end{split}
\end{equation}
Since $(-\triangle)^{\frac{\alpha}{2}}\phi=\mathcal{F}^{-1}
\left(|2 \, \pi \cdot|^\alpha \, \mathcal{F} \phi \right)$ and
$-\triangle \phi=\mathcal{F}^{-1} \left(|2 \, \pi \cdot|^2 \,
  \mathcal{F} \phi \right)$, by Plancherel
\begin{align}
\label{lim_2}-(-\triangle)^{\frac{\alpha}{2}} \phi \ra \triangle \, \phi \quad \text{in } L^2(Q_T)  \text{ as }
\alpha \ua 2.
\end{align}

 To get the entropy inequalities of Definition
 \ref{edef}\eqref{entropy-ineq-local-eq}, we must pass to the
 limit in \eqref{int}. This is straightforward for the local terms due to Lemma~\ref{Clem1} and~\eqref{nat:nonincrease}. For the nonlocal term, we first observe that 
\begin{equation*}
\begin{split}
& -\int_{Q_{T}} |\varphi(u^\alpha)-\varphi(k)| \,(-\triangle)^{\frac{\alpha}{2}}\phi\, \dif x \, \dif t\\
& =\int_{Q_{T}} \underbrace{\Big\{|\varphi(u^\alpha)-\varphi(k)|
    -|\varphi(k)|
  \Big\}}_{=:q(u^\alpha)} \Big\{\triangle\phi-\triangle\phi-(-\triangle)^{\frac{\alpha}{2}}\phi\Big\}\, \dif
x \, \dif t\\
&\leq\int_{Q_{T}} q(u^\alp) \, \triangle\phi \, \dif
x \, \dif t + \|q(u^\alp)\|_{L^2(Q_T)}\|\triangle\phi+(-\triangle)^{\frac{\alpha}{2}}\phi\|_{L^2(Q_T)}.
\end{split}
\end{equation*}
By \eqref{lim_2}, the second term tends to zero since
$\|q(u^\alp)\|_{L^2(Q_T)}$ is bounded independently of $\alp$. 
The boundedness follows from \eqref{nat:nonincrease} and an $(L^1,L^\infty)$-interpolation
argument since  $q \in W^{1,\infty}_{\rm loc}(\R)$ and $q(0)=0$.
By the $C([0,T];L^1_{\rm loc})$-convergence of $u^\alp$ (up to a subsequence), the first
term converges as $\alp\ua2$ to 
$$\int_{Q_{T}} |\varphi(w)-\varphi(k)| \, \triangle\phi\, \dif
x \, \dif t .$$
This completes the proof for $w$. 

\medskip

\noindent {\bf 2.} {\it Entropy inequalities for $u$.} Let us fix
$r>0$ for the duration of this proof and start from the entropy
inequalities \eqref{entropy-ineq-eta}, written for 
convex and $C^1$-entropies $\eta$. 

There is again no difficulty to pass to the limit as $\alpha \da 0$ in
the local terms of \eqref{entropy-ineq-eta}. For the first nonlocal
term, we use that $\Levy^\alpha_r [\phi] \ra 0$ uniformly on
$Q_T$. This is readily seen from \eqref{nat:Levy-form} and
\eqref{properties-Gdalpha}. Let us also notice that $q_\varphi^\eta$,
defined just below \eqref{entropy-ineq-eta}, satisfies $q_\varphi^\eta
\in W^{1,\infty}_{\rm loc}(\R)$ and $q_\varphi^\eta(0)=0$. Hence
$$
\int_{Q_{T}} q_\varphi^\eta(u^\alpha) \,\Levy^\alpha_r [\phi] \, \dif x \, \dif t \ra 0,
$$ 
since $q_\varphi^\eta(u^\alpha)$ is bounded in $C([0,T];L^1)$. For the remaining nonlocal term, we split the integral and get
\begin{equation*}
\begin{split}
& \int_{Q_{T}} \eta' (u^\alpha)\,\Levy^{\alpha,r} [\varphi(u^\alpha)]\,\phi
\, \dif x \, \dif t \\
&  \leq -\underbrace{G_d(\alpha) \int_{|z|>r} \frac{\dif z}{|z|^{d+\alpha}}}_{=:I} \int_{Q_{T}} \eta' (u^\alpha) \, \varphi(u^\alpha) \, \phi \, \dif x \, \dif t \\
& \quad +C \, \underbrace{\frac{G_d(\alpha)}{r^{d+\alpha}}}_{=:J} \|\varphi(u^\alpha)\|_{C([0,T];L^1)} \, \|\phi\|_{L^1(Q_T)}, 
\end{split}
\end{equation*}
where $C$ is an $L^\infty$-bound on $\eta'(u^\alpha)$. 
Notice that for all fixed $r$, $\lim_{\alpha \da 0} I=1$ and $\lim_{\alpha \da 0} J=0$ by \eqref{properties-Gdalpha}. 
Since $\eta'$ is continuous, we can pass to the limit as $\alpha \da 0$ in the inequality above, thanks to \eqref{nat:nonincrease}, the almost everywhere convergence of $u^\alpha$ (up to a subsequence), and the dominated convergence theorem. 

The limit in \eqref{entropy-ineq-eta} then implies that
\begin{equation*}
\begin{split}
& \int_{Q_{T}} \Big(\eta(u)\,\partial_t\phi+q_{f}^\eta(u) \cdot \nabla \phi -\eta' (u)\,\varphi(u)\,\phi \Big) \, \dif x \, \dif t +\int_{\R^d}
\eta(u_0(x))\,\phi(x,0) \, \dif x\geq 0,
\end{split}
\end{equation*}
for all convex $C^1$-entropies $\eta$ and fluxes $q_{f}^\eta(u)= \int_{0}^u \eta'(\tau) \, f'(\tau) \, \dif \tau$. 
It is then classical to get the desired Kruzhkov entropy inequalities of
Definition \ref{edef}\eqref{entropy-scl} from these inequalities, see
e.g. the if part of the proof in Appendix \ref{eta-entropy}.
\end{proof}

To prove that~$w$ satisfies \eqref{H1regularity}
of Definition \ref{edef}, we need to derive
an~$H^{\frac{\alpha}{2}}$-estimate on~$u^\alpha$. In the
sequel,~$H^{\frac{\alpha}{2}}(\R^d)$ denotes the 
fractional Sobolev space of~$u \in L^2(\R^d)$ such that
$\iint_{\R^{2d}} \frac{(u(x)-u(y))^2}{|x-y|^{d+\alpha}} \, \dif
x\,\dif y<+\infty$. The $H^{\frac{\alpha}{2}}$-semi-norm can be
defined in both the following equivalent ways: 
\begin{equation}\label{referee-def-norm}
|u|_{H^{\frac{\alpha}{2}}}^{2}:=\frac{G_d(\alpha)}{2} \iint_{\R^{2d}} \frac{(u(x)-u(y))^2}{|x-y|^{d+\alpha}} \, \dif x\,\dif y=\int_{\R^d} |2 \, \pi \, \xi|^{\alp} \, |\mathcal{F}u|^2 \, \dif \xi.
\end{equation}
The~$H^{\frac{\alpha}{2}}$-norm is defined
as~$\|u\|_{H^{\frac{\alpha}{2}}}^{2}:=\|u\|^2_{L^2}+|u|^2_{H^{\frac{\alpha}{2}}}$. The
equality in \eqref{referee-def-norm} is standard,
cf. e.g. \cite{Ada75}. In the sequel, the knowledge of the precise
constants will be important to get estimates uniform in $\alpha \ua
2$. For the sake of completeness, we therefore provide a short
computation of them in Appendix \ref{eta-entropy}.
\begin{lemma}
\label{Clem2}
Let $\alpha \in (0,2)$, $u_0 \in L^\infty\cap L^1(\R^d)$, \eqref{flux}--\eqref{Aflux} hold, and $u^\alpha$ be the unique entropy solution to \eqref{1}. Then
$$
\int_{\R^d} \Phi(u^\alpha(x,T)) \, \dif x+
|\varphi(u^\alpha)|^{2}_{L^2(0,T;H^{\frac{\alpha}{2}})}\leq
\int_{\R^d} \Phi(u_0(x)) \, \dif x,
$$
where $\Phi(u):=\int_0^u \varphi(\tau) \, \dif \tau$ for all $u \in \R$. 
\end{lemma}
\begin{remark}
Note that $\Phi$ is nonnegative,
convex, and $0$ at $0$.
\end{remark}
\begin{proof}
We can take $\eta=\Phi$ in~\eqref{entropy-ineq-eta}, since it is $C^1$ and convex by \eqref{Aflux}.
Using also
Lemma \ref{lem01} and the continuity of~$u^\alpha$ in time with values
in~$L^1(\R^d)$, 
as in Remark~\ref{rem-def-cont}, we find that for all~$\phi \in C_c^\infty(\R^{d+1})$,
\begin{equation}\label{conv-tech}
\begin{split}
& \int_{Q_{T}} \Big(\Phi(u^\alp)\,\partial_t\phi+ q_{f}^{\Phi}(u^\alp) \cdot \nabla \phi \Big) \, \dif x \, \dif t\\
& +\int_{Q_{T}} \Big(q_{\varphi}^{\Phi}(u^\alp) \,\Levy^\alpha_r [\phi]+ \varphi(u^\alp)\,\Levy^{\alpha,r} [\varphi(u^\alp)]\,\phi \Big)
\, \dif x \, \dif t \\
&+\int_{\R^d}
\Phi(u_0(x))\,\phi(x,0) \, \dif x \geq \int_{\R^d} \Phi(u^\alpha(x,T)) \, \phi(x,T) \, \dif x.
\end{split}
\end{equation}
Then take $\phi(x,t)=\gamma_R(x)$, where $R>0$ and $\gamma_R$ is an
approximation of~$\mathbf{1}_{|x| < R}$ such that~$\gamma_R \in
C_c^\infty(\R^d)$,~$\{\gamma_R\}_{R>0}$ is bounded
in~$W^{2,\infty}(\R^d)$, $\gamma_R \ra 1$ in~$W^{2,\infty}_{\rm
  loc}(\R^d)$ as~$R \ra +\infty$.  
It is obvious that the $\nabla$- and $\Levy_r^{\alpha}$-terms in \eqref{conv-tech} vanish as $R \ra +\infty$, since $q_g^\Phi \in W^{1,\infty}_{\rm loc}(\R)$ and $q_g^\Phi(0)=0$ for $g=f,\varphi$. For the $\Levy^{\alpha,r}$-term,
a standard computation shows that for all~$u,v \in L^2(\R^d)$ and~$r>0$,
\begin{equation}
\label{ineqH1-1}
\begin{split}
& -\int_{\R^d} u \, \Levy^{\alpha,r}[v]  \, \dif x\\
& = -G_d(\alpha) \iint_{|z|>r} u(x) \, \frac{v(x+z)-v(x)}{|z|^{d+\alpha}} \, \dif z \, \dif x \\
& = \frac{G_d(\alpha)}{2} \left\{ \iint_{|x-y|>r}  u(x) \, v(x) \, \frac{\dif x \, \dif y}{|x-y|^{d+\alpha}}+\iint_{|x-y|>r}  u(y) \, v(y) \, \frac{\dif x \, \dif y}{|x-y|^{d+\alpha}}\right.\\
& \quad \left.-\iint_{|x-y|>r}  u(x) \, v(y) \, \frac{\dif x \, \dif y}{|x-y|^{d+\alpha}}-\iint_{|x-y|>r}  u(y) \, v(x) \, \frac{\dif x \, \dif y}{|x-y|^{d+\alpha}} \right\}\\
& = \frac{G_d(\alpha)}{2}  \iint_{|x-y|>r}  
\frac{(u(x)-u(y)) \, (v(x)-v(y))}{|x-y|^{d+\alpha}} \, \dif x \, \dif y. 
\end{split}
\end{equation}
Hence, by the dominated convergence theorem,
\begin{equation*}
\begin{split}
&\int_{Q_T} \varphi(u^\alpha)\,\Levy^{\alpha,r} [\varphi(u^\alpha)]\, \gamma_R \, \dif x \, \dif t\\
& = \frac{G_d(\alpha)}{2} \int_0^T \iint_{|x-y|>r} \left(\varphi(u^\alpha(x,t))-\varphi(u^\alpha(y,t))\right)\\
& \quad \cdot \left(\varphi(u^\alpha(x,t)) \, \gamma_R(x)-\varphi(u^\alpha(y,t)) \, \gamma_R(y)\right) \frac{\dif x \, \dif y}{|x-y|^{d+\alpha}} \, \dif t\\
& \ra \frac{G_d(\alpha)}{2} \int_0^T \iint_{|x-y|>r} \frac{\left(\varphi(u^\alpha(x,t))-\varphi(u^\alpha(y,t))\right)^2}{|x-y|^{d+\alpha}} \, \dif x \, \dif y \, \dif t
\end{split}
\end{equation*}
as~$R \ra +\infty$. Going to the limit in
  \eqref{conv-tech}, we then find that
\begin{equation*}
\begin{split}
& \int_{\R^d} \Phi(u^\alp(x,T))\, \dif x\\
&+\frac{G_d(\alpha)}{2} \int_0^T \iint_{|x-y|>r} \frac{\left(\varphi(u^\alpha(x,t))-\varphi(u^\alpha(y,t))\right)^2}{|x-y|^{d+\alpha}} \, \dif x \, \dif y \, \dif t\\
& \leq \int_{\R^d} \Phi(u_0(x))\, \dif x.
\end{split}
\end{equation*}
The proof is complete
by sending $r \da 0$ and using the monotone convergence theorem.
\end{proof}

From this energy type of estimate, we have the following result:

\begin{lemma}
Under the assumptions of Lemma \ref{Clem1}, $\varphi(w)\in
L^2(0,T;H^1)$. 
\end{lemma}
\begin{proof}
Recall first that by \eqref{nat:nonincrease} and a $(L^1,L^\infty)$-interpolation argument, $\{u^\alpha\}_{\alpha \in (0,2)}$ is bounded in $L^2(0,T;L^2)$. Using in addition the preceding lemma, we find a constant $C$ such that for all $\alpha \in (0,2)$, 
$$
\|\varphi(u^\alpha)\|_{L^2(0,T;H^{\frac{\alpha}{2}})} \leq C.
$$ 
Using the Fourier formula in \eqref{referee-def-norm}, 
$$
\int_{Q_T} (1+|2 \, \pi \, \xi|^{\alpha}) \, |\mathcal{F} \varphi(u^\alpha)|^2 \, \dif \xi \, \dif t \leq C
$$
(recall that $\mathcal{F}$ is the Fourier transform in space).
Now we use the following inequalities: for all~$1 \leq \beta \leq \alpha$ and all~$\xi \in \R^d$,
$$
(1+|2 \, \pi \, \xi|^\beta) \leq (1+|2 \, \pi \, \xi|)^\beta \leq (1+|2 \, \pi \, \xi|)^\alpha \leq 2^{\alpha-1} \, (1+|2 \, \pi \, \xi|^\alpha).
$$
We deduce that
$$
\int_{Q_T} (1+|2 \, \pi \, \xi|^{\beta}) \, |\mathcal{F}\varphi(u^\alpha)|^2 \, \dif \xi \, \dif t \leq 2^{\alpha-1} \, C.
$$
Going back to the integral formula in \eqref{referee-def-norm}, 
$$
\|\varphi(u^\alpha)\|_{L^2(0,T;L^2)}^2+\frac{G_d(\beta)}{2} \int_0^T \iint_{\R^{2d}} \frac{(\varphi(u^\alpha)(x,t)-\varphi(u^\alpha)(y,t))^2}{|x-y|^{d+\beta}} \, \dif x \, \dif y \, \dif t \leq  2^{\alpha-1} \, C.
$$
By Fatou's lemma, applied for $\alpha \ua 2$ with fixed~$\beta$,  
$$
\|\varphi(w)\|_{L^2(0,T;L^2)}^2+\frac{G_d(\beta)}{2} \int_0^T \iint_{\R^{2d}} \frac{(\varphi(w)(x,t)-\varphi(w)(y,t))^2}{|x-y|^{d+\beta}} \, \dif x \, \dif y \, \dif t \leq 2 \, C.
$$
Finally, Fatou's lemma applied to the Fourier formula shows that
\begin{equation*}
\begin{split}
2 \, C & \geq \liminf_{\beta \ua 2} \int_{Q_T} (1+|2 \, \pi \, \xi|^{\beta}) \, |\mathcal{F} \varphi(w)|^2 \, \dif \xi \, \dif t \\
& \geq \int_{Q_T} (1+|2 \, \pi \, \xi|^{2}) \, |\mathcal{F} \varphi(w)|^2 \, \dif \xi \, \dif t.
\end{split}
\end{equation*}
The proof is complete. 
\end{proof}

We end by the proof of Theorem \ref{thm:limit}.
\begin{proof}[Proof of Theorem \ref{thm:limit}]
Let $u,w \in L^\infty(Q_T) \cap C([0,T];L^1)$ be defined  in Lemma~\ref{Clem1}. By previous lemmas, they are entropy solutions of~\eqref{nat:scl} and~\eqref{nat:local-eq}, respectively. By uniqueness (cf. \cite{Kru70,Car99,KaRi03}), the whole sequences converge and the proof is complete.
\end{proof}

\section{Optimal example}
\label{sec:optim}

In this last section, we exhibit an example of an equation for which Theorems \ref{th:nonlin} and \ref{nat:th-Levy} are optimal. Note that the modulus in $f$ is the same than in \cite{Daf72,Luc80}. This modulus is optimal for linear fluxes, i.e. for equations of the form $\partial_t u+F \cdot \nabla u=0$ where $F \in \R^d$. This is readily seen by the formula $u(x,t)=u_0(x-t \, F)$. Here, we focus on the new fractional diffusion term. The proofs work for $\alpha=2$ and our example is also optimal for the results in \cite{CoGr99}. Let us finally mention that this example is motivated by Remark 2.1 of \cite{Dro03}
and similar
remarks in \cite{Imb05,DrIm06,Ali07}. 

Let us consider, for every~$\alpha \in [0,2]$ and~$\gamma,a>0$, 
\begin{equation}
\label{ex}
\begin{cases}
\partial_t u+a \,
(-\triangle)^{\frac{\alpha}{2}} u=0,\\
u(x,0)=\gamma \, \mathbf{1}_{Q}(\gamma^{-1} \, x),
\end{cases}
\end{equation}
where~$Q:=[-1,1]^d$. This is \eqref{1} with $u_0$ as above, $f\equiv0$ and $\varphi'\equiv a$. Notice that
\begin{equation}\label{formula-gamma}
\begin{cases}
\|u_0\|_{L^1}=2^d \, \gamma^{d+1},\\
|u_0|_{BV}=d \, 2^{d} \, \gamma^{d},\\
\Ent_i(u_0)=d \, 2^{d} \,\gamma^{d} \left(1+\left(\ln \frac{\gamma}{d}\right)^i \right) \mathbf{1}_{\gamma>d}, \quad \text{($i=1,2$),}
\end{cases}
\end{equation}
where $\Ent_i(u_0)$ is defined in \eqref{log-entrop}.

\subsection{Optimality of Theorem \ref{th:nonlin}} Let us fix $\alpha \in [0,2]$ and let us use the notation $u=:u_a$. Given $T>0$ and other parameters $b,c>0$, we define \label{page-modulus}
\begin{equation*}
\begin{split}
\omega_{a-b} & :=\begin{cases}|a^{\frac1\alp}-b^{\frac1\alp}|, &\quad 
\alp>1,\\
|a \, \ln a-b \, \ln b|,& \quad \alp=1,\\
|a-b|,& \quad  \alp<1,
\end{cases}\\
\sigma_T & :=\begin{cases} T^{\frac{1}{\alpha}}, &\quad 
\alp>1,\\
T \, |\ln T|,& \quad \alp=1,\\
T,& \quad  \alp<1,
\end{cases}\\
\sigma_\gamma & :=\begin{cases} \gamma^d, &\quad 
\alp>1,\\
\gamma^d \, \ln \gamma,& \quad \alp=1,\\
\gamma^{d+1-\alpha},& \quad  \alp<1.
\end{cases}
\end{split}
\end{equation*}
We also introduce the best Lipschitz constant of $a \mt u_a$ at $a=c$:
\begin{equation*}
\Lip_\varphi (u;c) := \limsup_{a,b\ra c} \frac{\|u_a-u_b\|_{C([0,T];L^1)}}{|a-b|}.
\end{equation*}
Theorem \ref{th:nonlin} and \eqref{formula-gamma} imply that the function $a \geq 0 \mt u_a \in C([0,T];L^1)$ is continuous at $a=0$ and locally Lipschitz continuous for $a>0$ with for all $c>0$,
\begin{equation*}
\begin{split}
\|u_a-u_b\|_{C([0,T];L^1)} & = \cO (\omega_{a-b}) \quad \mbox{as $a,b \da 0$,}\\
\Lip_\varphi (u;c) & = \cO (\sigma_T) \quad \mbox{as $T \da 0$,}\\
\Lip_\varphi (u;c) & = \cO (\sigma_\gamma) \quad \mbox{as $\gamma \ra +\infty$,}\\
\end{split}
\end{equation*}
while all the respective remaining parameters are fixed. The result below states that these estimates are optimal.
\begin{proposition}
\label{prop:opt-nonlin}
Let~$\alpha \in [0,2]$ and $c > 0$. 
\begin{enumerate}[\rm (i)]
\item \label{item-modulus} For all $T,\gamma>0$, $\liminf_{a,b\da 0}\frac{\|u_a-u_b\|_{C([0,T];L^1)}}{\omega_{a-b}}>0.$
\item \label{item-time} For all $\gamma>0,$ $\liminf_{T \da 0} \frac{\Lip_\varphi (u;c)}{\sigma_T}>0.$
\item \label{test-item} For all $T>0$, $\liminf_{\gamma \ra +\infty} \frac{\Lip_\varphi (u;c)}{\sigma_\gamma}>0.$
\end{enumerate}
\end{proposition}
\begin{remark}\label{rem-opt-nonlin}
This result shows that the modulus of continuity in
  $\varphi-\psi$ derived in \eqref{main-constants-first} is
  optimal for linear diffusion functions. It also shows that the $T$- and $u_0$-dependencies of this
  modulus are optimal in the limits $T\da 0$ or
  $\frac{\|u_0\|_{L^1}}{|u_0|_{BV}}\ra +\infty$
 (recall that $\frac{\|u_0\|_{L^1}}{|u_0|_{BV}} \sim \gamma$ by \eqref{formula-gamma}).
\end{remark}

\subsection{Optimality of Theorem \ref{nat:th-Levy}} Let us now use the notation $u=:u^\alpha$ to emphasize the dependence on $\alpha$. Given $\lambda \in (0,2)$, we define
\begin{equation*}
\begin{split}
\tilde{\sigma}_M & :=\begin{cases} M^{\frac{1}{\lambda}} \, |\ln M|, &\quad 
\lambda>1,\\
M \, \ln^2 M,& \quad \lambda=1,\\
M,& \quad \lambda<1,
\end{cases}\\
\tilde{\sigma}_\gamma & :=\begin{cases} \gamma^d, &\quad 
\lambda>1,\\
\gamma^d \, \ln^2 \gamma,& \quad \lambda=1,\\
\gamma^{d+1-\lambda} \, \ln \gamma,& \quad  \lambda<1,
\end{cases}
\end{split}
\end{equation*}
where $M:=T \, a$. We also consider the best Lipschitz constant of $\alpha \mt u^\alpha$ at $\alpha=\lambda$ defined in \eqref{nat:lip}. Then, Theorem \ref{nat:th-Levy} and \eqref{formula-gamma} imply that for all $\lambda \in (0,2)$,
\begin{equation*}
\begin{split}
\Lip_\alpha (u;\lambda) & = \cO (\tilde{\sigma}_M) \quad \mbox{as $M \da 0$,}\\
\Lip_\alpha (u;\lambda) & = \cO (\tilde{\sigma}_\gamma) \quad \mbox{as $\gamma \ra +\infty$,}
\end{split}
\end{equation*}
while all the respective remaining parameters are fixed. The result below states that these estimates are optimal.

\begin{proposition}
\label{prop:opt-powers}
Let $T,a>0$, $M=T \, a$, and $\lambda \in (0,2)$. There exist $M_0,\gamma_0>0$ such that:
\begin{enumerate}[\rm (i)]
\item \label{item-time-Levy} For all $ \gamma_0 \geq \gamma>0$, $\liminf_{M \da 0} \frac{\Lip_\alpha (u;\lambda)}{\tilde{\sigma}_M}>0.$
\item \label{test-item-Levy} For all $M_0 \geq M>0$, $\liminf_{\gamma \ra +\infty} \frac{\Lip_\alpha (u;\lambda)}{\tilde{\sigma}_\gamma}>0.$
\end{enumerate}
\end{proposition}

\begin{remark}\label{rem-opt-powers}
This result shows that the $M$- and $u_0$-dependencies in
  \eqref{main-esti-second} are optimal at the limits $M=T \,
  \|\varphi'\|_\infty \da 0$ or $\frac{\|u_0\|_{L^1}}{|u_0|_{BV}}\ra+\infty$. 
\end{remark}

\subsection{Proofs} 

\begin{proof}[Proof of Proposition~\ref{prop:opt-nonlin}] Let us prove each items in order.

\medskip

\noindent {\bf 1.} {\it Item \eqref{item-modulus}.} 
Let us first assume that $T=\gamma=1$. The general case
  will follow from a rescaling argument given at the end of the
  proof. Let us define 
\begin{align}\label{def-error}
\E_Q:=\int_Q u_a(x,1) \, \dif x-\int_Q u_b(x,1) \, \dif x.
\end{align}
Since 
$
\|u_a-u_b\|_{C([0,1];L^1)} \geq \|u_a(\cdot,1)-u_b(\cdot,1)\|_{L^1}\geq |\E_Q|, 
$
it suffices to
show that $\liminf_{a,b \da 0} \frac{|\E_Q|}{\omega_{a-b}}>0$. It is well-known that~$u_a(x,t)=\mathcal{F}^{-1}(e^{-t \, a \, |2 \, \pi \cdot|^\alp}) \ast \mathbf{1}_Q(x)$.
A short
computation shows that
\begin{equation}\label{opt-nonlin-tech3}
\begin{split}
\E_Q & =\int \mathcal{F}^{-1}(e^{-a \, |2 \, \pi \cdot|^\alp }-e^{-b \, |2 \, \pi \cdot|^\alp }) \, (\mathbf{1}_Q \ast
\mathbf{1}_Q) \,\dif x\\
& = \int
(e^{-a \, |2 \, \pi \, \xi|^\alp }-e^{-b \, |2 \, \pi \, \xi|^\alp }) \left(\mathcal{F} \mathbf{1}_Q\right)^2 \dif \xi\\
& = \frac{2^d}{\pi^d} \int
(e^{-a \, |\xi|^\alp }-e^{-b \, |\xi|^\alp }) \prod_{i=1}^d \sinc \! {}^2 (\xi_i) \, \dif \xi\\
& = \frac{2^d}{\pi^d} \int \int_0^1
(b-a) \, |\xi|^\alpha \, e^{-(\tau \, a+(1-\tau) \, b) \, |\xi|^\alp} \prod_{i=1}^d \sinc \! {}^2 (\xi_i) \, \dif \tau \, \dif \xi,
\end{split}
\end{equation}
where $\xi=:(\xi_1,\dots,\xi_d)$ and $\sinc(\xi_i):=\frac{\sin \xi_i}{\xi_i}$. To get the third line, we have used the formula~$
\mathcal{F} \mathbf{1}_Q(\xi)=\prod_{i=1}^d \frac{\sin (2 \, \pi \, \xi_i)}{\pi \, \xi_i}
$
and the change of variable $2 \, \pi \, \xi \mt \xi$. We now give separate arguments for
  the cases $\alp<1$, $\alp=1$, and $\alp>1$.

\medskip

{\bf a.} \emph{The case $\alpha<1$.} This is obvious since $
0<\int |\xi|^\alpha \prod_{i=1}^d \sinc \! {}^2 (\xi_i) \, \dif \xi <+\infty.
$

\medskip

{\bf b.} \emph{The case $\alpha>1$.} Note that~$|\xi|^\alpha \leq d^{\alpha-1} \sum_{i=1}^d |\xi_i|^\alpha$. Hence, by~\eqref{opt-nonlin-tech3},   
\begin{equation}\label{referee-tech1}
|\E_Q| \geq I_{a,b}  \int \int_0^1
|a-b| \, |\xi_1|^\alpha \, e^{-d^{\alpha-1} \, (\tau \, a+(1-\tau) \, b) \, |\xi_1|^\alpha} \, \sinc \! {}^2 (\xi_1) \, \dif \tau \, \dif \xi_1
\end{equation}
where
$$
I_{a,b}=\frac{2^d}{\pi^d} \prod_{i=2}^d \int  e^{-d^{\alpha-1} \, (a \vee b) \, |\xi_i|^\alpha} \, \sinc \! {}^2 (\xi_i) \, \dif \xi_i.
$$
Since $e^{-d^{\alpha-1} \, (a \vee b) \, |\xi_i|^\alpha} \ra 1$ as $a,b \da 0$,
\begin{equation}\label{referee-tech2}
I_{a,b} \geq C_0:=\frac{2^{d-1}}{\pi^d} \prod_{i=2}^d \int \sinc \! {}^2 (\xi_i) \, \dif \xi_i>0,
\end{equation} 
for all~$a,b>0$ sufficiently small.
Hence, assuming e.g. that~$a > b$, we get 
\begin{equation}
\label{opt-nonlin-tech2}
\begin{split}
|\E_Q| & \geq C_0\, \underbrace{\int 
a \, |\xi_1|^{\alpha-2} \, e^{-d^{\alpha-1} \, a \, |\xi_1|^{\alpha}} \, \sin^2 (\xi_1)  \, \dif \xi_1}_{=:I_a}\\
& \quad - C_0 \int
b \, |\xi_1|^{\alpha-2} \, e^{-d^{\alpha-1} \, b \, |\xi_1|^{\alpha}} \, \sin^2 (\xi_1) \, \dif \xi_1.
\end{split}
\end{equation}
Before continuing, notice that this estimate is valid for $\alpha=1$; this is will be useful later. Let us continue the case $\alpha>1$ by changing variables, 
$$
I_a=a^{\frac 1 \alp} \int |\xi_1|^{\alpha-2} \, e^{-d^{\alpha-1} \, |\xi_1|^{\alp}}  \, \sin^2 (a^{-\frac{1}{\alpha}} \, \xi_1) \, \dif \xi_1.
$$
Doing the same for the~$b$-integral and adding and subtracting term, 
\begin{equation}\label{opt-nonlin-tech1}
\begin{split}
|\E_Q| & \geq C_0 \, (a^{\frac 1 \alp}-b^{\frac 1 \alp}) \int |\xi_1|^{\alpha-2} \, e^{-d^{\alpha-1} \, |\xi_1|^\alp} \, \sin^2 (a^{-\frac{1}{\alpha}} \, \xi_1) \,  \dif \xi_1\\
& \quad +C_0 \, b^{\frac 1 \alp}  \int |\xi_1|^{\alpha-2} \, e^{-d^{\alpha-1} \, |\xi_1|^\alp}   \left\{\sin^2 (a^{-\frac{1}{\alpha}} \, \xi_1)-\sin^2 (b^{-\frac{1}{\alpha}} \, \xi_1) \right\}\dif \xi_1\\
& =:C_0 \, (a^{\frac 1 \alp}-b^{\frac 1 \alp}) \, I_1+C_0 \, b^{\frac 1 \alp} \, I_2.
\end{split}
\end{equation}
By a Taylor expansion and an integration by parts,
\begin{align*}
|I_2| & \leq (a^{\frac 1 \alp}-b^{\frac 1 \alp}) \, \bigg|\int_0^1 \int a_{\alpha,\tau}^{-2} \, \underbrace{|\xi_1|^{\alpha-2} \, \xi_1 \, e^{-d^{\alpha-1} \, |\xi_1|^\alp}}_{=:f(\xi_1)} \nonumber \\ 
&  \quad \cdot \underbrace{2 \, \sin \left(a_{\alpha,\tau}^{-1} \, \xi_1 \right) \cos \left(
a_{\alpha,\tau}^{-1} \, \xi_1 \right) }_{=\sin \left(2 \, a_{\alpha,\tau}^{-1} \, \xi_1 \right)} \, \dif \xi_1 \, \dif \tau\bigg| \nonumber \\
& \leq \frac{1}{2} \, b^{-\frac{1}{\alpha}}  \, (a^{\frac 1 \alp}-b^{\frac 1 \alp}) \left|\int_0^1 \int f'(\xi_1) \, \cos \left(2 \, a_{\alpha,\tau}^{-1} \, \xi_1 \right) \dif \xi_1 \, \dif \tau\right|,
\end{align*}
where~$a_{\alpha,\tau}:=\tau \, a^{\frac 1 \alp}+(1-\tau) \, b^{\frac
  1 \alp}$ and~$f'$ is integrable when~$\alpha>1$. We deduce that~$
C_0 \, b^{\frac 1 \alp} \, I_2=( a^{\frac 1 \alp}-b^{\frac 1 \alp}
)\co(1)$ as~$a,b \da 0$, since for fixed~$\tau$,~$\cos \left(2 \, a_{\alpha,\tau}^{-1} \, \cdot \right)$ converges to its zero mean
value in~$L^\infty$-weak-$\star$. By a similar argument
  $\sin^2 (a^{-\frac{1}{\alpha}} \, \cdot )$ also
  weakly-$\star$ converges to its positive mean value $m$ and hence
\begin{equation*}
\lim_{a,b \da 0} I_1=m \int |\xi_1|^{\alpha-2} \, e^{-d^{\alpha-1} \, |\xi_1|^\alp} \,  \dif \xi_1>0.
\end{equation*}
We thus conclude the result from~\eqref{opt-nonlin-tech1}.

\medskip

{\bf c.} \emph{The case $\alpha=1$.} We restart from~\eqref{opt-nonlin-tech2} assuming again that~$a > b$,~$a,b$ small. This time we cut~$I_a$ into three pieces.
$$
I_a=\int_{1<|\xi_1| <a^{-1} } \dots+\int_{|\xi_1| < 1} \dots+\int_{|\xi_1| > a^{-1} } \dots.
$$
We do the same for the~$b$-integral and we get 
\begin{align*}
  |\E_Q| & \geq C_0 \int_{1<|\xi_1|<a^{-1}} a \, |\xi_1|^{-1} \, e^{- a \, |\xi_1|} \, \sin^2 (\xi_1) \,  \dif \xi_1\nonumber\\
  & \quad -C_0 \int_{1<|\xi_1|<b^{-1}} b \, |\xi_1|^{-1} \, e^{- b \, |\xi_1|} \, \sin^2 (\xi_1)  \,  \dif \xi_1\nonumber\\
  & \quad +C_0 \left(\int_{|\xi_1| <1 } \dots-\int_{|\xi_1| <1 } \dots
  \right)+C_0 \left(\int_{|\xi_1| > a^{-1} } \dots-\int_{|\xi_1| >
      b^{-1} } \dots\right).
\end{align*}
The last two terms are $\cO(a-b)=(b \, \ln b-a \, \ln a) 
\co(1)$ as~$a,b \da 0$. To show this, 
we follow line by line the arguments of {\bf a} and~{\bf b}
  respectively, noting that all integrals are well-defined because of
  the new domains of integration. Let now $I$ denote the remaining
term. Recalling that~$a > 
b$, 
\begin{equation*}
\begin{split}
I & = C_0 \int_{1<|\xi_1|<a^{-1}} |\xi_1|^{-1} \, (a \, e^{- a \, |\xi_1|}-b \, e^{- b \, |\xi_1|}) \, \sin^2 (\xi_1) \,  \dif \xi_1\\
& \quad -C_0 \int_{a^{-1}<|\xi_1|<b^{-1}} b \, |\xi_1|^{-1} \, e^{- b \, |\xi_1|} \, \sin^2 (\xi_1)  \,  \dif \xi_1\\
& =: I_1+I_2.
\end{split}
\end{equation*}
Note that 
\begin{equation*}
\begin{split}
|I_2| & \leq C_0 \int_{a^{-1}<|\xi_1|<b^{-1}} b \, |\xi_1|^{-1} \, \dif \xi_1 \\
 & = 2 \, C_0 \, b \, (\ln a-\ln b) \leq 2 \, C_0 \, (a-b)=(b \, \ln b-a \, \ln a)\co(1)
\end{split}
\end{equation*}
as $a,b \da 0$.
Hence it remains to show that~$\liminf_{a,b \stackrel{a > b>0}{\ra} 0}
\frac{I_1}{b \, \ln b-a \, \ln a}>0$. Since
\begin{equation*}
\begin{split}
a \, e^{- a \, |\xi_1|}-b \, e^{- b \, |\xi_1|} & =(a-b) \int_0^1 \left\{1-(\tau \, a +(1-\tau) \, b) \, |\xi_1|\right\} \, e^{-(\tau \, a +(1-\tau) \, b) \, |\xi_1|} \, \dif \tau\\
& \geq 
\begin{cases}
\frac{e^{-1}}{2} \, (a-b) & \mbox{for all~$|\xi_1| \leq \frac{a^{-1}}{2}$,}\\
0 & \mbox{for all~$|\xi_1| \leq a^{-1}$,} 
\end{cases}
\end{split}
\end{equation*}
we find that
\begin{equation*}
\begin{split}
I_1 & \geq C_0 \, \frac{e^{-1}}{2} \, (a-b) \int_{\frac{5 \, \pi}{4}<|\xi_1|<\frac{a^{-1}}{2}} |\xi_1|^{-1} \,  \sin^2 (\xi_1) \, \dif \xi_1\\
& \geq \frac{C_0}{4} \frac{e^{-1}}{2} \, (a-b) \int_{\frac{5 \, \pi}{4}<|\xi_1|<\frac{a^{-1}}{2}} |\xi_1|^{-1} \,  \dif \xi_1.
\end{split}
\end{equation*}
To get the last line, we have used that since $\sin^2 (\cdot) \geq \frac{1}{2}$ on $E:=[\frac{\pi}{4},\frac{3 \, \pi}{4}]+\pi \,
\mathbb{Z}$, with $\R \setminus E=E+\frac{\pi}{2}$,
\begin{equation}
\label{lbnd_int}
\begin{split}
\int_{\frac{5 \, \pi}{4}<|\xi_1|<\frac{a^{-1}}{2}} |\xi_1|^{-1} \, \sin^2 (\xi_1) \, \dif \xi_1
& = \int_{\frac{5 \, \pi}{4}<|\xi_1|} \underbrace{|\xi_1|^{-1} \, \mathbf{1}_{|\xi_1| < \frac{a^{-1}}{2}}}_{=:g(|\xi_1|)} \,  \sin^2 (\xi_1) \, \dif \xi_1\\
& \geq 
\frac{1}{2} \int_{E \cap \{\frac{5 \, \pi}{4}<|\xi_1|\}} g(|\xi_1|) \, \dif \xi_1\\
& \geq \frac{1}{4} \int_{\frac{5 \, \pi}{4}<|\xi_1|} g(|\xi_1|) \, \dif \xi_1,
\end{split}
\end{equation}
by translation and since $g$ is nonincreasing.
It follows that
$$
I_1 \geq \tilde{C}_0 \, (b-a) \, \ln a + \cO(a-b) \geq \tilde{C}_0 \, ( b \, \ln b-a \, \ln a)+(b \, \ln b-a \, \ln a) \co(1) 
$$
as $a,b \da 0$, where $\tilde{C}_0=\frac{C_0}{4 \, e}>0$.
Here we have
used that $b \, \ln a \geq b \, \ln b$, and since $b \, \ln b-a \,
\ln a>0$ for small $a>b>0$,  the proof of
\eqref{item-modulus} is complete under the assumption that
$T=\gamma=1$.

\medskip

For general $T,\gamma>0$ fixed, the result follows from rescaling. Let
$w(x,t):=\gamma^{-1} \, 
u(\gamma \, x,T \, t)$ and note that
$$
\begin{cases}
w_t+T \, \gamma^{-\alpha} \, a \, (-\triangle)^\frac{\alpha}{2} w=0,\\
w(x,0)=\mathbf{1}_{Q}(x).
\end{cases}
$$
Set $\mu:=T \, \gamma^{-\alpha}$ and $w=:w_{\mu \, a}$ to emphasize
the dependence on the new ``nonlinearity'' $\mu \, a$. Then by the
results of the $T=\gamma=1$ case above,
$$
\liminf_{a,b \da 0} \frac{\|w_{\mu \, a}-w_{\mu \, b}\|_{C([0,1];L^1)}}{\omega_{\mu \, a-\mu \, b}} >0,
$$
where  $\omega_{\cdot-\cdot}$ is defined on page \pageref{page-modulus}. By a simple change of variables,
$$
\|u_a-u_b\|_{C([0,T];L^1)} = \gamma^{d+1} \, \|w_{\mu \, a}-w_{\mu \, b}\|_{C([0,1];L^1)},
$$
and since $\omega_{\mu \, a-\mu \, b} \sim \omega_{a-b}$ as $a,b \da
0$ ($\mu$ is fixed!), \eqref{item-modulus} holds for any $T,\gamma>0$. 

\medskip

\noindent {\bf 2.} {\it Item \eqref{item-time}.} Let us adapt the
preceding arguments. We only give the proof for the case 
  $\gamma=1$ and $c=1$, 
noting that the general result then easily follows from the rescaling
$w(x,t)=\gamma^{-1} \, u(\gamma \, x,\gamma^\alp \, c^{-1} \, t)$. We have
\begin{align}
\Lip_\varphi(u;1) & \geq \lim_{a,b \ra 1} \left|\frac{\int_Q u_a(x,T) \, \dif x-\int_Q u_b(x,T) \, \dif x}{a-b}\right|\nonumber\\
& =\frac{2^d}{\pi^d} \int T \, |\xi|^\alpha \,  e^{-T \, |\xi|^\alp } \prod_{i=1}^d \sinc \! {}^2 (\xi_i) \, \dif \xi,\label{ref-scaling}
\end{align}
thanks to \eqref{opt-nonlin-tech3} written for time $T$. 
At this stage, the case $\alpha<1$ follows from a direct passage to the limit. For the other ones, we argue as in \eqref{referee-tech1}--\eqref{referee-tech2}, and find that there exists $C_0>0$  such that for all sufficiently small
$T$, 
\begin{equation*}
\Lip_\varphi(u;1) \geq C_0 \underbrace{\int T \, |\xi_1|^\alpha \, e^{-d^{\alpha-1} \, T \, |\xi_1|^\alpha} \, \sinc \! {}^2 (\xi_1) \, \dif \xi_1}_{=:I}.
\end{equation*}
It remains to prove that $\liminf_{T \da 0} \frac{I}{\sigma_T}>0.$ 
The
case $\alpha>1$ follows, as before, from the change of variable
$T^{\frac{1}{\alpha}} \, \xi_1  \mt \xi_1$ and the
$L^\infty$-weak-$\star$ convergence of $\sin^2 (T^{-\frac{1}{\alpha}} \, \cdot)$. For the $\alp=1$ case, we again
split $I$ into three parts,
$$
I=\int_{1<|\xi_1| <T^{-1} } \dots+\int_{|\xi_1| <1 } \dots+\int_{|\xi_1| > T^{-1} } \dots.
$$
As in case \eqref{item-modulus}, the two last terms are
$\cO(T)=T \, |\ln  T|\co(1)$ as $T\da0$, and the remaining integral
can be bounded below as in \eqref{lbnd_int} by   
\begin{equation*}
\tilde{C}_0 \int_{\frac{5 \, \pi}{4}<|\xi_1|<T^{-1}} T \, |\xi_1|^{-1} \,
\dif \xi_1
\geq \tilde{C}_0 \, T \, |\ln T| + T \, |\ln
T|\co(1)\quad\text{as }T\da0,
\end{equation*}
where $\tilde{C}_0>0$ is another constant independent of
$T$ small enough. The proof is complete. 

\medskip

\noindent {\bf 3.} {\it Item \eqref{test-item}.} We assume that
$T=c=1$, and note that the general case follows from the rescaling
$w(x,t)=u(T^{\frac1\alp} \, c^{\frac1\alp} \, x,T \, t)$. We
start as in the preceding case, considering 
this time integrals on $\gamma \, Q$
in \eqref{def-error}. Arguing as in \eqref{opt-nonlin-tech3} by replacing $Q$ by $\gamma \, Q$, we find that
\begin{equation*}
\begin{split}
\mathcal{E}_{\gamma \, Q} & =\int_{\gamma \, Q} u_a(x,1) \, \dif x-\int_{\gamma \, Q} u_b(x,1) \, \dif x\\
& = \frac{2^d}{\pi^d} \, \gamma^{2 \, d+1} \int \int_0^1 (b-a) |\xi|^\alpha \,  e^{-(\tau \, a+(1-\tau) \, b) \, |\xi|^\alp } \prod_{i=1}^d \sinc \! {}^2 (\gamma \, \xi_i) \, \dif \tau \, \dif \xi,
\end{split}
\end{equation*}
and hence 
\begin{equation*}
\Lip_\varphi(u;1) \geq \lim_{a,b \ra 1} \left|\frac{\mathcal{E}_{\gamma \, Q}}{a-b}\right|=\frac{2^d}{\pi^d} \, \gamma^{2 \, d+1} \int |\xi|^\alpha \,  e^{-|\xi|^\alp } \prod_{i=1}^d \sinc \! {}^2 (\gamma \, \xi_i) \, \dif \xi. 
\end{equation*}
After changing variables $\gamma \, \xi \mt \xi$, we then get that
$$
\Lip_\varphi(u;1)\geq \frac{2^d}{\pi^d} \, \gamma^{d+1} \int \gamma^{-\alpha} \, |\xi|^\alpha \,  e^{-\gamma^{-\alpha} \, |\xi|^\alp } \prod_{i=1}^d \sinc \! {}^2 (\xi_i) \, \dif \xi.
$$
This is the same expression as in \eqref{ref-scaling} with
$\gamma^{-\alpha}$ in place of $T$. Note that
$$\gamma^{d+1} \, \sigma_T {}_{\big|_{T=\gamma^{-\alp}}} = \sigma_\gamma$$ according to
the definitions of $\sigma_T$ and $\sigma_\gamma$ on page
\pageref{page-modulus}, and hence by the proof of \eqref{item-time} we
have that $\liminf_{\gamma \ra +\infty}
\frac{\Lip_\varphi(u;1)}{\sigma_\gamma} >0$.
The proof of \eqref{test-item} is complete.
\end{proof}

\begin{remark}
\label{rem_optt}
In the proof of Corollary \ref{th:time}, a rescaling
in time transformed the continuous dependence estimate
\eqref{main-esti-first} into the time continuity estimate
\eqref{esti-time}. Hence, we leave it to the reader to verify that the
same rescaling allows us to prove that \eqref{ex} is also an example
for which Corollary \ref{th:time} is optimal.
\end{remark}

\begin{proof}[Proof of Proposition \ref{prop:opt-powers}]
We adapt the arguments of the proof of Proposition \ref{prop:opt-nonlin}\eqref{item-modulus}. 

\medskip

\noindent {\bf 1.} {\it Item \eqref{item-time-Levy}.} To avoid confusion with the proof of \eqref{test-item-Levy} below, we denote the fixed parameter $\gamma$ by $\tilde{\gamma}$. We consider the new difference 
$$
\E_Q:=\int_{\tilde{\gamma} \, Q} u^\alpha(x,T) \, \dif x-\int_{\tilde{\gamma} \, Q}u^\beta(x,T) \, \dif x
$$
with moving powers $\alpha,\beta \in (0,2)$ and time $T$. We let
$M=T\, a$ and argue as in \eqref{opt-nonlin-tech3} to see that
\begin{equation*}
\begin{split}
& \E_Q \\
& =\frac{2^d}{\pi^d} \, \tilde{\gamma}^{2 \, d+1} \int
(e^{-M \, |\xi|^\alp }-e^{-M \, |\xi|^\beta }) \prod_{i=1}^d \sinc \! {}^2 (\tilde{\gamma} \, \xi_i) \, \dif \xi\\
& = \frac{2^d}{\pi^d} \, \tilde{\gamma}^{2 \, d+1} \int \int_0^1
(\beta-\alpha) \, (\ln |\xi|) \, M \, |\xi|^{\tau \, \alpha+(1-\tau) \, \beta}\\
& \quad \cdot e^{-M \, |\xi|^{\tau \, \alpha+(1-\tau) \, \beta} } \prod_{i=1}^d \sinc \! {}^2 (\tilde{\gamma} \, \xi_i) \, \dif \tau \, \dif \xi,
\end{split}
\end{equation*}
so that 
\begin{equation}\label{opt-tech-100}
\Lip_\alpha (u;\lambda) \geq \frac{2^d}{\pi^d} \, \tilde{\gamma}^{2 \, d+1} \, \Big|\underbrace{\int (\ln |\xi|) \, M \, |\xi|^\lambda \, e^{-M \, |\xi|^{\lambda }} \prod_{i=1}^d \sinc \! {}^2 (\tilde{\gamma} \, \xi_i) \, \dif \xi}_{=:I}\Big|.
\end{equation} 
To complete the proof, we must show that $\liminf_{M \da 0}
\frac{|I|}{\tilde{\sigma}_M} >0.$  

\medskip

{\bf a.} \emph{The case $\lambda<1$.} Now
$$
\lim_{M \da 0} \frac{I}{M}= \int (\ln |\xi|) \, |\xi|^\lambda \prod_{i=1}^d \sinc \! {}^2 (\tilde{\gamma} \, \xi_i) \, \dif \xi=:I_{\tilde{\gamma}}.
$$ 
Since $\sinc(0)\neq 0$,
$\lim_{\tilde{\gamma} \da 0} I_{\tilde{\gamma}}=+\infty$, and we see that
\eqref{item-time-Levy} holds for $\tilde{\gamma}$ small enough. 

\medskip

In the other two cases we split $I$ in two, $
I=\int_{|\xi| <1} \dots + \int_{|\xi|>1} \dots.$
The first integral is of order $\cO(M) = \tilde{\sigma}_M \co(1)$ as $M \da 0$, by a direct passage to the limit. Arguing as in the preceding proof (cf. \eqref{referee-tech1}--\eqref{referee-tech2}),  the last integral can be bounded from below by 
\begin{equation*}
\begin{split}
 \int_{|\xi_1|>1} 
(\ln  |\xi_1|) \, M \, |\xi_1|^{\lambda-2} \, e^{-d^{\lambda-1} \, M \, |\xi_1|^{\lambda}} \, \sin^2 (\tilde{\gamma} \, \xi_1)  \, \dif \xi_1=:J,
\end{split}
\end{equation*}
up to some positive multiplicative constant $C_0$ independent of 
$M$ small enough. 
Note that $C_0$ will also depend on
  $\tilde{\gamma}>0$ which is constant in this proof.
Hence it suffices to show that $\liminf_{M \da 0} \frac{J}{\tilde{\sigma}_M}>0$.
\medskip

{\bf b.} \emph{The case $\lambda>1$.} By the change of variables
$M^{\frac{1}{\lambda}} \, \xi_1 \mt \xi_1$,  
\begin{equation*}
\begin{split}
J& =M^{\frac{1}{\lambda}} \int_{|\xi_1|>M^{\frac{1}{\lambda}}} (\ln |\xi_1|) \, 
|\xi_1|^{\lambda-2} \, e^{-d^{\lambda-1} \, |\xi_1|^{\lambda}} \, \sin^2 (M^{-\frac{1}{\lambda}} \,  \tilde{\gamma} \, \xi_1)  \, \dif \xi_1\\
 & \quad - \lambda^{-1} \, M^{\frac{1}{\lambda}} \, (\ln M) \int_{|\xi_1|>M^{\frac{1}{\lambda}}} 
|\xi_1|^{\lambda-2} \, e^{-d^{\lambda-1} \, |\xi_1|^{\lambda}} \, \sin^2 (M^{-\frac{1}{\lambda}} \,  \tilde{\gamma} \, \xi_1)  \, \dif \xi_1.
\end{split}
\end{equation*}
It is clear that the first term is
$\cO(M^{\frac{1}{\lambda}})=M^{\frac{1}{\lambda}} \, |\ln M| \co(1)$
as $M \da 0$, and that the second one has the expected behavior
due to  $L^\infty$-weak-$\star$ convergence arguments.  

\medskip

{\bf c.} \emph{The case $\lambda=1$.} We write $ J =
\int_{|\xi_1|>M^{-1}}\dots+\int_{1 <|\xi_1|< M^{-1}}\dots.  $ The
first term is $\cO(M \, |\ln M|)=M \, (\ln^2 M) \co(1)$
as $M \da 0$ by the change of variables argument of the $\lambda>1$
case. For the remaining term, we argue as in
\eqref{lbnd_int}, using this time that $\sin^{2} (\tilde{\gamma} \, \cdot)$ is bounded below by $\frac{1}{2}$ on
$\tilde{\gamma}^{-1} \, E$. Taking $N$ so 
large that the new function $g$ (defined below) is
nonincreasing on $((4 \, N+1) \, \frac{\pi}{4} \,
\tilde{\gamma}^{-1},+\infty)$, we get a lower bound of the form
\begin{equation*}
\begin{split}
& \int_{1<|\xi_1|<M^{-1}} 
(\ln  |\xi_1|) \, M \, |\xi_1|^{-1} \, e^{- M \, |\xi_1|} \sin^2 (\tilde{\gamma} \, \xi_1)  \, \dif \xi_1 
\\
& \geq  e^{-1} \, M \int_{(4 \, N+1) \, \frac{\pi}{4} \,
\tilde{\gamma}^{-1}<|\xi_1|} \underbrace{(\ln  |\xi_1|) \, |\xi_1|^{-1} \, \mathbf{1}_{|\xi|<M^{-1}}}_{=:g(|\xi_1|)} \, \sin^2 (\tilde{\gamma} \, \xi_1) \, \dif \xi_1\\
& \geq \frac{e^{-1}}{4} \, M \int_{(4 \, N+1) \, \frac{\pi}{4} \, \tilde{\gamma}^{-1} <|\xi_1|<M^{-1}}  (\ln |\xi_1|) \, |\xi_1|^{-1} \,  \, \dif \xi_1\\
& = \frac{e^{-1}}{4} \, M \, \ln^2 M+M \, (\ln^2 M) \co(1)\quad\text{as } M \da 0.
\end{split}
\end{equation*}
The proof of \eqref{item-time-Levy} is now complete.

\medskip

\noindent {\bf  2.} {\it Item \eqref{test-item-Levy}.} To avoid confusion with the preceding proof,  we denote the fixed parameter $M=T \, a$ by $\tilde{M}$. Then, by \eqref{opt-tech-100},
\begin{equation}
\label{opt-restart}
\Lip_\alpha (u;\lambda) \geq \frac{2^d}{\pi^d} \, \tilde{M} \, \Big|\underbrace{ \gamma^{2 \, d+1} \int (\ln |\xi|) \, |\xi|^\lambda \, e^{-\tilde{M} \, |\xi|^{\lambda }} \prod_{i=1}^d \sinc \! {}^2 (\gamma \, \xi_i) \, \dif \xi}_{=:I}\Big|
\end{equation}
and it suffices to show that $\liminf_{\gamma \ra +\infty} \frac{|I|}{\tilde{\sigma}_\gamma} >0$.

\medskip 

{\bf a.} {\it The case $\lambda>1$.} Since $\ln |\xi|$ has
  different signs inside and outside the unit ball, we split the
  integral $I$ in two,
$$
I=\int_{|\xi|<1} \dots +\int_{|\xi|>1} \dots=:I_1+I_2.
$$
By the inequality $|\ln |\xi|| \, |\xi|^{\lambda} \leq d^{\lambda-1}
\sum_{i=1}^d |\ln |\xi_i|| \, |\xi_i|^{\lambda}$ for $|\xi|<1$ and the
change of variables $\gamma \, \xi_j \mt \xi_j$ for $j \neq i$, we
find that
\begin{equation*}
\begin{split}
|I_1| & \leq  d^{\lambda-1} \, \gamma^{2 \, d+1} \sum_{i=1}^d \int_{|\xi|<1} |\ln |\xi_i|| \, |\xi_i|^\lambda \prod_{j=1}^d \sinc \! {}^2 (\gamma \, \xi_j) \, \dif \xi\\
& \leq  d^{\lambda-1} \, \gamma^{d} \sum_{i=1}^d   \left\{ \int_{|\xi_i|<1} \frac{|\ln |\xi_i|| \, |\xi_i|^\lambda}{\xi_i^2} \,  \dif \xi _i \prod_{j\neq i} \int \sinc \! {}^2 (\xi_j) \, \dif \xi_j \right\}.
\end{split}
\end{equation*}
Here we also have used that $ \sin^2(\gamma \, \xi_i) \leq 1$. It follows that $\limsup_{\gamma \ra
  +\infty} \frac{|I_1|}{\gamma^d} \leq C(d,\lambda)$, a constant that does not depend on $\tilde{M}$.

Let us now see that $I_2$ is the dominant term provided that the fixed
parameter $\tilde{M}$ is chosen sufficiently small. We have 
$$
I_2 \geq \gamma^{2 \, d+1} \int (\ln |\xi_1|) \, |\xi_1|^\lambda \, e^{-\tilde{M} \, |\xi|^{\lambda }} \prod_{i=1}^d \sinc \! {}^2 (\gamma \, \xi_i) \, \dif \xi,
$$
and then, letting $\xi_\gamma:=(\xi_1,\gamma^{-1} \,
\xi_2,\dots,\gamma^{-1} \, \xi_d)$ and changing variables $\gamma \,
\xi_i \mt \xi_i$ for $i\neq 1$, we find that
\begin{equation*}
I_2 \geq \gamma^{d} \int \left\{ (\ln |\xi_1|) \, |\xi_1|^{\lambda-2} \int  e^{-\tilde{M} \, |\xi_\gamma|^{\lambda }} \prod_{i=2}^d \sinc \! {}^2 (\xi_i) \, \dif \xi_2 \dots \dif \xi_d \right\}\sin^2(\gamma \, \xi_1) \, \dif \xi_1.
\end{equation*} 
By $L^\infty$-weak-$\star$ convergence arguments, $
\liminf_{\gamma \ra +\infty} \frac{I_2}{\gamma^d} \geq \tilde{m} \, I_{\tilde{M}},
$
where 
$$
\tilde{m}:=\int_0^{2 \, \pi} \sin^2(\xi_1) \, \dif \xi_1 \prod_{i=2}^d \int \sinc \! {}^2 (\xi_i) \, \dif \xi_i>0
$$ 
and $
I_{\tilde{M}} :=\int (\ln |\xi_1|) \, |\xi_1|^{\lambda-2} \, e^{-\tilde{M} \, |\xi_1|^\lambda} \, \dif \xi_1.
$
Since $\lim_{\tilde{M} \da 0} I_{\tilde{M}}=+\infty$, it suffices to
fix $\tilde{M}>0$ small to get  \eqref{test-item-Levy} in the
$\lambda>1$ case. 

\medskip 

{\bf b.} {\it The case $\lambda<1$.} We restart from
\eqref{opt-restart}, change the variables $\gamma \, \xi \mt \xi$, and
pass to the limit as $\gamma\ra +\infty$. The result follows.

\medskip 

{\bf c.} {\it The case $\lambda=1$.} Let us rewrite $I$ in
\eqref{opt-restart} as 
$$
I=\underbrace{ \gamma^{2 \, d+1} \int_{\gamma^{-1} < |\xi| <1} \dots}_{=:-J_1} + \underbrace{\gamma^{2 \, d+1} \int_{|\xi|>1} \dots}_{=:J_2}+\gamma^{2 \, d+1} \int_{|\xi|<\gamma^{-1}}\dots .
$$
By the arguments of the $\lambda<1$ case, the last integral is of order $\cO(\gamma^d \, \ln \gamma)=\gamma^d \, (\ln^2 \gamma) \co(1)$ as $\gamma \ra +\infty$. For $J_2$, we use that
\begin{equation*}
\begin{split}
\left|\int_{|\xi|>1} \dots \right| \leq \int_{\cup_{i=1}^d \left\{\xi:|\xi_i|>\frac{1}{\sqrt{d}}\right\}} |\dots|  & \leq \sum_{i=1}^d \int_{\xi:|\xi_i|>\frac{1}{\sqrt{d}}} |\dots|=d  \int_{\xi:|\xi_1|>\frac{1}{\sqrt{d}}} |\dots|
\end{split}
\end{equation*}
by symmetry of the $I$-integrand (cf. \eqref{opt-restart}). We then bound $|\ln |\xi|| \, |\xi| \, e^{-
 \tilde{M}
|\xi|}$ by some constant $C$, change the variables $\gamma \, \xi_i
\mt \xi_i$ for $i \neq 1$, get 
\begin{equation*}
\begin{split}
|J_2| & \leq d \, C \, \gamma^{d} \int_{|\xi_1|>\frac{1}{\sqrt{d}}} \xi_1^{-2} \, \dif \xi_1 \prod_{i=2}^d \int \sinc \! {}^2 (\xi_i) \, \dif \xi_i
\end{split}
\end{equation*}
and conclude that $J_2=\cO(\gamma^d)=\gamma^d \, (\ln^2 \gamma) \co(1)$ as $\gamma \ra +\infty$. 

Since $J_1>0$, it remains to show that $\liminf_{\gamma \ra +\infty}
\frac{J_1}{\gamma^d \,\ln^2 \gamma}>0$. It will be convenient to use
the notation $\hat{\xi}:=(\xi_2,\dots,\xi_d)$. By the change of
variables $\gamma \, \xi \mt \xi$ and the inequality $|\xi| \,
\xi_1^{-2} \geq |\xi|^{-1}$,
\begin{equation*}
\begin{split}
J_1 & \geq - \gamma^{d} \int_{1<|\xi|< \gamma}  
\frac{\ln  (\gamma^{-1} \, |\xi|)}{ |\xi|} \, f(\hat{\xi}) \, \sin^2(\xi_1) \, \dif \xi,
\end{split}
\end{equation*}
where $f(\hat{\xi}):= e^{-\tilde{M}} \prod_{i=2}^d \sinc \! {}^2 (\xi_i)$. 
Let us restrict to the domain of integration
$$
A:=\left\{(\xi_1,\hat{\xi}) \mbox{ s.t. } \frac{5 \, \pi}{4}<|\xi_1| <\frac{1}{2} \,  \sqrt{\gamma^2-|\hat{\xi}|^2} \mbox{ and } \epsilon^2 <|\hat{\xi}| < \epsilon \right\},
$$ 
where $\epsilon>0$ is fixed and so small that $A \subset \{1<|\xi|<\gamma\}$ and $C_0:=\min_A f(\hat{\xi})>0$. Then,
\begin{equation*}
\begin{split}
J _1& \geq -C_0 \, \gamma^{d} \int_{\epsilon^2 < |\hat{\xi}|<\epsilon} \int_{\frac{5 \, \pi}{4}<|\xi_1|<\frac{1}{2} \sqrt{\gamma^2-|\hat{\xi}|^2}}  
\frac{\ln  (\gamma^{-1} \, |\xi|)}{ |\xi|} \, \sin^2(\xi_1) \, \dif \xi_1 \, \dif \hat{\xi}.
\end{split}
\end{equation*}
Arguing as in \eqref{lbnd_int},
\begin{equation*}
J_1 \geq -\tilde{C}_0 \, \gamma^{d} \int_{\epsilon^2 < |\hat{\xi}|<\epsilon} \int_{\frac{5 \, \pi}{4}<|\xi_1|<\frac{1}{2} \, \sqrt{\gamma^2-|\hat{\xi}|^2}}  
 \frac{\ln  (\gamma^{-1} \, |\xi|)}{ |\xi|} \, \dif \xi_1 \, \dif \hat{\xi},
\end{equation*}
where $\tilde C_0=\frac{C_0}{4}>0$. 
If $\gamma$ is large enough, then
for all $\xi \in A$, $\gamma^{-1} \, |\xi|\leq \gamma^{-1} \,
(|\xi_1|+|\hat{\xi}|)<1$. We can then use that
$$
-  \frac{\ln  (\gamma^{-1} \, |\xi|)}{ |\xi|} \geq
-  \frac{\ln  \left(\gamma^{-1} \, (|\xi_1|+|\hat{\xi}|)\right)}{|\xi_1|+|\hat{\xi}|}, 
$$
and by integrating the right-hand side, we get
\begin{equation*}
\begin{split}
J_1 & \geq  \tilde{C}_0 \, \gamma^{d} \int_{\epsilon^2 < |\hat{\xi}|<\epsilon} \ln^2 \left( \gamma^{-1} \left(\frac{5 \, \pi}{4} + \, |\hat{\xi}|\right) \right) \dif \hat{\xi}\\
& \quad -\tilde{C}_0 \, \gamma^{d} \int_{\epsilon^2 < |\hat{\xi}|<\epsilon} \ln^2 \left(\frac{1}{2} \, \sqrt{1-\gamma^{-2} \, |\hat{\xi}|^2}+\gamma^{-1} \, |\hat{\xi}|\right) \dif \hat{\xi}\\
& \geq \tilde{\tilde{C}}_0 \, \gamma^{d} \, \ln^2 \gamma+\gamma^d \,
(\ln^2 \gamma)\co(1) \quad\text{as } \gamma\ra+\infty.
\end{split}
\end{equation*}
Here $\tilde{\tilde{C}}_0$ is another positive constant independent of
$\gamma$ large enough. The proof is complete.
\end{proof}

\appendix

\section{Proofs of \eqref{YL}, \eqref{YBV} and \eqref{YC}}\label{app-app}

\begin{proof}[Proof of \eqref{YL} and \eqref{YBV}]
Recall that~$\Omega_\xi(\cdot)$ is defined on page \pageref{po} and
$\chi_a^b$ in \eqref{si}.
For $A$, we use that
\begin{align*}
& \int \|\Omega_\xi(u)\|_{L^1(Q_T)} \, \dif \xi  = \int \int_{Q_T} \left|\int \chi_0^{u(x,t)}(\zeta) \,  \omega_\delta(\xi-\zeta) \,  \dif \zeta\right| \dif x \, \dif t \, \dif \xi\\
& = \int \int_{Q_T } \left|\chi_0^{u(x,t)}(\zeta) \right| \int \omega_\delta(\xi-\zeta) \, \dif \xi \, \dif x \, \dif t \, \dif \zeta= \|u\|_{L^1(Q_T)}.
\end{align*}
For $B$, we consider $\{u_n\}_n \subset C([0,T];W^{1,1})$ converging to $u$ in $C([0,T];L^1)$ and such that $\int_{Q_T} |\nabla u_ n| \ra |u|_{L^1(0,T;BV)}$. Then
\begin{equation*}
\begin{split}
& \int \int_{Q_T} \left| \nabla \Omega_{\xi}(u_n) \right| \dif x \, \dif t \, \dif \xi \\
& = \int \int_{Q_T} \omega_\delta (\xi-u_n(x,t)) \, |\nabla u_n(x,t)| \, \dif x \, \dif t \, \dif \xi= \int_{Q_T} |\nabla u_n|,
\end{split}	
\end{equation*}
so that
\begin{equation*}
\begin{split}
& \int |\Omega_\xi(u)|_{L^1(0,T;BV)} \, \dif \xi \leq \int \int_0^T \left\{ \liminf_{n \ra +\infty} \int_{\R^d} \left| \nabla \Omega_{\xi} (u_n) \right| \dif x \right\} \dif t \, \dif \xi \\
& \leq \lim_{n \ra +\infty} \int \int_{Q_T} \left| \nabla \Omega_{\xi} ( u_n) \right| \dif x \, \dif t \, \dif \xi  =|u|_{L^1(0,T;BV)},
\end{split}	
\end{equation*}
due to the lower semi-continuity of the $BV$-semi-norm with respect to the $L^1$-norm and to Fatou's lemma.\footnote{For the reverse inequality, we use that, at all fixed time and for all $\Phi \in C^1_c(\R^d,\R^d)$ such that $|\Phi| \leq 1$,
\begin{equation*}
\begin{split}
& \int |\Omega_\xi(u)|_{BV} \, \dif \xi \geq \int \int_{\R^d} \Omega_\xi(u) \, \Div \Phi \,  \dif x  \, \dif \xi   \\
& = \int \int_{\R^d} \int \chi_0^u(\zeta) \, \omega_{\delta}(\xi-\zeta) \, \Div \Phi \,  \dif \zeta \, \dif x \, \dif \xi 
= \int_{\R^d} u \, \Div \Phi \, \dif x, 
\end{split}	
\end{equation*}
and next we take the supremum with respect to $\Phi$.}
\end{proof}

\begin{proof}[Proof of~\eqref{YC}]
Recall that~$\E_3$ and~$\E_3(\delta)$ are defined in~\eqref{nat:cnl-1} and~\eqref{def-a-lp}, respectively. See also the original assumption~\eqref{set-assumptions} of the theorem, and the simplifying assumption~\eqref{hnd}. First we define
$$
F_v(x,t,y,s,z,\xi):=\sgn(v-u) \, \frac{\chi_{v(x,t)}^{v(x+z,t)}(\xi)}{|z|^{d+\alpha}} \, \phi^{\epsilon,\nu}.
$$
Let us recall that~$F_v$ is integrable on~$Q_T^2 \times \{|z|>r\}
\times \R$ since $\int |\chi_{a}^{b}(\xi) | \, \dif \xi = |b-a|$. Hence, by Fubini the function
$$
G_v(\xi):= \int_{Q_T^2} \int_{|z|>r} F_v(x,t,y,s,z,\xi) \, \dif z \, \dif w
$$
is integrable with respect to~$\xi \in \R$. But, by~\eqref{e3-lp},~\eqref{def-a-lp}, 
$$
\E_3= G_d(\alpha) \int G_v(\xi) \, \psi'(\xi)-G_u(\xi) \, \varphi'(\xi) \, \dif \xi
$$
and
$
\E_{3}(\delta) = G_d(\alpha) \int \psi'(\xi) \, G_v \ast \omega_\delta (\xi) - \varphi'(\xi) \, G_u \ast \omega_\delta (\xi) \,  \dif \xi,
$
where~$\ast$ is the convolution product in~$\R$. Since~$\omega_\delta$ is an approximate unit, the convolution products inside the integral respectively converge to~$G_v$ and~$G_u$ in~$L^1(\R)$ as~$\delta \da 0$.
Using in addition that~$\varphi'$ and~$\psi'$ are bounded by \eqref{hnd}, $\lim_{\delta \da 0} \E_3(\delta)=\E_3$ . 
\end{proof}

\section{Some technical lemmas}

\begin{lemma}\label{lem-tech-nl}
For all~$a,b >0$,~$|a-b| \, (-\ln (a \vee b))^+ \leq |a-b|+|a \, \ln a-b \, \ln b|$.
\end{lemma}

\begin{proof}
We assume without loss of generality that $a \vee b=a$ and that $a\leq e^{-1}$ (the result is trivial otherwise). Then
$$
|a \, \ln a - b \, \ln b|= -\int_b^a (1+ \ln \tau) \, \dif \tau =-|a-b| -\int_b^a (\ln \tau) \, \dif \tau,
$$
since $1+\ln \tau$ is negative, and hence
$$
|a-b|+|a \, \ln a - b \, \ln b| \geq -|a-b| \, \ln a,
$$ 
since the logarithm is nondecreasing. 
This completes the proof.
\end{proof}

\begin{lemma}\label{lem-taylor}
For all~$x>0$,~$a,b\neq 0$ and $c>0$, 
\begin{enumerate}[\rm (i)]
\item \label{t3}
$\left|\frac{x^a-1}{a}-\frac{x^b-1}{b}\right| \leq |a-b| \, (1
\vee x^{a} \vee x^{b}) \, \ln^2 x$, 
\item \label{t5} $\lim_{a,b \ra c}
\left\{\frac{1}{(a-b)^2} \left|\frac{x^{2 \, a}}{2 \, a}+\frac{x^{2 \, b}}{2 \, b}-\frac{2 \, x^{a+b}}{a+b}\right|\right\} \leq C \, x^{2 \, c} \, (1 + \ln^2 x),$ 
\end{enumerate}
where $C=C(c)$. 
\end{lemma}

\begin{proof}
\eqref{t3} Let $f(a)=\frac{x^a-1}{a}$. Observe that $
(\ln x) \int_0^1 x^{\tau \, a} \, \dif \tau=f(a)$
by a Taylor expansion of $x^a$ at $a=0$. It then follows that by differentiating under the integral sign that
$
f'(a)=
(\ln^2 x) \int_0^1 \tau \, x^{\tau \, a} \, \dif \tau
$
and
\begin{align*}
& f(a)-f(b) =(a-b) \int_0^1 f'(\tau \, a+ (1-\tau) \, b) \, \dif \tau
\\& 
= (a-b) \, (\ln^2 x) \int_0^1 \int_0^1 \tilde{\tau} \, x^{\tilde{\tau} \, (\tau \, a+(1-\tau) \, b)} \, \dif \tilde{\tau}\, \dif \tau .
\end{align*}
Since $x^{\tilde{\tau} \, (\tau \, a+(1-\tau) \, b)} \leq 1 \vee x^{a
  \vee b} \vee x^{a \wedge b}$,
we find that
\begin{align*}
& |f(a)-f(b)| \leq |a-b| \, (\ln^2 x) \, (1 \vee x^{a } \vee
x^{ b}), 
\end{align*}
and the proof of \eqref{t3} is complete.
\medskip

\noindent \eqref{t5} Note that
\begin{equation*}
\begin{split}
\frac{x^{2 \, a}}{2 \, a}+\frac{x^{2 \, b}}{2 \, b}-\frac{2 \, x^{a+b}}{a+b} & = \frac{b \, (a+b) \, x^{2 \, a}
+a \, (a+b) \, x^{2 \, b}-4 \, a\, b\, x^{a+b} }{2 \, a \,  b \, (a+b)}\\
& = \frac{(x^{a}-x^b)^2}{2 \, (a+b)}+\frac{(b \, x^{a}-
a \, x^{b})^2}{2 \, a \,  b \, (a+b)}.
\end{split}
\end{equation*}
The first term satisfies~$|x^a-x^b| \leq |a-b| \, (x^a \vee x^b) \, |\ln x|$. Moreover, by adding and  
subtracting terms and the inequality $\frac{(A+B)^2}{2} \leq A^2+B^2$,
we find that the second term is bounded by 
$$
(b \, x^{a}-
a \, x^{b})^2 \leq \frac{1}{2} \, (a-b)^2 \,
(x^a+x^b)^2+\frac{1}{2} \, (a+b)^2 \, (x^a-x^b)^2.
$$ 
The proof now follows from these two inequalities. 
\end{proof}

\section{Proofs of Lemma \ref{lem01} and \eqref{referee-def-norm}}
\label{eta-entropy}

\begin{proof}[Proof of Lemma \ref{lem01}] The if part follows by approximating the
Kruzhkov 
entropy $u \mt |u-k|$ 
by 
smooth convex entropies
$u \mt \eta_n(u):=\sqrt{(u-k)^2+n^{-2}}-n^{-1}$. The functions $\eta_n(\cdot)$
and
$\eta_n'(\cdot)$ are locally uniformly bounded and converge pointwise
to $|\cdot-k|$ 
and the everywhere representative of its weak
derivative given by \eqref{nat:representation-sign}. 
Hence, if a function $u=u(x,t)$ is bounded and such that
$\eta_n(u)$ satisfies \eqref{entropy-ineq-eta}, we can use the dominated
convergence theorem to pass to the limit and find that $|u-k|$
satisfies \eqref{entropy_ineq}. 

To prove the only if part, we note that we may approximate (locally
uniformly) any convex entropy $\eta \in C^1(\R)$ 
by a family of 
piecewise linear functions $\tilde\eta_n$ of the form
\begin{align}
\label{teta}
u \mt \tilde\eta(u)=a+b \, u+\sum_{i=1}^m c_i \, |u-k_i|
\end{align}
where $a,b,k_i\in\R$, $c_i\geq0$, and $m\in\N$. See e.g.
\cite[p.~27]{HoRi07} for a proof. We need to refine this construction to
ensure everywhere convergence of the derivatives
$\tilde\eta'_n$. Consider the everywhere defined representative of
$\tilde\eta'$ defined by
\begin{align}
\label{teta'}
u \mt \tilde\eta'(u)=b+\sum_{i=1}^m c_i \, \sgn(u-k_i),
\end{align}
where the sign function is everywhere defined by
\eqref{nat:representation-sign}. Since $\eta'$ is continuous, it can
be approximated uniformly on compact sets 
by piecewise constant functions of
the form \eqref{teta'}. Take such a sequence $\{\tilde\eta'_n\}_n$ that converges locally uniformly on $\R$
and 
redefine $\{\tilde\eta_n\}_n$ to be the primitives such that
$\tilde\eta_n(0)=\eta(0)$, i.e. functions of the form \eqref{teta}. It
follows that both $\tilde\eta_n$ and $\tilde\eta'_n$ converge
locally uniformly towards $\eta$ and $\eta'$. 

Consider next the entropy solution $u=u(x,t)$ of \eqref{1} and note that the left-hand side of the entropy inequality \eqref{entropy-ineq-eta} is linear
with respect to $\eta$, that \eqref{entropy-ineq-eta} holds with  $\eta(u)=a+b \, u$ since $u$
is a weak (distributional) 
solution of \eqref{1}, and that \eqref{entropy-ineq-eta} holds with
$\eta(u)=c_i \, |u-k_i|$ by the Kruzhkov inequality \eqref{entropy_ineq}
since $c_i\geq0$. The reader may then check that 
\eqref{entropy-ineq-eta} also holds 
with $\eta=\tilde\eta$ and the everywhere
representative of $\tilde\eta'$ given by \eqref{teta'}. 

Since $u$ is bounded, we 
may use
the dominated convergence theorem to 
pass to the limit in
\eqref{entropy-ineq-eta} with $\eta=\tilde\eta_n$ to find that
\eqref{entropy-ineq-eta} holds also for the $\eta$ in the limit. The proof is complete. 
\end{proof}

\begin{proof}[Proof of \eqref{referee-def-norm}]
Combining \eqref{referee-fourier-def} and \eqref{nat:Levy-form},
\begin{equation*}
\begin{split}
\int_{\R^d} |2 \, \pi \, \xi|^\alpha \, |\mathcal{F} u|^2 \, \dif \xi &= \int_{\R^d} u \, \mathcal{F}^{-1} \left(|2 \, \pi \, \cdot|^\alpha \, \mathcal{F} u\right) \dif x\\
& = \int_{\R^d} u \, (-\triangle)^{\frac{\alpha}{2}} u \, \dif x\\
& = -\lim_{r \da 0} \int_{\R^d} u \, \Levy^{\alpha,r}[u] \, \dif x\\  
& = \frac{G_d(\alpha)}{2}  \iint_{\R^{2 \, d}}  
\frac{(u(x)-u(y))^2}{|x-y|^{d+\alpha}} \, \dif x \, \dif y, 
\end{split}
\end{equation*}
thanks to \eqref{ineqH1-1} with $v=u$ to get the last line.
\end{proof}

\end{document}